\documentclass[12pt]{amsart}
\usepackage{amssymb, amscd, tikz}

\usepackage{mathrsfs}
\usepackage{amsmath,amsfonts,amsthm}
\usepackage{latexsym}
\usepackage{hyperref}
\usepackage{graphicx}
\usepackage[all]{xy}
\usepackage{layout}

\usepackage{bbm}
\usepackage{tikz}
\usetikzlibrary{cd,decorations.pathmorphing}

\newcommand{\field}[1]{\mathbb{#1}}
\newcommand{\FF}{\field{F}}
\newcommand{\NN}{\field{N}}
\newcommand{\RR}{\field{R}}

\newcommand{\ZZ}{\field{Z}}
\newcommand{\dd}{{\rm deg}}

\newcommand{\Ff}{\mathcal F}
\newcommand{\Gg}{\mathcal G}

\newcommand{\Ka}{\mathscr{K}_A}
\newcommand{\La}{\mathscr L_A}
\newcommand{\Laa}{\mathscr L_{A_{r(p)}}}
\newcommand{\Kaa}{\mathscr K_{A_{r(p)}}}

\newcommand{\Oo}{\mathcal O}
\newcommand{\Tt}{\mathcal T}
\newcommand{\Gz}{G^{(0)}}
\newcommand{\Gt}{G^{(2)}}
\newcommand{\I}{{\rm Ind}\mu}

\newcommand{\NO}[1]{\operatorname{\mathcal{NO}}_{\!#1}}

\newcommand{\fvcl}[1]{\mathcal{P}_{\rm fin}^{\vee}(#1)}
\newcommand{\reduced}{r}
\newcommand{\normal}{n}

\newcommand{\CNP}{\operatorname{CNP}}
\newcommand{\id}{\operatorname{id}}

\newcommand{\clsp}{\operatorname{\overline{span\!}\,\,}}
\newcommand{\op}{\operatorname{op}}
\newcommand{\CNPgaug}{\nu}
\newcommand{\redgaug}{\nu^{\normal}}

\newcommand{\Xle}[1]{\widetilde{X}_{#1}}
\newcommand{\iotale}{\tilde{\iota}\hskip1pt}
\newcommand{\phile}{\tilde{\phi}}
\newcommand{\intfrm}[1]{\Pi{#1}}

\newcommand{\Tc}{\Tt_{\rm cov}}

\newcommand{\supp}{\operatorname{supp}}
\newcommand{\Augl}{\widetilde{l}}

\providecommand{\abs}[1]{\lvert#1\rvert}
\providecommand{\norm}[1]{\lVert#1\rVert}
 \newcommand{\XtY}[1]{X \overset{#1}{\otimes} Y}
\newcommand{\YtX}[1]{Y \overset{#1}{\otimes} X}
\newcommand{\myTau}[5]{\iota_{(#3,#4)}^{(#1,#2)}(#5)}
\newcommand{\ordTau}[2]{\iota_{#1}^{#2}}

\newcommand{\ordRep}{(\pi_A , \pi_B , t_X , t_Y)}
\newcommand{\uniRepT}{(\overline{\pi}_A , \overline{\pi}_B , \overline{t}_X , \overline{t}_Y)}
\newcommand{\uniRepO}{(\hat{\pi}_A , \hat{\pi}_B , \hat{t}_X , \hat{t}_Y)}
\newcommand{\fockRep}{(\phi_{A}^{\infty} \, , \, \phi_{B}^{\infty} \, , \, \tau_{X}^{\infty} \, , \, \tau_{Y}^{\infty})}
\newcommand{\fockcovRep}{(\phi_{A} \, , \, \phi_{B} \, , \, \tau_{X} \, , \, \tau_{Y})}
\newcommand{\LL}{\mathscr{L}}
\newcommand{\F}{\mathscr{F}}
\newcommand{\K}{\mathscr{K}}
\newcommand{\la}{\langle}
\newcommand{\ra}{\rangle}

\newcommand{\upDownArrow}[1]{%
	{\left\downarrow\vbox to #1{}\right.\kern-\nulldelimiterspace}
}
\newcommand{\leftRightArrow}[1]{%
	{\left\rightarrow\vbox to #1{}\right.\kern-\nulldelimiterspace}
}

\def\mathcs{C^{*}}
\newcommand{\cs}{\ensuremath{\mathcs}}

\DeclareMathSymbol{\rtimes}{\mathbin}{AMSb}{"6F}

\def\NN{\mathbf{N}}

\def\K{\mathcal{K}}

\def\set#1{\{\,#1\,\}}

\newcommand\A{\mathcal A}

\newcommand\go{\Sigma^{(0)}}

\newcommand\cc{C_{c}}

\renewcommand\H{\mathcal{H}}

\newcommand\Sec{\Gamma}

\newcommand\auto{\vartheta}

\newcommand\gtw{\kappa}

\newcommand\CC{\mathcal{C}}
\newcommand\N{\mathcal{N}}

\newcommand\ccgsgtw{\cc(G,\Sigma,E,\gtw)}
\newcommand\csgsgtw{\cs(G,\Sigma,E,\gtw)}

\theoremstyle{plain}
\newtheorem{theorem}{Theorem}[section]
\newtheorem*{theorem*}{Theorem}
\newtheorem*{prop*}{Proposition}
\newtheorem{cor}[theorem]{Corollary}
\newtheorem{lemma}[theorem]{Lemma}
\newtheorem{prop}[theorem]{Proposition}

\theoremstyle{remark}

\theoremstyle{definition}
\newtheorem{dfn}[theorem]{Definition}
\newtheorem{example}[theorem]{Example}

\numberwithin{equation}{section}

\begin{document}

\title[Cuntz-Nica-Pimsner algebras]
{Cuntz-Nica-Pimsner algebras of product systems over groupoids}
\author[M. Amini, M. Moosazadeh]{Massoud Amini and Mahdi Moosazadeh}
\address{Department of Mathematics, Faculty of Mathematical Sciences, Tarbiat Modares University, Tehran 14115134, Iran}
\email{mamini@modares.ac.ir, mahdimoosazadeh7@gmail.com}

\thanks{}

\begin{abstract}
Let $X$ be a product system over a quasi-lattice ordered groupoid $(G,P)$.
Under mild hypotheses, we associate to $X$ a $C^*$-algebra
which is couniversal for injective Nica covariant
Toeplitz representations of $X$ which preserve the gauge coaction. When $(G,P)$ is a quasi-lattice ordered group this couniversal
$C^*$-algebra coincides with the Cuntz-Nica-Pimsner algebra
introduced by Carlsen-Larsen-Sims-Vittadello, and under some mild amenability conditions with that of Sims and Yeend. We prove related gauge invariant uniqueness
theorems in this general setup. 
\end{abstract}

\keywords{quasi-lattice ordered groupoid, product system, Cuntz-Pimsner algebra, Hilbert bimodule}
\subjclass{Primary 46L05}

\maketitle

\section{Introduction}
The Cuntz algebra $\mathcal {O}_{n}$, introduced by Joachim Cuntz in the late 70s, is the universal $C^*$-algebra generated by $n$ isometries on an infinite-dimensional Hilbert space whose range projections add up to 1.  These  were  the first concrete examples of a separable infinite simple $C^*$-algebra, and any
such algebra contains a subalgebra having some $\mathcal {O}_{n}$ as a quotient. Cuntz  algebras are generalized in the early 80s by  Cuntz and Krieger \cite{CK1980} and in the late 90s by Pimsner \cite{Pimsner1997}. In the Pimsner construction, one gets a $C^*$-algebra $\mathcal O_X$, now known as the Cuntz-Pimsner $C^*$-algebra, for any $C^*$-correspondence $X$ over a $C^*$-algebra $A$. Pimsner also defined a Toeplitz $C^*$-algebras $\Tt_X$, having $\mathcal O_X$ as a quotient \cite{Pimsner1997}. He showed that his construction of $\mathcal O_X$ has crossed
products by $\ZZ$ as an special case. In this case, the Toeplitz 
$\Tt_{X}$ is like a crossed product  by $\mathbb N$.

Katsura was able to relax the assumption of  isometric left
action in the Pimsner construction, and unify the construction by that of $C^*$-algebras of topological graphs \cite{ka1,ka2,ka3}. Katsura proved  that a Cuntz-Pimsner
representation of a $C^*$-correspondence $X$ generates an isomorphic copy
of $\Oo_X$, precisely when the representation is injective and
admits a gauge action. 

In Pimsner and Katsura construction, the representations are handled using the Fock space, built by considering $n$-fold tensors of $X$ for $n\geq 0$, with $A$ playing the role of the 0-fold tensor, which suggests that we are dealing with a system of modules indexed by semigroup of non-negative integers. Building on the notion of Fock space, Fowler introduced and studied
$C^*$-algebras associated to product systems of Hilbert bimodules indexed by a semigroup, sitting in a group \cite{Fowler1999}.
The idea of product systems go to William Arveson, who first studied product systems of Hilbert spaces over the semigroup of positive reals \cite{Arveson1989} (see also,  \cite{Dinh1991}). The Fowler's construction employs Nica's notion of quasi-lattice ordered groups \cite{N}. Fowler
was able to show that $\Tc(X)$ has much of the structure
of a twisted crossed product of the coefficient algebra $A$ by
the semigroup $P$, with a twist coming from $X$.

The drawback in Fowler's construction is that his generalized
Cuntz-Pimsner algebra $\Oo_X$ need not be a quotient of $\Tc(X)$, and  the canonical
homomorphism from $A$ to $\Oo_X$ is not necessarily injective,
leaving little hope for a gauge-invariant uniqueness
theorem.  Sims and Yeend \cite{SY} addressed these issues by introducing the Cuntz-Nica-Pimsner algebra $\mathcal{NO}_{X}$ of a product
system $X$ over a quasi-lattice ordered group $(G,P)$, in such a way that it is a quotient of $\Tc(X)$, and under relatively mild hypotheses, the canonical
representation of the product system $X$ on the
Cuntz-Nica-Pimsner algebra $\mathcal{NO}_{X}$ is isometric.
Finally, Carlsen, Larsen, Sims and Vittadello
establish a gauge-invariant uniqueness theorem for $\mathcal{NO}_{X}$, by an analysis of the fixed-point algebra (core) of the canonical coaction of $G$ on $\mathcal{NO}_{X}$ \cite{clsv}.

In  the above constructions, all correspondences are over a fixed $C^*$-algebra $A$. This restriction could be best understood via the Fock space. In order to tensor two bimodules over $A$ we need to have both the left and right actions to be by the same $C^*$-algebra. In the algebraic language, if the modules in a product system are indexed by elements of a semigroup sitting in a group, it is natural to expect only one $C^*$-algebra, which is then the module indexed by the identity element of the group. Now what if we allow the possibility of some tensors to remain undefined in the Fock space? Then, we could expect bimodules with left and right actions of different $C^*$-algebras, sitting beside each other in a tensor, exactly when the corresponding right and left module structures match, i.e., come from  actions of the same $C^*$-algebra. In the algebraic language, we then have a family of bimodules indexed by a semigroupoid sitting in a groupoid and containing its unit space, and the tensors of two bimodules are allowed exactly when their indices are multipliable in the groupoid, i.e., their source and range match appropriately. In this case, we deal with bimodules over a family of $C^*$-algebras indexed by the unit space of the groupoid.      

The construction of the universal $C^*$-algebra of a product system over a semigroupoid is the main objective of the present paper. We work with a compactly aligned product system $X$ of correspondences over a semigroupoid $P$ inside a quasi-lattice ordered groupoid $G$, and construct the Cuntz-Nica-Pimsner $C^*$-algebra $\NO{X}$ and Toeplitz $C^*$-algebra $\Tc{X}$. Our construction follows closely that of Carlsen et al in \cite{clsv}, but there are many technicalities we have to overcome, due to the fact that we are working with a Fell bundle over a groupoid, whose properties are much less known compared to the fairly well  studied notion of Fell bundle over a group. This being said, we should emphasize that our arguments here are adaptations of those already presented by Fowler \cite{Fowler1999, F99}, Yeend \cite{y}, Sims and Yeend \cite{SY}, and Carlsen et al \cite{clsv}. For Fell bundles over groupoids, we use the definitions and results of Kumjian \cite{kum}. Finally, for coactions, we adapt the ideas and notions used by Exel \cite{Exel} and Quigg \cite{qui:discrete coactions}, though we extend them from coactions of groups to that of groupoids. 

Our main result gives sufficient conditions for the existence and uniqueness of the reduced Cuntz-Nica-Pimsner $C^*$-algebra $\NO{X}^r$, which is couniversal for gauge-compatible injective Nica covariant representations of $X$  (Theorem \ref{thm:projective property}) and the gauge-invariant uniqueness theorem (Corollary~\ref{cor:guit}). Under a natural condition (which is automatic in the group case), we show that the full Cuntz-Nica-Pimsner $C^*$-algebra $\NO{X}$ has gauge-invariant uniqueness property when the canonical coaction of $G$ on the Toeplitz algebra is normal and $G$ is exact (Corollary \ref{cor:amenable}).  We give some of our notions and results for an \'{e}tale groupoid, but in most cases, including the above cited results, we restrict ourselves to discrete groupoids. We give concrete examples to illustrate our construction, and in particular show, in an appendix, that the crossed product of $C_0(X)$-algebras by discrete groupoids, and $C^*$-algebras of multilayered topological $k$-graphs are particular cases of our general construction. We also show how to extend the  Katsura construction of the $C^*$-algebra $\mathcal O_X$ of an $A$-$A$-correspondence $X$, to construct the $C^*$-algebra $\mathcal O_{X,Y}$ of a pair of $A$-$B$-correspondence $X$ and $B$-$A$-correspondence $Y$, and more importantly, show that it satisfies our gauge-invariant uniqueness property.

\section{Preliminaries}

For an \'{e}tale groupoid $G$ with Haar system $\lambda$, we write $g\mapsto i_G(g)$ for the
canonical inclusion of $G$ as partial isometries in the full groupoid $C^*$-algebra
$C^*(G)$. We let $\mu$ be a quasi-invariant measure on $\Gz$ and  
$\I$ denote the  regular representation of $G$.
All $C^*$-algebra tensor products are w.r.t. the minimal tensor product.

\subsection{Hilbert bimodules and correspondences}

Let $A$ and $B$ be $C^*$-algebras. A Hilbert $A$-$B$-bimodule is a
complex vector space $X$ endowed with a left $A$-module
structure and an $A$-valued $A$-sesquilinear form ${}_A\langle
\cdot\,,\,\cdot\rangle$ and a right $B$-module
structure and a $B$-valued $B$-sesquilinear form $\langle
\cdot\,,\,\cdot\rangle_B$, with compatibility condition ${}_A\langle
x\,,\,y\rangle\cdot z=x\cdot\langle
y\,,\,z\rangle_B$, for $x,y,z\in X$, 
such that $X$ is complete in the norm $\|x\| := \|{}_A\langle x,
x \rangle \|^{1/2}=\|\langle x,
x \rangle_B \|^{1/2}$. A map $T \colon X \to X$ is said to be
$A$-adjointable if there is a map $T^* \colon X \to X$ such
that ${}_A\langle Tx, y\rangle = {}_A\langle x, T^*y\rangle$, for
 $x,y \in X$. Every adjointable operator on $X$ is
norm-bounded and linear, and the adjoint $T^*$ is unique. The
collection $\La(X)$ of $A$-adjointable operators on $X$ endowed
with the operator norm is a $C^*$-algebra. The ideal of
 ``compact'' operators $\Ka(X)$ is the
closed span of the ``rank'' one operators $x \otimes y^* \colon z \mapsto x
\cdot{}_A\langle y, z \rangle$, with $x$ and $y$ range over $X$. The same notations are used for $A$ replaced by $B$.

An $A$-$B$-correspondence is a right Hilbert $B$-module $X$ endowed with a left action $\phi_A \colon A \to
\La(X)$ of $A$ by adjointable
operators. 
Let $X$ and $Y$ be an $A$-$B$-correspondence $B$-$C$-correspondence with left actions $\phi_A$ and $\phi_B$, respectively. The balanced tensor product $X \otimes_B Y$ is the completion of the
vector space spanned by
elements $x \otimes_B y$ with $x \in X$ and $y \in Y$, subject
to  $x \cdot b \otimes_B y = x \otimes_A \phi_B(b)y$, in the norm coming by
the inner-product
\[
\langle x_1 \otimes_B y_1, x_2 \otimes_B y_2 \rangle_B := \langle y_1,
\langle x_1, x_2\rangle_B \cdot y_2\rangle_B.
\]
With the right action of $C$ on $X\otimes_B Y$ given by
$(x\otimes_B y)\cdot c=x\otimes_B(y\cdot c)$ and left action of $A$ implemented by  the
homomorphism $a
\mapsto \phi_A(a) \otimes 1_{\mathcal{L}_A(Y)}$,   $X \otimes_B Y$ is an $A$-$C$-correspondence.  

\subsection{Semigroupoids and product systems}\label{subsec_prod_syst}

In this section, we define product systems over subemigroupoids of quasi-lattice ordered groupoids. These base structures are modeled after the notion of quasi-lattice
ordered groups due to Nica \cite{N}. An \'{e}tale groupoid $G$ and a subsemigroupoid $P$ of $G$ with $P \cap
P^{-1} = \Gz$, we say that $(G,P)$ is a quasi-lattice
 ordered groupoid if, under the partial order $g \le h \iff
g^{-1}h \in P$, for $g,h\in G$ with the same range, any pair $p, q$ in $G$ with a common
upper bound in $P$ have a least common upper bound $p \vee q$
in $P$ (compare with the definition in the group case  
\cite{N, F99, CL2002, CL}).
We write $p \vee q <\infty$ when $p\vee q$ exists and $p\vee q= \infty$ when $p,q \in G$ have
no common upper bound in $P$. If $p \vee q<\infty$,
for all $p,q \in P$, we say that $(G,P)$ is directed, or simply say that $P$ is directed. Finally, a quasi-lattice ordered groupoid $(G,P)$ is called totally ordered if for each $p,q\in P$ with the same range, either $p^{-1}q\in P$ or $q^{-1}p\in P$. In this case, we simply say that $P$ is totally ordered. 

A product system over $P$ consists
of a bundle $\{A_u\}$ of $C^*$-algebras over $\Gz$, a semigroupoid $X$, and a semigroupoid homomorphism $d
\colon X \to P$ such that $X_p := d^{-1}(p)$ is an 
$A_{r(p)}$-$A_{s(p)}$-correspondence for each $p\in P$,  $X_u = A_u$, for $u\in\Gz$, and the
multiplication on $X$ implements isomorphisms $X_p \otimes_{A_{r(q)}}
X_q \cong X_{pq}$, for $p,q \in P \setminus \{e\}$ with $(p,q)\in\Gt$, as well as the right $A_{s(p)}$-module structure  
and the left $A_{r(p)}$-actions on each $X_p$.
For $p \in P$, we denote the homomorphism 
 which implements the corresponding left action by $\phi_p: A_{r(p)}\to \mathscr{L}_{A_{r(p)}}(X_p)$. We denote the  $A_{s(p)}$-valued inner product on
 $X_p$ by  $\langle
 \cdot,\cdot\rangle_p$. For $(p,q)\in\Gt$, we 
 have $\phi_{pq}(a)(xy) =
(\phi_p(a)x)y$, for  $x \in X_p$, $y \in X_q$ and $a \in A_{r(p)}$.
Given $p, q \in P\backslash \Gz$, there is a homomorphism
$\iota^{pq}_p \colon \Laa(X_p) \to \Laa(X_{pq})$ given by
$\iota^{pq}_p(S)(xy) = (Sx)y,$ for  $x \in X_p$, $y \in
X_{q}$ and $S \in \Laa(X_p)$.
Identifying $\Kaa(X_{r(p)})$ with $A_{r(p)}$ as
above, we define $\iota^p_{r(p)} \colon \Kaa(X_{r(p)})\to
\Laa(X_{p})$  by $\iota^p_{r(p)}=\phi_p$, for $p\in P$ (compare to  \cite[\S 2.2]{SY}). Finally, under the identification $X_p\otimes_{A_{r(p)}}A_{r(p)}\cong X_p$, we could write 
$\iota_p^p(S)=S\otimes{\rm id}_{A_{r(p)}}.$

A product system $X$ over $P$ is compactly aligned  if $\iota^{p \vee q}_p(S) \iota^{p
\vee q}_q(T) \in \mathscr{K}_{A_{r(p\vee q)}}(X_{p\vee q})$, for each $S \in \Kaa(X_p)$
and $T \in \mathscr{K}_{A_{r(q)}}(X_q)$, whenever $p \vee q$ exists (cf., 
\cite[Definition 5.7]{F99}).
Let us observe that a product system $X$ over a totally ordered quasi-lattice ordered groupoid $(G,P)$ is automatically compactly aligned: 
given $p$ and $q$ with $p\vee q<\infty$, and $S \in \Kaa(X_p), T \in \mathscr{K}_{A_{r(q)}}(X_q)$, we have $p\leq p\vee q$ and $q\leq p\vee q$, thus $r(p)=r(p\vee q)=r(q)$. By total order assumption, $p\vee q$ is either $p$ or $q$, say $p\vee q=p$. Then $\iota_p^{p\vee q}(S)=\iota_p^p(S)=S\otimes{\rm id}_{A_{r(p)}}\in \mathscr{K}_{A_{r(p)}}(X_p)$, thus,
$$\iota_p^{p\vee q}(S)\iota_q^{p\vee q}(T)\in \mathscr{K}_{A_{r(p)}}(X_p)=\mathscr{K}_{A_{r(p\vee q)}}(X_{p\vee q}),$$
as required.

\subsection{Representations of product systems}\label{subsection:reps_prod_syst}
Given a  product system $X$ over $P$, a Toeplitz representation (or simply a representation)
$\psi$ of $X$ in a $C^*$-algebra $D$ is a map $\psi : X \to D$
satisfying,
\begin{enumerate}
\item   $\psi_p:=\psi\vert_{X_p} \colon X_p \to D$
    is linear and $\psi_u: A_u\to D$ is a homomorphism;
\item $\psi_{s(p)}(\langle x,y\rangle_p) = \psi_p(x)^*
\psi_p(y)$ and, $\psi(xz)=\psi(x)\psi(z)$, if $s(p)=r(q),$
\item $\psi(x)\psi(z)=0$, if $s(p)\neq r(q)$, 
\end{enumerate}
for each $p,q\in P$, $u\in\Gz$, and $x,y \in X_p$, $z\in X_q$.

\vspace{.2cm}
We would define the notion of a Nica covariant representation of a compactly aligned product system $X$, and later show that, for such representations, we also have,

\vspace{.2cm}
(4) $\psi(x)\psi(y)^{*}=0$, if $s(p)\neq s(q)$, and  $\psi(x)^*\psi(z)=0$, if $r(p)\neq r(r)$, 

\vspace{.2cm}
\noindent for each $p,q,r\in P$, and $x \in X_p$, $y\in X_q$, $z\in X_r$. 

\vspace{.2cm}
 A Toeplitz representation $\psi$
is injective if  homomorphisms
$\psi_u$ are all injective, for $u\in\Gz$. In this case, the linear maps $\psi_p$ are isometry, for $p\in P$. 
Given a representation $\psi$, we have $*$-homomorphisms $\psi^{(p)} \colon \Kaa(X_p)\to D$
defined by $\psi^{(p)}(x \otimes y)=\psi_p(x) \psi_p(y)^*$, for
 $x,y\in X_p$ (compare to \cite{Pimsner1997}). 
 When $D=\mathscr{L}(\mathcal H)$, for some Hilbert space $\mathcal H$, we say that $\psi$ is nondegenerate if $\psi(X)\mathcal H$ is a total set in $H$, and cyclic if $\psi(X)\xi$ is dense in $\mathcal H$, for some $\xi\in\mathcal H$, called a cyclic vector. Two representations $\psi_1: X\to \mathscr{L}(\mathcal H_1)$ and $\psi_2: X\to\mathscr{L}(\mathcal H_2)$ are unitarily equivalent if there is a unitary $U\in\mathscr{L}(\mathcal H_1,\mathcal H_2)$ such that $\psi_2(x)={\rm ad}_U(\psi_1(x))$, for each $x\in X$.    
 
 As in 
 \cite[Proposition~2.8]{F99}, we have the following result.
 
\begin{lemma} \label{univ}
There is a universal
$C^*$-algebra $\Tt_X$ generated by a universal Toeplitz
representation $i$ of $X$. 	
\end{lemma}
\begin{proof}
We adapt the proof of \cite[Proposition 1.3]{FR}. For a Toeplitz representation $\psi:X\to\mathscr{L}(\mathcal H)$, let $C^*(\psi)$ be the $C^*$-subalgebra of $\mathscr{L}(\mathcal H)$ generated by the range of $\psi$, $\mathcal H_\psi$ be the closed subspace of $\mathcal H$ generated by $C^*(\psi)\mathcal H$, and $P_\psi: \mathcal H\to \mathcal H_\psi$ be the corresponding orthogonal projection. Then it is easy to check that 
$$P\psi: X\to \mathscr{L}(\mathcal P_\psi H); \ \ x\mapsto P_\psi\psi(x)P_\psi,$$
is again a Toeplitz representation of $X$, and is nondegenerate. Since $(1-P_\psi)\psi$ is the zero representation, it follows that $\psi$ is a direct sum of a nondegenerate representation and a zero representation. 

Next let us show that nonzero cyclic representations exist. For this we need to construct the Fock representation. Consider the $\Gz$-bundle of right Hilbert modules, $
F(X)_u := \bigoplus_{r(p)=u} X_p,$ where the right hand side is
the subset of $\prod_{p\in P} X_p$
consisting of all elements $(x_p)$ for which
$\sum_{r(p)=u} \langle x_p, x_p \rangle_p$ is summable in $A_{r(p)}$.
We write $\oplus x_p$ for $(x_p)$ to indicate that the above series is summable.
Each $F(X)_u$ is a right Hilbert $A_u$-module with the right action  $(\oplus x_p)\cdot a := \oplus (x_p \cdot a)$, for $a\in A_u$, and inner product $\langle \oplus x_s, \oplus y_s \rangle_{A_u}
:= \sum_{r(p)=u} \langle x_p, y_p \rangle_p$.
Define the Fock representation of $X$ by $$l:=\oplus l_u: X\to \bigoplus_{u\in\Gz} \mathscr{L}_{A_u}(F(X)_u),$$ where $l_u: X\to \mathscr{L}_{A_u}(F(X)_u)$ is defined  by 
$l_u(x)(\oplus y_p):=\oplus z_q,$ with $z_q:=xy_{r^{-1}q},$ if $r\leq q$, and $z_q:=0,$ otherwise; when $x\in X_r$, for some $r\in P$, and the other two indices run over $p,q\in P$ with $r(p)=r(q)=u$.     
Each $l(x)$ is adjointable, and 
$l(x)^*(yz) = \langle x, y\rangle_p\cdot z$ for $x,y\in X_p$ and $z\in X_q$, with $r(q)=s(p)$. 

Let us observe that  $l$ 
is a Toeplitz representation. For $x_r\in X_r, x_s\in X_s$ with $r(s)=s(r)$, 
$$l_u(x_r)l_u(x_s)(\oplus y_p)=l_u(x_r)(\oplus z_q)=(\oplus w_t),$$ with $z_q:=x_sy_{s^{-1}q},$ if $s\leq q$, and $z_q:=0,$ otherwise, and $w_t:=x_rz_{r^{-1}t}$, if $r\leq t$, and $w_t:=0,$ otherwise, for $r(t)=r(p)=r(q)=u$. Similarly, 
$$l_u(x_rx_s)(\oplus y_p)=(\oplus w^{'}_t),$$ with $w^{'}_t:=x_rx_sy_{s^{-1}r^{-1}t},$ if $rs\leq t$, and $w^{'}_t:=0,$ otherwise,  for $r(t)=r(p)=r(q)=u$. Now since the inequalities $r\leq t$ and $s\leq r^{-1}t$ both hold iff $rs\leq t$, we have $w_t=x_rx_sy_{s^{-1}r^{-1}t}=w^{'}_t$, for each $t$, that is, $l(x_r)l(x_s)=l(x_rx_s)$, when $r(s)=s(r)$. On the other hand, if the inequalities $r\leq t$ and $s\leq r^{-1}t$ both hold, we have $r(s)=r(r^{-1}t)=s(r)$, which means that if $r(s)\neq s(r)$, then $w_t=0$, for each $t$, that is,  $l(x_r)l(x_s)=0$, in this case.

Next, the restriction of $l$ to $A_u$ is nothing but $\oplus_{r(p)=u} \phi_p$, which is known to be isometric. If we faithfully represent the  $C^*$-algebra $\bigoplus_{u\in\Gz} \mathscr{L}_{A_u}(F(X)_u)$ in a Hilbert space $\mathcal H$, the Fock representation $l$ would have a nonzero cyclic summand.

Now by Zorn lemma, every nondegenerate representation is a direct sum of cyclic representations. Let $\mathfrak S$ be a set of cyclic representations of $X$ such that every representation $\psi: X$ in a Hilbert space $H$ is unitarily equivalent to an element of $\mathfrak S$. Such a set $\mathfrak S$ exists an could be realized as the set of cyclic representations on subspaces of a Hilbert space of sufficiently large dimension. Now let $i$ be the direct sum of all the representations in $\mathfrak S$, which makes sense, since for each $\psi\in\mathfrak S$ and $x\in X$,
$$\|\psi(x)\|=\|\psi_{s(d(x))}(\langle x,x\rangle_{d(x)})\|^{\frac{1}{2}}\leq\|\langle x,x\rangle_{d(x)}\|^{\frac{1}{2}}=\|x\|.$$    
Put $\Tt_X:=C^*(i)$, then $\Tc(X)$ is universal among all $C^*$-algebras generated by range of a Toeplitz representation of $X$, by the definition of $\mathfrak S$.  
  \end{proof}
When
$X$ is a compactly aligned, 
a Toeplitz representation $\psi$ of $X$ is Nica-covariant if
\[
\displaystyle \psi^{(p)}(S)\psi^{(q)}(T) =
\begin{cases}
\psi^{(p\vee q)}\big(\iota^{p\vee q}_p(S)\iota^{p\vee q}_q(T)\big)
& \text{if $p\vee q$\ exists} \\
0 &\text{otherwise,}
\end{cases}
\]
for all $S \in \Kaa(X_p)$ and $T \in \mathscr{K}_{A_{r(q)}}(X_q)$ (compare to \cite[Definition 5.7]{F99}). Let $\Tc(X)$ be the quotient of
$\Tt_X$ by the ideal generated by the elements
$$
i^{(p)}(S)i^{(q)}(T) - i^{(p \vee q)}(\iota^{p \vee q}_p(S)
\iota^{p \vee q}_q(T)),
$$ where $p,q \in P$, $S \in \Kaa(X_p)$, $T
\in \mathscr{K}_{A_{r(q)}}(X_q)$, with convention, $\iota^{p \vee q}_p(S)
\iota^{p \vee q}_q(T) = 0$ if $p \vee q = \infty$. The
composition  of  the quotient map from $\Tt_X$ onto $\Tc(X)$
with the universal Toeplitz representation $i$ is a Nica covariant Toeplitz representation $i_X
\colon X\to \Tc(X)$ with the following universal property: if
$\psi$ is a Nica covariant Toeplitz representation of $X$ in
$D$ there is a $*$-homomorphism $\psi_* : \Tc(X) \to D$ such
that $\psi_*\circ i_X=\psi$. When $X$ is a compactly aligned
product system of essential Hilbert bimodules, $\Tc(X)$
resembles Fowler's algebra \cite{F99} for
product systems of essential Hilbert bimodules over quasi-lattice ordered groups. 

\begin{lemma}\label{closed2}
	Let $X$ be compactly-aligned product system over $P$,
	$\psi: X\to D$ be a Nica-covariant Toeplitz representation of $X$,
	$p,q\in P$, $y\in X_p$, and $z\in X_q$.
	Then  $\psi(y)^*\psi(z) = 0$, if $p\vee q = \infty$; and
	otherwise,
	\[
	\psi(y)^*\psi(z)
	\in\clsp\{\psi(u)\psi(v)^*: u\in X_{p^{-1}(p\vee q)}, v\in X_{q^{-1}(p\vee q)}\}.
	\]
\end{lemma}

\begin{proof} Let us write $y = Sy'$ with $S\in\mathscr{K}_{A_{r(p)}}(X_p)$ and $y'\in X_p$, and $z = Tz'$ with $T\in\mathscr{K}_{A_{r(q)}}(X_q)$ and $z'\in X_q$.
	Since $\psi$ is Nica covariant,
	\[
	\psi(y)^*\psi(z) = \psi(y')^*\psi^{(p)}(S^*)\psi^{(q)}(T)\psi(z')
	\]
	which is zero if $p\vee q = \infty$, and is otherwise equal to
	$
	\psi(y)^*\psi(z) = \psi(y')^*\psi^{(p\vee q)}(K)\psi(z'),
	$
	for $K := (S^*\otimes_{A_u} 1)(T\otimes_{A_u} 1) \in\mathscr{K}_{A_u}(X_{p\vee q})$, $u:=r(p\vee q)$.
	Since  $K$ is compact, it can be approximated in norm by a finite sum of operators of the form
	$u\otimes v^*$,  with $u,v\in X_{p\vee q}$,
	thus $\psi^{(p\vee q)}(K)$ can be approximated by finite sums
	of elements in $D$ of the form $\psi(u_1)\psi(v_1)^*$.
	We may respectively approximate $u_1$ and $v_1$ by finite sums
	of products of the form $u_2u'$ and $v_2v'$, with $u_2\in X_p$, $v_2\in X_q$, and $u'\in X_{p^{-1}(p\vee q)}$, $v'\in X_{q^{-1}(p\vee q)}$.
	Hence $\psi(y')^*\psi^{(p\vee q)}(K)\psi(z')$ can be approximated in norm
	by finite sums of operators of the form
	\[
	\psi(y')^*\psi(u_2)\psi(u')\psi(v')^*\psi(v_2)^*\psi(z')
	= \psi(\langle y',u_2 \rangle_p u')\psi(\langle z',v_2 \rangle_q v')^*=:\psi(u)\psi(v)^*,
	\]
	as claimed. 
\end{proof}

\begin{lemma}\label{zero}
Let $X$ be compactly-aligned product system over $P$,
$\psi: X\to D$ be a Nica-covariant Toeplitz representation of $X$, then $\psi(x)\psi(y)^{*}=0$, if $s(p)\neq s(q)$, and  $\psi(x)^*\psi(z)=0$, if $r(p)\neq r(r)$, for each $p,q,r\in P$, and $x \in X_p$, $y\in X_q$, $z\in X_r$. 
\end{lemma}
\begin{proof}
Let us write $x = Sx'$ with $S\in\mathscr{K}_{A_{r(p)}}(X_p)$ and $x'\in X_p$, and $y = Ty'$ with $T\in\mathscr{K}_{A_{r(q)}}(X_q)$ and $y'\in X_q$. Then,
\[
\psi(x)\psi(y)^* = \psi^{(p)}(S)\psi(x')\psi(y')^*\psi^{(q)}(T^*).
\]
As above, let us approximate compact operators $S$ and $T^*$ in norm, respectively  by  finite sums of the form
$u\otimes v^*$ and $u'\otimes v'^{*}$, with $u,v\in X_{p}$ and $u',v'\in X_q$, and observe that if $s(p)\neq s(q)$,
\begin{align*}
\psi^{(p)}(u\otimes v^*)\psi(x')\psi(y')^*\psi^{(q)}(u'\otimes v'^{*})&=\psi_p(u)\psi_p(v)^*\psi_p(x')\psi_q(y')^*\psi_q(u')\pi_q(v')^*\\&=\psi_p(u)\psi_{s(p)}(\langle v,x'\rangle_p)\psi_{s(q)}(\langle y',u'\rangle_q)\pi_q(v')^*=0,
\end{align*}
since the product of  two middle terms is zero by part (3) of the definition of a Toeplitz representation. This proves the first statement. The second statement, follows from Lemma \ref{closed2}, as $p\vee r=\infty$ if $r(p)\neq r(r)$. 
\end{proof}

As in  \cite[Theorem~6.3]{F99}, it follows that, 
\begin{equation*}\label{eq:Tc spanners}
\Tc(X)=\clsp\{\, i_X(x)i_X(y)^* \mid x, y \in X \,\}.
\end{equation*}
Indeed, the right hand side is a norm closed, selfadjoint subspace of $\Tc(X)$, which includes $i_X(X)$, as each $z\in X$ could be written as $z=z'\cdot a$, for some $z'\in X_p$ and $a\in A_{s(p)}$, that gives $i_X(z)=i_X(z')i_X(a^*)^*$. The equality now follows from the fact that the right hand side is also closed under multiplication, which in turn follows from Lemma \ref{closed2}.    
By a similar argument, we have the following more general result.

\begin{lemma}\label{image}
If the image of a Nica covariant
Toeplitz representation $\psi$ of $X$ in $D$ generates $D$ as a $C^*$-algebra, then
$D= \clsp\{\, \psi(x)\psi(y)^* \mid x,y \in X \,\}$.
\end{lemma}

\subsection{The Cuntz-Nica-Pimsner algebra $\NO{X}$}\label{subsection_def_NOX}
Next, let us define Cuntz-Pimsner covariance of representations, adapted from \cite[Section~3]{SY}. A predicate statement $\mathcal{P}(s)$, with $s \in
P$, is said to be true for large $s$ if for every $q \in P$ there exists
$r \ge q$ such that $\mathcal{P}(s)$ is true whenever $r \le
s$.

When $X$ is a compactly
aligned over $P$, we let $I_u = A_u$, for $u\in\Gz$, and put $I_q := \bigcap_{0<p \le q}
\ker(\phi_p)$, for 
$q \in P \backslash \Gz$, which is a closed two-sided ideal in $A_{r(q)}$, where the notation $0<p\leq q$ simply means that the intersections runs over those $p\in P\backslash \Gz$ satisfying $p\leq q$. We put,
$
\widetilde{X}_q := \textstyle{\bigoplus}_{p \le q} X_p \cdot I_{p^{-1}q}, 
$
regarded as an $A_{r(q)}$-$A_{s(q)}$-correspondence, with the homomorphism implementing the left action denoted by
$\tilde\phi_q$. When all the homomorphisms
$\tilde\phi_q$ are  injective, we say that $X$ is
$\tilde\phi$-injective. 

With the convention of  $\iota^q_p(T)$ being 0 in ${\mathscr{L}_{A_{r(q)}}(X_q)}$, for $T \in \Laa(X_p)$, when $p \not\le q$, we 
have homomorphisms $\tilde\iota^q_p \colon \Laa(X_p) \to
\mathscr{L}_{A_{r(q)}}(\widetilde{X}_q)$, for $p \in P\backslash\Gz$ and  $q \in P$, defined by $\tilde\iota^q_p(T) = \bigoplus_{r \le q}
\iota^r_p(T)$ for $T\in \Laa(X_p)$. Similarly, for $u\in\Gz$,
there is a homomorphism $\tilde\iota^q_u :
\mathscr{K}_{A_u}(X_u) \to \mathscr{L}_{A_{r(q)}}(\widetilde{X}_q)$. 

When  $X$ is moreover $\tilde\phi$-injective, a Nica covariant
Toeplitz representation $\psi$ of $X$ in a $C^*$-algebra $D$ is called Cuntz-Nica-Pimsner covariant (abbreviated as CNP-covariant)
if for each finite subset $F \subset P$ and any choice of $T_p\in \mathscr{K}_{A_{r(p)}}(X_p)$, for $p\in F$, the finite sum $\sum_{p \in F} \psi^{(p)}(T_p)$ is zero in $D$, whenever 
the finite sum $\sum_{p \in F} \tilde\iota^q_p(T_p)$ is zero in $\mathscr{L}_{A_{r(q)}}(\widetilde{X}_q)$, for large $q$.
We write $\NO{X}$ for the universal
$C^*$-algebra generated by a CNP-covariant representation $j_X$
of $X$, and call it the Cuntz-Nica-Pimsner algebra of $X$. 

Consider the closed two sided ideal $I_{\rm cov}$ of $\Tc(X)$
generated by
the set of finite sums $\sum_{p \in F} i_X^{(p)}(T_p)$, with  $T_p \in \Kaa(X_p)$, for $p\in F\subset P$, such that $\sum_{p \in E}
\iotale_p^s(T_p) = 0$, for large $s$, then it follows immediately from definition that $\NO{X}$ is isomorphic to the quotient $\Tc(X)/I_{\rm cov}$. Also, for each CNP-covariant representation $\psi: X\to D$, and corresponding homomorphism $\psi_*: \Tc(X)\to D$ satisfying $\psi=\psi_*\circ i_X$, by covariance property of $\psi$, we have $I_{cov}\subseteq \ker(\psi_*)$, therefore, $\psi_*$ descends to a homomorphism $\Pi\psi: \NO{X}\to D$, satisfying  $\psi=\Pi\psi\circ j_X$.   

Similar to \cite[Theorem~4.1]{SY}, let us show that  $j_X$
is an injective representation, whenever each representation  $\tilde\phi_p$ is injective, or each bounded subset of $P$ has a maximal element. This needs some preparation. 

Following \cite{SY}, we define the augmented Fock representation using the $\Gz$-bundle $
\tilde F(X)_u := \bigoplus_{r(q)=u} \tilde X_q=\bigoplus_{r(q)=u}\bigoplus_{p \le q} X_p\cdot I_{p^{-1}q}.$
Let $X$ be a compactly aligned product system over $P$ and $l$ be its Fock representation constructed as in Lemma \ref{univ}. For $p,q,r \in P$ with $p \le q$, and $r(p)=r(q)=s(r)=u$, let $x \in
		X_r$ and $z \in X_p \cdot I_{p^{-1} q}$. Write $z = z' \cdot a$ with  $z' \in X_p$ and $a
		\in I_{p^{-1}q} \subset A_u$, then $xz = x(z' \cdot a) = (xz')\cdot a \in
		X_{rp} \cdot I_{(rq)^{-1}rq}$. This observation shows that, if we regard each $\tilde X_q$ as a submodule of $F(X)_u$, and given $x\in X$, let $\tilde l_u(x)$ be the restriction of $l_u(x)$ to $\tilde F(X)_u$, then $\Augl(x) \in
		\mathscr{L}_{A_u}(\tilde F(X)_u)$. Let us observe that  The map $x \mapsto \Augl(x):=\oplus \tilde l_u(x)$ is a Nica covariant representation of $X$ in
		$\bigoplus_{u\in\Gz} \mathscr{L}_{A_u}(\tilde F(X)_u)$:
		if $x\in X_p$ and $y=\oplus y_s\in \tilde F(X)_u$
		with $y_s= 0$, for all $s \ge p$, we need to show that $\Augl(x)^*y =
		0$. By linearity and continuity of $\Augl(x)^*$, we
	may assume that $y$ has just one nonzero coordinate, say there exists $r \in P$ such that $r \not\ge
	p$ and $y_s= 0$, for $s \not= r$. For $z \in \tilde F(X)_u$, we 
	have,
	\[
	\langle \Augl(x)^* y, z \rangle_{A_u}
	= \langle y, \Augl(x) z \rangle_{A_u}
	= \sum_{s \in P} \langle y_s, (\Augl(x) z)_s \rangle_s
	= \langle y_r, (\Augl(x) z)_r \rangle_r
	= 0,
	\]
since $r \not\ge p$.

Let $B_X$ (resp., $B_{\tilde X}$) be the $C^*$-algebra generated by the range of the (resp., augmented) Fock representation $\Augl$ in 
$\bigoplus_{u\in\Gz}\mathscr{L}_{A_u}(F(X)_u)$ (resp., in $\bigoplus_{u\in\Gz}\mathscr{L}_{A_u}(\tilde F(X)_u)$), which is called the (resp., augmented) Fock $C^*$-algebra of $X$. It follows from Lemma \ref{closed2} that $B_{\tilde X} = \clsp\{\Augl(x)
\Augl(y)^* : x,y \in X, d(x)=d(y)\}$.

Consider the closed two sided ideal $I_{\tilde X}$ of $B_{\tilde X}$
generated by
the set $K_{\tilde X}$ consisting of finite sums $\sum_{p \in E} \Augl^{(p)}(T_p)$, with  $T_p \in \Kaa(X_p)$, for $p\in E\subset P$, such that $\sum_{p \in E}
		\iotale_p^s(T_p) = 0$, for large $s$, then,
		\[I_{\tilde X} = \clsp\{\Augl(x)\Augl(y)^* k \Augl(x') \Augl(y')^* : k\in
		\overline{K_{\tilde X}}, x,y,x',y' \in X, d(x)=d(y), d(x')=d(y')\}.
		\]
Let  $q_{\tilde X} : B_{\tilde X} \to B_{\tilde X}/I_{\tilde X}$ be the quotient
map. Then $\psi_{\tilde X} := q_{\tilde X} \circ \Augl$ is a CNP-covariant
representation of $X$ in $B_{\tilde X}/I_{\tilde X}$.

\begin{lemma} \label{faithful}
	Let $X$ be a compactly aligned product system over $P$, and $\phile_p$ is injective, for $p\in P$.
	Then the restriction of $\psi_{\tilde X}$ on each fibre $X_u=A_u$, $u\in\Gz$,  is faithful.
\end{lemma}
\begin{proof}
	We just need to show that $I_{\tilde X} \cap \Augl_u(A_u) = \{0\}$, for each $u\in \Gz$.
	For $u\in \Gz$ and $a \in A_u$, the restriction of
	$\Augl(a)$ to the summand $\tilde X_q$ is the same as $\phile_{q}(a)$. In particular, 	$\|\Augl(a)|_{\Xle{q}}\|=\|a\|$, for each $q$ with $r(q)=u$.  It therefore suffices to show that,
	for each $\psi \in I_{\tilde X}$ and  $\varepsilon > 0$,
			there is $s \in P$ with $\|\psi|_{\Xle{s}}\| <
			\varepsilon$.
	
	Let $K_{\tilde X}$ be the generating set of $I_{\tilde X}$, defined above. By a continuity argument, for each $\psi \in \overline{K}$, $\varepsilon
	> 0$, and $q \in P$, there exists $r \ge q$ such that
	$\|\psi|_{\Xle{s}}\| < \varepsilon$, for $s \ge r$.
For $k \in \overline{K_{\tilde X}}$, norm one elements $x,y \in
	X_p$, $x',y' \in X_t$, and  $q \in P$ with $r(q)=u$, let us observe that there exists
	$r \ge q$ such that $\|\Augl(x)\Augl(y)^* k \Augl(x')
	\Augl(y')^*|_{\Xle{s}}\| < \varepsilon$, for each $s \ge r$: if $t \vee q < \infty$, choose  $r' \ge s(t)(t \vee q)$ with
	$\|k|_{\Xle{s'}}\| < \varepsilon$, for $s' \ge
	r'$. Put $r =s(t)r' \ge q$, then for each $s \ge
	r$, we have $s(t)s \ge r'$, and 
	$
	\|\Augl(x)\Augl(y)^* k \Augl(x') \Augl(y')^*|_{\Xle{s}}\| < \varepsilon.
	$
	Now if $t \vee q = \infty$, for $s \ge q$, 
	$s \not\ge t$, thus $\Augl(y')^*|_{\Xle{s}} = 0$, and for $r := q$, we get $\Augl(x)\Augl(y)^* k \Augl(x')
	\Augl(y')^*|_{\Xle{s}} = 0$, for each $s \ge r$, as claimed.
	\end{proof}

A similar argument shows that, for an appropriate ideal $I_X$ of $B_X$ and quotient map $q_X: B_X\to I_X$, $\psi_X:=q_X\circ l$ is a CNP-covariant representation of $B_X/I_X$, which is injective on each fibre $X_u=A_u$, $u\in\Gz$,  is faithful, whenever all representations $\phile_p$ are injective.
  
\begin{prop} \label{faithful2}
	Let $X$ be a compactly aligned product system over $P$, such that each $\tilde\phi_p$ is injective, or each bounded subset of $P$ has a maximal element. Then $j_X$
is an injective representation of $X$.
\end{prop}
\begin{proof}
Under the first assumption, by Lemma 	\ref{faithful} and the universal
	property of $\NO{X}$,  implies that $j_X$ is injective, and so isometric, on each fibre $X_{u}
	= A_u$, for $u\in\Gz$. Now for $p\in P$ and $x\in X_p$, 
	$
	\|x\|^2 = \|\langle x,x \rangle_p \| = \|j_X(\langle x,x\rangle_{r(p)})\| =
	\|j_X(x)^* j_X(x)\| = \|j_X(x)\|^2,
	$
	as required.
\end{proof}

For the rest of paper, $(G,P)$ be a quasi-lattice ordered groupoid and $X$ is a
product system over $P$ consisting of $C^*$-correspondences.

\section{Analysis of the core}\label{sec:core}

In this section we discuss the fixed-point algebra of $\Tc(X)$ under a
canonical coaction $\delta$, and show that 
under certain conditions, $\NO{X}$ satisfies an analog of the criterion (B) of Seems and Yeend \cite[Section~1]{SY}. 

\subsection{The core}
Let us use $q_{\CNP} \colon \Tc(X) \to \NO{X}$ to denote the canonical surjection
arising from the universal property of $\Tc(X)$.

\begin{lemma}\label{lem:psi equal} 
Let $\psi \colon X \to D$ be a Toeplitz representation of $X$.

$(i)$  $\psi_t(\iota_{p}^{t}(T)(x)) = \psi^{(p)}(T) \psi_t(x)$, for  $p \le t \in P$, $T \in
\Kaa(X_p)$, and $x \in X_t$, and $\psi^{(r)} (T) \psi_t (x \cdot a) = 0$, for $t < r \le s \in P$, $T \in
\mathscr{K}_{A_{r(r)}}(X_r)$, and $x \cdot a \in X_t \cdot I_{t^{-1}s}$,

$(ii)$ when $X$ is compactly aligned and  $\psi$ is Nica
covariant, also  $\psi^{(p)} (T)
\psi_t (x \cdot a) = 0$, for  $p,t \le s
\in P$, $p \not\le t$, $T \in \Kaa(X_p)$, and $x
\cdot a \in X_t \cdot I_{t^{-1}s}$,

$(iii)$ when $X$ is
compactly aligned and either  the left action on each
fibre is injective $P$ is directed, and $\psi$ is injective and Nica covariant, for any finite subset $F \subset
P$, $p\in F$, and operators $T_p \in \Kaa(X_p)$, we have $\sum_{p
	\in F} \iotale^s_p(T_p) = 0$, for large $s$, whenever $\sum_{p \in F} \psi^{(p)}(T_p) = 0$. 
\end{lemma}

\begin{proof}
$(i)$ For the first equality, if $u\in\Gz$, the first statement  follows from 
$\mathscr{K}_{A_u}(X_u) \cong A_u$ and $\iota_{u}^{t} = \phi_{t}$, and if $p \in P\backslash\Gz$, since given $t\in P$ with $r(t)=r(p)$, $\{\, xy \mid x \in X_p,\, y \in X_{p^{-1}t} \,\}$
is dense in $X_t$, and $\{\, w \otimes z^* \mid w,z \in X_p \,\}$ is
total in $\Kaa(X_p)$, we need only to observe that for
$x,w,z \in X_p$ and $y \in X_{p^{-1}t}$, 
\begin{align*}
\psi_t (\iota_{p}^{t} (w \otimes z^*) (xy))
&= \psi_p (w \cdot \langle z, x \rangle_{A}^{p}) \psi_{p^{-1}t} (y)
= \psi_p (w) \psi_e (\langle z, x \rangle_{A}^{p}) \psi_{p^{-1}t} (y)
\\&= \psi_p (w) \psi_{p}^* (z) \psi_p (x) \psi_{p^{-1}t} (y)
= \psi^{(p)} (w \otimes z^*) \psi_{t} (xy).
\end{align*}

For the second equality, if $u\in \Gz$, then $x \cdot a \in I_{s}\subseteq \ker (\phi_{r})$. The equality $\psi_u(y)\psi^{(r)}(S)=\psi^{(r)}(\phi_r(y)S)$ first follows for $y\in X_r$ and $S:=w\otimes z^*$, $w,z\in X_r$, and then for any 
$S\in\mathscr{K}_{A_r(r)}(X_r)$. Taking adjoint and letting $y:=(x\cdot a)^*$ and $S=T^*$, we get $
\psi^{(r)} (T) \psi_{u} (x \cdot a) = \psi^{(r)} (T \phi_r (x \cdot a)) = 0$, as required. When $t\in P\backslash\Gz$, for $y \in X_t$, $z \in X_{t^{-1}r}$,
and $v \in X_r$, with $r,s,t$ having the same range, since $a \in I_{t^{-1}s}$, $\phi_{t^{-1}r} (a) = 0$, 
\begin{align*}
\psi^{(r)} &(v \otimes (yz)^*) \psi_t (x \cdot a)
= \psi_r (v) \psi_{t^{-1}r} (z)^* \psi_t (y)^*  \psi_t (x \cdot a)
= \psi_r (v) \psi_{t^{-1}r} (z)^* \psi_{s(t)} (\langle y,x \cdot a \rangle_{t})
\\&= \psi_r (v) \psi_{t^{-1}r}(z)^* \psi_{s(t)} (\langle y ,x \rangle_{t}a)
= \psi_r (v) (\psi_{s(t)} (a^* \langle x,y \rangle_{t}) \psi_{t^{-1}r} (z))^*
\\&= \psi_r (v) \psi_{t^{-1}r} (\phi_{t^{-1}r} (a^* \langle x,y \rangle_{t})z)^*
= \psi_r (v) \psi_{t^{-1}r} (\phi_{t^{-1}r} (a)^* \phi_{t^{-1}r} (\langle x,y \rangle_{t})z)^*
=0.
\end{align*}

$(ii)$ Let $(E_k)$ be an approximate identity for $\mathscr{K}_{A_{r(s)}}(X_t
\cdot I_{t^{-1}s})$. Since $p,t \le s$, $p \vee t$ exists and by the Nica covariance of $\psi$, 
\begin{align*}
	\psi^{(p \vee t)}\big(\iota^{p \vee t}_p(T)\iota^{p \vee t}_t(E_k)\big)
	\psi_t(x \cdot a) &=\psi^{(p)}(T) \psi^{(t)}(E_k) \psi_t(x \cdot a) \psi^{(p)}(T)\psi_t (E_k(x \cdot a))\\&\to \psi^{(p)}(T)\psi_t (x \cdot a),
\end{align*}
as $k\to \infty$. Since $p \not\le t$, we have $t <  p \vee t\le
s$, and the statement follows from the second statement in $(i)$.

$(iii)$ 
Given $q \in P$, list the elements of $F$ as $p_1,
\dots, p_n$. Define $r_0 := q$, and inductively, for $1 \le
i \le n$, define $r_i := r_{i-1} \vee p_i$, if the right hand side exists, and  $r_i:=	r_{i-1}$, otherwise.
Then $r_n \ge p$ and $r_n \vee p$ exists, , for each $p \in F$, and so $r_n$ is an upper bound for $F$ if $P$ is directed. 

For $s \ge r_n$, and
$x \cdot a \in X_{t} \cdot I_{t^{-1}s}$, for $r(t)=r(s)$. By part $(i)$, 
\[
	\textstyle{\psi_t \big(\sum_{p \le t} \iota_{p}^{t} (T_p) (x \cdot a)\big)}
	= \textstyle{\sum_{p \le t} \psi_t (\iota_{p}^{t} (T_p) (x \cdot a))} \notag
= \textstyle{\sum_{p \le t} \psi^{(p)} (T_p) \psi_t (x \cdot a)},
\]
where the sums run over $p\in F$ with $p\le t$. 
If each $\phi_p$ is injective, then $I_{t^{-1}s}=0$
for $t<s$, and so $a=0$ unless $t=s$. Let $(E_k)$ be
an approximate identity for $\mathscr{K}_{A_{r(s)}}(X_s)$. By assumption $p\not\le s$ and $p\vee s=\infty$. By Nica covariance,
$\psi^{(p)}(T_p)\psi_s(x)=\lim_{k} \psi^{(p)}(T_p)
\psi^{(s)}(E_k) \psi_s(x)=0$, for $T_p\in\Kaa(X_p)$ and $x\in
X_s$. Therefore, $$\psi_t(\sum_{p\le
	t}\iota_p^t(T_p)(x\cdot a))=\sum_{p\in F}\psi^{(p)}(T_p)\psi_t(x\cdot
a)=0.$$
The same holds when $P$ is directed, since in this case, $p\le r\le s$ and part $(ii)$ applies. Since $\psi$ is injective,
it follows that  $\sum_{p \le t}
\iota_{p}^{t} (T_p) (x \cdot a) = 0$, as required. 
\end{proof}

By part $(i)$ in the above lemma, if $r \in
tP\setminus\{t\}$, then the element $\psi^{(r)}(T) \in D$ annihilates
$\psi_{t}(X_{t} \cdot I_{t^{-1}s})$, whenever $s \in rP$, and under the assumptions of part $(iii)$, we may replace the condition 
$r \in tP\setminus\{t\}$ with the weaker condition $t \not\in rP$, and $s=t\vee r$.

When $X$ is compactly aligned, by an argument as the one for $\Tc(X)$ (see the last two paragraphs of subsection \ref{subsection:reps_prod_syst}), it follows  from the Nica-covariance of $i_X$ that
\begin{equation*}\label{eq:F-def}
\Ff := \clsp\{\, i_X(x) i_X(y)^* \mid x,y\in X, d(x) = d(y) \,\}
\end{equation*}
is a $C^*$-subalgebra of $\Tc(X)$, called the core of $\Tc(X)$.

For $u\in \Gz$, let us observe that $$\Ff_u := \clsp\{\, i_X(x) i_X(y)^* \mid x,y\in X, d(x) = d(y)\in G^u \,\}$$
 is a $C^*$-subalgebra of $\Ff$.  Given $x_1, y_1\in X_p$ and $x_2,y_2\in X_q$ with $r(p)=r(q)=u$, the element $i_X(x_1) i_X(y_1)^*i_X(x_2) i_X(y_2)^*
$
is zero if $p\vee q=\infty$, and otherwise is a limit of linear combinations of elements of the form $$i_X(x_1) i_X(u)i_X(v)^* i_X(y_2)^*=i_X(x_1u) i_X(y_2v)^*
,$$ with $u\in X_{p^{-1}(p\vee q)}$ and $v\in X_{q^{-1}(p\vee q)}$, which are again in $\Ff_u$, as $d(x_1u)=pp^{-1}(p\vee q)=p\vee q\in G^u$ and $d(y_2v)=qq^{-1}(p\vee q)=p\vee q\in G^u$, showing that $\Ff_u$ is closed under multiplication, and so a $C^*$-subalgebra of $\Ff$. 

\subsection{Gauge coaction}

When $G$ is discrete groupoid with Haar system consisting of counting measures, there is a  $*$-homomorphism $\delta_G
\colon C^*(G)\to C^*(G) \otimes C^*(G)$, given by
$\delta_G(i_G(g))=i_G(g)\otimes i_G(g)$, for $g\in G$. This is a coaction in the following sense (compare to \cite{qui:discrete coactions}). We should warn the reader that our notion of groupoid coaction on $C^*$-algebras deviates from the one defined on $W^*$-algebras by Yamanouchi \cite{yama} based on Connes' notion of fibred products (cf., \cite{or}, \cite{tim}). 

Recall that a $*$-homomorphism $\phi: A\to D$ between $C^*$-algebras is called nondegenerate if $\phi(A)D$ is dense in $D$. This is automatic when $A$ and $D$ are unital and $\phi$ is a unital map. In general, if $\phi$ maps an approximate identity of $A$ to one for $D$, then $\phi$ is nondegenerate. A nondegenerate $*$-homomorphism $\phi: A\to D$ has a unique extension to an $*$-homomorphism $\tilde\phi: M(A)\to M(D)$ between the multiplier algebras.  Indeed, we have the following useful characterization, which would be needed later. This is stated and proved in \cite[Proposition 2.5]{Lan} for the case where $D=\mathscr{L}_B(E)$, for some $C^*$-algebra $B$ and right Hilbert $B$-module $E$, but the same proof essentially works in the general case as well. Recall that the strict topology on $M(A)$ is the locally convex topology given by the family of seminorms  $x\mapsto \|ax\|$ and $x\mapsto \|xa\|$ on $M(A)$, for $a$ running over $A$.     

\begin{lemma} \label{ndeg}
	For a $*$-homomorphism $\phi: A\to D$ between $C^*$-algebras, the following are equivalent:
	
	$(i)$ $\phi$ is nondegenerate,
	
	$(ii)$  $\phi$ has an extension to a unital $*$-homomorphism $\tilde\phi: M(A)\to M(D)$, which is strictly continuous on the closed unit ball of $M(A)$,
	
	$(iii)$ $\phi$ maps some approximate identity of $A$ to an approximate identity of $D$,
	
	$(iv)$ the net $(\phi(e_i))$ converges strictly to $1_{M(D)}$, for some approximate identity of $A$.
\end{lemma}

It turns out that such a notion of nondegeneracy is too strong for coactions of groupoids on $C^*$-algebras. Here we propose a weaker notion. A subset $D_0$ of a $C^*$-algebra $D$ is called selfadjoint if $x\in D_0$ iff $x^*\in D_0$, for each $x\in D$.   

\begin{dfn}
	Let $\phi: A\to D$ be a $*$-homomorphism  between $C^*$-algebras. Let $\mathfrak{Z}$ be a $C^*$-subalgebra of $M(D)$ whose set of nonzero projections is nonempty. We say that $\phi$ is nondegenerate modulo $\mathfrak{Z}$, or simply nondegenerate (mod $\mathfrak{Z}$), if there is a selfadjoint total subset $D_0\subseteq D$ such that for each $x\in D_0$ there is a nonzero projection $p\in\mathfrak{Z}$ such that $px$ is in the closed linear span of elements $\phi(a)b$, with $a\in A, b\in D$.   
\end{dfn}

The above definition looks nonsymmetric, but in fact if $p$ is the nonzero projection such that $px^*$ is within $\varepsilon$ of the span of $\phi(A)D$, then so is $xp$, and vice versa, so the above definition has an equivalent right version. Also, note that when $\mathfrak{Z}=\mathbb C1_{M(D)}$, the notion of nondegeneracy (mod $\mathfrak{Z}$) is the same as the classical notion of nondegeneracy.  

When $D=A\otimes C^*(G)$, for a discrete groupoid $G$, there is a natural candidate for $\mathfrak{Z}$, namely $\mathfrak{Z}=C_0(\Gz)$, identified with its copy  $1_{M(A)}\otimes C_0(\Gz)$  inside $M(A\otimes C^*(G))$, and if we choose  $D_0$ to be the set of all basic tensors $a\otimes f$ with $a\in A$ and $f\in C_c(G)$, then we say that a $*$-homomorphism $\phi: A\to A\otimes C^*(G)$ is nondegenerate (mod $\Gz$) if for each $\varepsilon>0$, $a\in A$ and $g\in G$, there is a unit element  $u \in\Gz$ such that $a\otimes (1_{u}*1_{g})$ is within $\varepsilon$ of the linear span of $\phi(A)(A\otimes C^*(G))$. It is easy to see that being nondegenerate (mod $\Gz$) is the same as being nondegenerate (mod $C_0(\Gz)$), in the above sense. 

\begin{dfn}
A (full) coaction
of a discrete groupoid $G$ on a $C^*$-algebra $A$ is a  $*$-homomorphism
$\delta \colon A\to A\otimes C^*(G)$,  satisfying the identity $(\delta\otimes
\id)\circ \delta= (\id_A \otimes \delta_G)\circ
\delta$, which is nondegenerate (mod $\Gz$). 
\end{dfn}

When $G$ is a discrete group, $\Gz$ is a singleton consisting of the identity of $G$, and being nondegenerate (mod $\Gz$) is the same as being nondegenerate in the usual sense. 

For a coaction
$\delta$ of a discrete groupoid $G$ on a $C^*$-algebra $A$, the fixed-point
algebra is by definition, the cross-sectional $C^*$-algebra of the $\Gz$-bundle $A_u^\delta:=\{a\in
A\mid \delta(a)=a\otimes i_G(u)\}$, $u\in \Gz$.
There is a coaction of $G$ on $\Tc(X)$
whose fixed-point algebra is the core
$\Ff$.

\begin{lemma}\label{rep}
	Let $G$ be discrete  and $X$ be compactly aligned. The map $\psi$ defined by $\psi: X\to \Tc(X)\otimes C^*(G);\ x\mapsto i_X(x)\otimes i_G(d(x))$, is a Nica covariant  representation
of $X$.
\end{lemma}
\begin{proof}
The fact that $\psi$ is a representation follows from Lemma \ref{univ} and the fact that $\delta_p*\delta_q=\delta_{pq}$ when $s(p)=r(q)$, and zero otherwise, in $C^*(G)$. To see that $\psi$ is Nica covariant, first observe that,
\begin{align*}
\psi^{(p)}(x\otimes y^*)&=\psi_p(x)\psi_p(y)^*=\big(i_X(x)\otimes i_G(p)\big)\big(i_X(y)^*\otimes i_G(p^{-1})\big)\\&=i_X(x)i_X(y)^*\otimes i_G(r(p))=i_X^{(p)}(x\otimes y^*)\otimes i_G(r(p)),
\end{align*}
for each $p\in P, x,y\in X_p$. Thus, $$\psi^{(p)}(S)=i_X^{(p)}(S)\otimes i_G(r(p)),\ \ (p\in P, S\in\mathscr{K}_{A_{r(p)}}(X_p)),$$
Now, $\psi$ is Nica covariant, as $i_X$ is so, by the paragraph after Lemma \ref{univ}. 
\end{proof}

\begin{lemma}\label{prop:existence_coaction}
  Let $G$ be discrete  and  $X$ be compactly aligned. There is a coaction $\delta$ of $G$ on
  $\Tc(X)$, satisfying $\delta(i_X(x))=i_X(x)\otimes i_G(d(x))$, for  $x\in X$.
\end{lemma}
\begin{proof}
  Since the representation $\psi: X\to \Tc(X)\otimes C^*(G);\ x\mapsto i_X(x)\otimes i_G(d(x))$ is Nica covariant by Lemma \ref{rep}, by the universal property of $\Tc(X)$,  one gets a
  $*$-homomorphism $\delta:\Tc(X)\to \Tc(X)\otimes C^*(G)$, satisfying
  $\delta(i_X(x))=\psi(x)$, for $x\in X$. The coaction identity is immediate on generators, 
  and so holds everywhere. 
  
  To see that $\delta$ is nondegenerate (mod $\Gz$), let $(\theta_\lambda)$ be an approximate identity for $\Ff$, which is then also an approximate identity for $\Tc(X)$, and observe that for  $x,y\in X$ with $d(x)=d(y)$, $\theta_\lambda i_X(x)i_X(y)^*\to i_X(x)i_X(y)^*,$ and $ i_X(x)i_X(y)^*\theta_\lambda\to i_X(x)i_X(y)^*,$ as $\lambda\to\infty$. For $p:=d(x)\in P$, 
  $x=z\cdot\langle z,z\rangle_p=(z\otimes z^*)(z)$, for some $z\in X$, and 
  $i_X(x)i_X(y)^*=i_X^{(p)}(z\otimes z^*)i_X(z)i_X(y)^*$. Since $i_X^{(p)}(z\otimes z^*)$ is in $\Ff$,
   $\theta_\lambda i_X^{(p)}(z\otimes z^*)\to i_X^{(p)}(z\otimes z^*)$, that is, 
  $\theta_\lambda i_X(x)i_X(y)^*\to i_X(x)i_X(y)^*$, as claimed. The same holds, by a similar argument,  when $\theta_\lambda$ is multiplied from right. Next, observe that for basic tensors of the form 
  $i_X(x)i_X(y)^*\otimes i_G(g)$, for $x\in X_p, y\in X_q$ and $g
 \in G$, let  $u:=r(p)=r(q)$, then, $1_{u}*1_{g}=1_{g}$, when $r(g)=u$, and zero otherwise, and in the former case,
    \begin{align*} \delta\big(i_X(x)i_X(y)^*\big)(\theta_\lambda\otimes 1_{qp^{-1}g})&= (i_X(x)i_X(y)^*\otimes 1_{pq^{-1}})(\theta_\lambda\otimes 1_{qp^{-1}g})\\&= i_X(x)i_X(y)^*\theta_\lambda\otimes 1_{ug}
    	\\&=i_X(x)i_X(y)^*\theta_\lambda\otimes 1_{g}\\&\to i_X(x)i_X(y)^*\otimes 1_{g}
    	\\&= (1_{M(\Tc(X))}\otimes 1_u)(i_X(x)i_X(y)^*\otimes 1_{g}),
    \end{align*}
when $\lambda\to\infty$, as required.
\end{proof}

The above coaction $\delta$ of $G$ on  $\Tc(X)$ is called the gauge coaction.  Note that unlike the group case \cite[Proposition 3.5]{clsv}, $\delta$ is not injective in general. We would comment on the kernel of $\delta$ soon (see Lemmas \ref{ker}, \ref{ker2}, \ref{inj2}). 
Also, note that $\delta$ is not nondegenerate in general. However, we have some sort of nondegeneracy for the restriction of $\delta$ to the core (see Lemma \ref{nondeg}). This restricted version of nondegeneracy later appears in the definition of induced ideals (see Definition \ref{ind})).  

\begin{lemma} \label{core}
Let $G$ be discrete  and let $X$ be compactly aligned.

$(i)$ The core $\Ff$ is the same as
the cross-sectional $C^*$-algebra $\Tc(X)_0^\delta$ of the fixed-point $\Gz$-bundle with fibres $\Tc(X)_u^\delta = \{ a \in
\Tc(X) \mid \delta(a) = a \otimes i_G(u)\}$, $u\in \Gz$.

$(ii)$ $	\Ff = \clsp\{i^{(p)}_X(T) \mid p \in P \text{ and } T \in \Kaa(X_p)\}.$

\end{lemma} 
\begin{proof} 
$(i)$  Since $\Gz$ is discrete, $\Tc(X)_0^\delta$ is nothing but the direct sum of the fibres \cite[Example 3.5(i)]{kum}. 
By the above lemma, for $x,y\in X_p$, 
$$\delta(i_X(x)i_X(y)^*)=\delta(i_X(x))\delta(i_X(y))^*=i_X(x)i_X(y)^*\otimes i_G(s(p)),$$
thus, $\Ff\subseteq \Tc(X)_0^\delta$. The reverse inclusion follows similarly.

$(ii)$ This is immediate, since $i^{(p)}_X \colon \Kaa(X_p)\to \Tt_X$ satisfies $i^{(p)}_X(x \otimes
y)=i_X(x) i_X(y)^*$, for $x,y\in X_p$, and $\Kaa(X_p)$ is the closed linear span of the set of elements $x \otimes y$, with $x,y \in X_p$.
\end{proof}  

A
subset $F$ of $P$ is said to be closed if for $p,q \in
F$,  $p \vee q \in F$, whenever it exists. Let
$\fvcl{P}$ denote the set of finite closed subsets of
$P$, which is directed under inclusion, by an argument as in \cite[p. 367]{F99}. Also, each bounded $F\in \fvcl{P}$ has a maximal element.

For $p \in P$, let $B_p$ be 
the $C^*$-subalgebra $i_X^{(p)}(\Kaa(X_p)) $ of  $\Tc(X)$. For
 finite closed subset $F$ of $P$, let $
B_F := \textstyle{\sum}_{p \in F} B_p$, which is a linear subspace of $ \Tc(X)$, then $\bigcup_{F \in \fvcl{P}} B_{F}$ is dense in $\Ff$. Let us observe that when $X$  is
compactly aligned, $B_F$ is indeed a $C^*$-subalgebra of $\Ff$:
$B_F$ is closed under adjoint and multiplication, by the Nica covariance of the universal
representation $i_X$ of $X$ in $\Tc(X)$. To show that it is norm closed, we proceed by induction on $n:=|F|$. If $n = 1$, say $F = \{p\}$, then $B_F = B_p = i_X^{(p)}(\Kaa(X_p))$ is the
range of a $C^*$-homomorphism and hence closed. Assuming this for $|F| \le k$, take $F \subset P$ with $n=|F| = k+1$, and take a minimal element $m$, in
the sense that, for $p \in F\setminus\{m\}$, $p \not\le
m$. By the induction hypothesis, $B_m$
and $B_{F\setminus\{m\}}$ are $C^*$-subalgebras of $\mathcal{F}$.
Choose $p \in F$, $p\neq m$, then $p
\not\le m$ by minimality of $m$, and  $p \vee m\neq m$. Thus, for $S
\in \Kaa(X_p)$ and $T \in \mathscr{K}_{A_{r(m)}}(X_m)$, 
$
i_X^{(p)}(S) i_X^{(m)}(T)
 = i_X^{(p \vee m)}(\iota^{p \vee m}_p(S) \iota^{p \vee m}_m(T))$ is in $i_X^{(p \vee m)}(\mathscr{K}_{A_{r(p\vee m)}}(X_{p \vee m})) \subseteq B_{F \setminus \{m\}}$.
Similarly, $i_X^{(m)}(T) i_X^{(p)}(S) \in B_{F \setminus
\{m\}}$, therefore $B_F =
B_m + B_{F \setminus \{m\}}$ is norm closed
 \cite[Corollary~1.8.4]{Dix:C*-algebras}.

\begin{lemma} \label{lemma:inj}
Let $X$ be compactly aligned. Suppose that either  the left action on each
fibre is injective, or  $P$ is directed.
Let $\psi \colon X \to D$ be an injective Nica covariant
representation of $X$, and let $\psi_* : \Tc(X) \to D$
be homomorphism satisfying $\psi = \psi_* \circ i_X$.
Then $\ker(\psi_*)\cap\Ff\subset \ker (q_{\CNP})$.
\end{lemma}

\begin{proof}
By an argument similar to that of the proof of  
\cite[Lemma~1.3]{ALNR}, and above observations, it is enough to show that $\ker(\psi_*)
\cap B_F \subset \ker(q_{\CNP})$, for each $F \in \fvcl{P}$. For
a typical element $c := \sum_{p \in F}
i_X^{(p)}(T_p)$ of $B_F$, if $c
\in \ker(\psi_*)$, since $\psi$ is injective, by part $(iii)$ of Lemma~\ref{lem:psi equal}, $\sum_{p \in F} \iotale^s_p(T_p) = 0$, for large
$s$. Since $j_X$ is
CNP-covariant, $
\textstyle{q_{\CNP}(c) = \sum_{p \in F} j_X^{(p)}(T_p) = 0}$, that is, $c \in \ker(q_{\CNP})$.
\end{proof}

The above lemma gives a way to control the kernel of canonical coaction $\delta$ of $G$ on $\Tc(X)$ defined in Lemma \ref{prop:existence_coaction}.

\begin{lemma} \label{ker}
Let $X$ be compactly aligned. Suppose that either  the left action on each
fibre is injective, or  $P$ is directed. Then for the canonical coaction 
$$\delta: \Tc(X)\to \Tc(X)\otimes C^*(G); \ i_X(x)\mapsto i_X(x)\otimes i_G(d(x)), \ \ (x\in X),$$
we have $\ker(\delta)\cap \Ff\subset \ker (q_{\CNP})$, where $q_{\CNP}: \Tc(x)\to \NO{X}$ is the quotient map.  
\end{lemma} 
\begin{proof}
Let $\psi: X\to \Tc(X)$ be the Nica covariant representation defined in Lemma \ref{rep}. Then $\psi_{*}=\delta$, and the result follows from Lemma \ref{prop:existence_coaction}. 
\end{proof} 
We are now ready to show that $\NO{X}$ satisfies analog of criterion (B) of Sims and Yeend \cite[Section 1]{SY}, whenever the left
actions on fibres of $X$ are all injective, or $P$ is
directed. This in particular gives 
a gauge-invariant uniqueness theorem for $\NO{X}$, when $G$ is a discrete amenable groupoid.

\begin{lemma}\label{thm:inj on core}
Let $X$ be compactly aligned. Assume that either  the left actions on the
fibres of $X$ are all injective, or  $P$ is directed and
$X$ is $\tilde\phi$-injective. Let $\psi \colon X \to D$ be a
CNP-covariant representation of $X$ in a $C^*$-algebra $D$.
Then the induced homomorphism  $\intfrm{\psi} : \NO{X} \to D$
is injective on $q_{\CNP}(\Ff)$ if and only if $\psi$ is
injective as a Toeplitz representation.
\end{lemma}

\begin{proof}
Suppose that $\intfrm{\psi}$ is injective on $q_{\CNP}(\Ff)$.
By Proposition \ref{faithful2}, $j_X$ is injective, and   $\psi = \intfrm{\psi}\circ(j_X)$. 
Conversely, if $\psi$ is injective, by Lemma ~\ref{lemma:inj}, 
$$\ker(\psi_*)=\ker(\Pi\psi \circ q_{\CNP}) \cap \Ff \subset \ker(q_{\CNP}),$$ that is, 
$\intfrm{\psi}$ is injective on $q_{\CNP}(\Ff)$. 
\end{proof}

\begin{lemma}\label{rep2}
	Let $G$ be discrete  and $X$ be compactly aligned. For the canonical Nica covariant  representation $$\psi: X\to \Tc(X)\otimes C^*(G);\ x\mapsto i_X(x)\otimes i_G(d(x)),$$ 
	the induced representation $$\tilde\psi:=(q_{CNP}\otimes {\rm id})\circ\psi: X\to \NO{X}\otimes C^*(G);\ x\mapsto j_X(x)\otimes i_G(d(x)),$$
	is a CNP-covariant  representation
	of $X$. When each $\tilde\phi_p$ is injective, or each bounded subset of $P$ has a maximal element, $\tilde\psi$ is also injective as a representation. 
\end{lemma}
\begin{proof}
	First observe that,
	\begin{align*}
		\tilde\psi^{(p)}(x\otimes y^*)&=\tilde\psi_p(x)\tilde\psi_p(y)^*=\big(j_X(x)\otimes i_G(p)\big)\big(j_X(y)^*\otimes i_G(p^{-1})\big)\\&=j_X(x)j_X(y)^*\otimes i_G(r(p))=j_X^{(p)}(x\otimes y^*)\otimes i_G(r(p)),
	\end{align*}
	for each $p\in P, x,y\in X_p$, thus, $$\psi^{(p)}(S)=j_X^{(p)}(S)\otimes i_G(r(p)),\ \ (p\in P, S\in\mathscr{K}_{A_{r(p)}}(X_p)),$$
	Now, $\psi$ is CNP-covariant, as $j_X$ is so by the definition. 
	
	Next, assume that each $\tilde\phi_p$ is injective, or each bounded subset of $P$ has a maximal element. Given $u\in \Gz$, let $1_u$ be the characteristic function of singleton $\{u\}$, then for each $n\geq 1$, scalars $c_1,\cdots c_n$, and elements $g_1,\cdots, g_n\in G$ with the same range, 
	$$\sum_{i,j=1}^{n}c_i\bar c_j 1_u(g^{-1}_ig_j)=\sum_{(i,j)\in I}c_i\bar c_j=\sum_{i=1}^n |c_i|^2+2{\rm Re}\big(\sum_{(i,j)\in J}c_i\bar c_j\big) \geq 0,$$ 
	where $I:=\{(i,j): g_i=g_j, s(g_i)=u\}$ and $J:=\{(i,j)\in I: i<j\}$; that is, $1_u$ is positive definite on $G$ \cite[Proposition 3]{pat}, and so $1_u: C_c(G)\to \mathbb C$ extends to a bounded linear functional on $C^*(G)$ \cite[Proposition 4]{pat}, mapping $i_G(u)$ to 1. Now for $a\in A_u$, if
	$\tilde\psi(a)=j_X(a)\otimes i_G(u)=0$, then,
	$0=({\rm id}\otimes 1_u)\tilde\psi(a)=j_X(a),$ which in turn implies that $a=0$, by Proposition \ref{faithful2}.  
\end{proof}

\begin{lemma} \label{ker2}
	Let $X$ be compactly aligned. Suppose that either  the left action on each
	fibre is injective, or  $P$ is directed and $X$ is $\tilde\phi$-injective. Then for the canonical coaction 
	$$\delta: \Tc(X)\to \Tc(X)\otimes C^*(G); \ i_X(x)\mapsto i_X(x)\otimes i_G(d(x)), \ \ (x\in X),$$
	induces a coaction of $G$ on $\NO{X}$, again denoted by $\delta$, given by,
	$$\delta: \NO{X}\to \NO{X}\otimes \ C^*(G); \ j_X(x)\mapsto j_X(x)\otimes i_G(d(x)), \ \ (x\in X),$$
	which  
	is injective on $q_{\CNP}(\Ff)$.  
\end{lemma} 
\begin{proof}
	Let $\tilde\psi: X\to \NO{X}$ be the CNP-covariant representation defined in Lemma \ref{rep2}. Then $\Pi\tilde\psi=\delta$ on $\NO{X}$, and the result follows from Lemmas   \ref{thm:inj on core} and \ref{rep2}.  
\end{proof}

Karla Oty introduced the notion of  groupoids with uniformly bounded fibres in her study of the Fourier-Stieltjes algebras of groupoids \cite{o}, which are those with a common upper bound on the cardinality of the sets $G^u$ of elements with a given range $u$. We give examples of such groupoids in Section \ref{example}. In the proof of the next lemma, we use notations of \cite{pat}.

\begin{lemma}\label{at last}
	If $G$ has uniformly bounded fibres with uniform bound $N$, then for each  finite subset $F\subseteq G$, there is a linear functional $1_F$ on $C^*(G)$, of norm at most $N$, such that $1_F(i_G(g))=1$, for each $g\in F$.
\end{lemma}
\begin{proof}
	For $u\in \Gz$,  $1_u$ is a positive definite function on $G$ by the proof of Lemma \ref{rep2}. For a finite subset $F\subseteq  G$, by \cite[Theorem 1]{pat}, the positive definite function $1_{s(F)}=\sum_{u\in s(F)} 1_u$,    the  coefficient function $(\xi,\xi)$ of some  representation $L$ of $G$ on a $G$-Hilbert bundle $E$. Moreover, $\|\xi\|^2=\|1_{s(F)}\|_\infty=1$, by \cite[Proposition 4]{pat}. Next,  it follows from \cite[Proposition 10]{pat} that the left convolution by $1_F$ extends to a bounded linear operator on $E^2$ of norm at most $\|1_F\|_{I,r}=\sup_{u\in \Gz}\sum_{g\in G^u}|1_F(g)|=\sup_u |G^u|\leq N$. For $g\in F$,  we have $1_g*1_u=1_g$, if $u=s(g)$, and zero otherwise, thus  $1_F=1_F*1_{s(F)}$, which is then the coefficient function $(\xi, 1_F*\xi)$ of $L$. Now $L$ integrates to a representation  of $C^*(G)$, again denoted by $L$, and we have,
	$$\langle 1_F,a\rangle=\langle L(a)\xi, 1_F*\xi\rangle,\ \ (a\in C^*(G)),$$  
	that is, $1_F$ could be regarded as a bounded linear functional on $C^*(G)$ of norm at most $\|\xi\|\|1_F*\xi\|\leq N$.  Moreover,
	$$\langle 1_F,f\rangle=\sum_{g\in G} f(g)1_F(g)=\sum_{g\in F}f(g),$$
	for $f\in C_c(G)$. In particular, $1_F(i_G(g))=\langle 1_F, 1_g\rangle=1$, for each $g\in F$, as required.
\end{proof}

\begin{lemma} \label{inj2}
	If $G$ has uniformly bounded fibres, the canonical coactions of $G$ on $\Tc(X)$ and $\NO{X}$ are injective. 
\end{lemma}
\begin{proof}
	Let $\delta$ be the canonical coaction of $G$ on $\Tc(X)$, defined by 
	$$\delta(i_X(x)i_X(y)^*)=i_X(x)i_X(y)^*\otimes i_G(pq^{-1}), \ \ \big(x,y\in X, r(d(x))=r(d(y))\big).$$
	Let $a\in \ker(\delta)$, and given $\varepsilon>0$, choose a finite sum $b=\sum_{i=1}^{n} i_X(x_i)i_X(y_i)^*\in \Tc(X)$, 
	within $\varepsilon$ of $a$. Let  $x_i\in X_{p_i}, y_i\in X_{q_i}$, for $p_i, q_i\in P$, and put $g_i:=p_iq_i^{-1}$, for $1\leq i\leq n$. Since $\delta$ is norm decreasing, $\delta(b)$ is of norm at most $\varepsilon$. Put $F:=\{g_i; 1\leq i\leq n\}$ and let $1_F$ be the bounded functional on $C^*(G)$ constructed in Lemma \ref{at last}, then  
	\begin{align*}
		(\id\otimes 1_F)\delta(b)&=(\id\otimes 1_F)\sum_{i} i_X(x_i)i_X(y_i)^*\otimes i_G(g_i)\\&=\sum_{i} i_X(x_i)i_X(y_i)^*=b, 
	\end{align*}
	is of norm at most $N\varepsilon$. Therefore, $a$ is of norm at most $(N+1)\varepsilon$. Since $\varepsilon$ was arbitrary, it follows that $a=0$, as required. The argument for the coaction of $G$ on $\NO{X}$ is similar, with $i_X$ replaced by $j_X$.
\end{proof}

Finally let us prove a restricted version of nondegeneracy for the canonical coactions of $G$ on $\Tc(X)$ and $\NO{X}$, as promised. 

\begin{lemma} \label{nondeg}
	If $G$ has uniformly bounded fibres, the restriction of the gauge coaction $\delta$ to the core $\Ff$ is a  coaction of $\Gz$ on $\Ff$. Moreover, the restricted $*$-homomorphism  $\delta: \Ff\to(\Ff\otimes C^*(G))\cap \delta(\Tc(X))$ is surjective. Similarly, the restriction of the canonical coaction of $G$ on $\NO{X}$ (again denoted by $\delta$) to $q_{\rm CNP}(\Ff)$ is a  coaction of $\Gz$ on $q_{\rm CNP}(\Ff)$ and $\delta: q_{\rm CNP}(\Ff)\to\big(q_{\rm CNP}(\Ff)\otimes C^*(G)\big)\cap \delta(\Tc(X))$ is a $*$-epimomorphism.
\end{lemma}
\begin{proof}
	It is clear that $\delta(\Ff)\subseteq \Ff\otimes C^*(G)$, so $\delta$ restricts to a coaction of $G$ on $\Ff$. Put $D:= (\Ff\otimes C^*(G))\cap \delta(\Tc(X))$. To prove surjectivity of the restricted map $\delta: \Ff\to D$, it is enough to show that the set of basic tensors $\delta(i_X(x)i_X(y)^*)=i_X(x)i_X(y)^*\otimes i_G(r(p))$, with $p\in P$ and $x,y\in X_p$, is a total set in $D$. Given $\varepsilon>0$ and $b\in D$, choose a finite sum $\sum_i i_X(x_i)i_X(y_i)^*\otimes i_G(g_i)\in \Ff\odot C^*(G)$, with $p_i\in P$ and $x_i,y_i\in X_{p_i}$, within $\varepsilon$ of $b$. Choose $a\in \Tc(X)$ with $b=\delta(a)$. Finally, choose a finite sum $c:=\sum_j i_X(z_j)i_X(w_j)^*$, with $r_j, q_j\in P$, $r(r_j)=r(q_j)$, and $z_j\in X_{r_j},w_j\in X_{q_j}$,  within $\varepsilon$ of $a$. Since $\delta$ is norm decreasing, $b=\delta(a)$ is within $\varepsilon$ of $\delta(c)=\sum_j i_X(z_j)i_X(w_j)^*\otimes i_G(r_jq^{-1}_j)$. Thus the following finite sum is of norm at most $2\varepsilon$,
	$$d:= \sum_i i_X(x_i)i_X(y_i)^*\otimes i_G(g_i)-\sum_j i_X(z_j)i_X(w_j)^*\otimes i_G(r_jq^{-1}_j).$$   
	For the finite set $F\subseteq G$ including  elements $g_i$ and $r_jq^{-1}_j$, for each $i$ and $j$, let $1_F$ be the bounded linear functional of norm at most $N$ constructed in the proof of Lemma \ref{at last}. Then, $$\langle\id\otimes 1_F,b \rangle=\sum_i i_X(x_i)i_X(y_i)-\sum_j  i_X(z_j)i_X(w_j)^*$$
	is of norm at most $2N\varepsilon$. Applying $\delta$ to the right hand side, the following finite sum is of norm at most $2N\varepsilon$,
	$$ \sum_i i_X(x_i)i_X(y_i)^*\otimes i_G(r(p_i))-\sum_j i_X(z_j)i_X(w_j)^*\otimes i_G(r_jq^{-1}_j).$$ 
	Therefore, $b$ is within $(2N+1)\varepsilon$ of  $\sum_i i_X(x_i)i_X(y_i)^*\otimes i_G(r(p_i))$, proving the claim. The proof for $\NO{X}$ is similar, with $i_X$ replaced by $j_X$ in the above argument.
\end{proof}

\section{gauge-invariant uniqueness theorem}\label{sec:couniversal}

In this section we give a gauge-invariant uniqueness theorem for product systems over quasi-lattice ordered groupoids. A Toeplitz representation
$\psi \colon X \to D$ is gauge-compatible if there is a coaction
$\beta$ of $G$ on $D$ such that
$\beta(\psi(x)) = \psi(x) \otimes i_G(d(x))$, for $x \in X$.
For gauge-compatible Toeplitz representations $\psi_i:X\to D_i$, with 
coaction $\beta_i$ of $G$ on $D_i$ satisfying
$\beta_i(\psi_i(x))=\psi_i(x)\otimes i_G(d(x))$, for  $x\in
X$ and $i=1,2$, an $*$-homomorphism  $\phi :D_1\to D_2$ is equivariant if $\phi\circ \psi_1=\psi_2$ and 
$(\phi\otimes \id)\circ \beta_1=\beta_2\circ \phi$.

The main result of this section is an extension of \cite[Theorem 4.1]{clsv}.

\begin{theorem}\label{thm:projective property}
Let $G$ be discrete and $X$ be compactly aligned. Suppose that either the left action on each
fibre is injective, or $P$ is directed and $X$
is $\tilde\phi$-injective. Then there exists a triple
$(\NO{X}^\reduced, j_X^\reduced, \delta^n)$ which is
couniversal for gauge-compatible injective Nica covariant
representations of $X$ in the following sense:
\begin{enumerate}
\item $\NO{X}^\reduced$ is a $C^*$-algebra, $j^\reduced_X$ is an injective
    Nica covariant representation of $X$ whose image
    generates $\NO{X}^\reduced$, and $\delta^n$ is a normal coaction of $G$ on $\NO{X}^\reduced$
    with $\delta^n(j^\reduced_X(x))=j^\reduced_X(x)\otimes i_G(d(x))$, for  $x\in X$.
\item\label{it:c-u property} If $\psi \colon X \to D$ is an
    injective Nica covariant gauge-compatible representation whose image
    generates $D$ then there is a
    surjective $*$-homomorphism $\phi \colon D \to
    \NO{X}^\reduced$ such that $\phi(\psi(x)) =
    j^\reduced_X(x)$, for $x \in X$.
\end{enumerate}
Moreover, the representation $j^\reduced_X$ is CNP-covariant, the coaction $\delta^n$ is normal, and
$(\NO{X}^\reduced, j^\reduced_X, \redgaug)$ is the unique
triple satisfying \textnormal{(1)} and \textnormal{(2)}, in the sense that, if $(C, \rho, \gamma)$ satisfies the same
two conditions, then there is an isomorphism $\phi \colon C\to
\NO{X}^\reduced$ such that $j^\reduced_X = \phi\circ \rho$ and $\phi$ is equivariant for $\gamma$ and $\redgaug$.
\end{theorem}

To prove Theorem~\ref{thm:projective property}, we essentially adapt  argument of the proof of \cite[Theorem 4.1]{clsv} to the groupoid case. We first need
to recall some facts about Fell bundles over groupoids and their cross-sectional 
$C^*$-algebras \cite{kum}. Then we need to slightly extend the known results on the relationship
between topologically graded $C^*$-algebras and $C^*$-algebras
associated to Fell bundles as in \cite{Exel}, to the case where both gradings and bundles are over groupoids. Finally we need to extend the
connection between discrete coactions and Fell bundles, as 
explored in \cite{qui:discrete coactions}, to this more general setting. 

In \cite[Definition
3.5]{qui:discrete coactions}, Quigg introduced the notion of the reduced
$C^*$-algebra of a Fell bundle over a discrete group and its coaction. Note that the constructions in \cite{Exel}
and \cite{qui:discrete coactions} are compatible \cite[page 749]{EQ}. We need to extend some of these results to our setting. 

Let us recall the notion of a Fell bundle over an \'{e}tale groupoid $G$ \cite{kum}, \cite{yam}. Let $p: E\to G$ be a Banach bundle, and set 
$$E^2 = \{(e_1, e_2) \in E \times E \mid  (p(e_1), p(e_2))\in\Gt \}.$$
A multiplication on $E$ is a continuous map: $E^2\to E;  (e_1, e_2)\mapsto
e_1e_2$, satisfying:

$(i)$  $p(e_1e_2) =p(e_1)p(e_2)$, 

$(ii)$ the induced map: $E_g \times E_h\to E_{gh}$ is bilinear, 

$(iii)$ $(e_1e_2)e_3 = e_1(e_2e_3)$, 

$(iv)$ $\|e_1e_2\|\leq \|e_1\|\|e_2\|,$

\noindent for $e_1,e_2,e_3\in E$ and $g,h\in G$, whenever the multiplications in $E$ or $G$ are defined.
 
\noindent An involution on $E$ is a continuous map: $E\to E; e\mapsto e^*$,  satisfying:

$(v)$ $p(e^*) = p(e)^{-1}$,

$(vi)$  the induced map: $ E_g\to E_{g^{-1}}$ is conjugate linear,

$(vii)$ $e^{**} = e,$

\noindent for $e\in E, g\in G$. 
The bundle $E$ together with these structure maps is called a Fell bundle over $G$ if in addition the following conditions hold:

$(viii)$ $(e_1e_2)* = e^*_1e^*_2$, 

$(ix)$ $\|e^*e\|= \|e\|^2$,

$(x)$ $e^*e \geq 0$, 

\noindent for  $e,e_1,e_2\in E$, whenever the multiplications in the first equality are defined. We say that $E$ is nondegenerate if all fibres $E_g$ are nonzero, and saturated if $E_gE_h$ is total in $E_{gh}$, for $(g,h)\in \Gt$. Let $E^0$ be the restriction of the bundle $E$ to $\Gz$.  

Next let us define the full and reduced cross-sectional $C^*$-algebra of $E$ \cite{yam}, \cite{kum}. Regarding $E$ as a Banach bundle over $G$, let $C_c(E)$ and $C_0(E)$ be respectively the normed and Banach spaces consisting of continuous sections which are respectively compactly supported and vanishing at infinity.   
 Given $h,h_1,h_2\in C_c(E)$, define the convolution
and involution by,
$$h_1*h_2(g) = \sum_{g=g_1g_2} h_1(g_1)h_2(g_2), \ \ h^*(g) = h(g^{-1})^*,\ \ (g\in G),$$
where the sum runs over pairs for which the multiplication is defined. Then $C_c(E)$ is a $*$-algebra. Note that each fibre $E_g$ is a
Hilbert $E_{r(g)}$-$E_{s(g)}$-bimodule with inner products $\langle e_1,e_2\rangle_{E_{s(g)}}:= e^*_1e_2$ and ${}_{E_{r(g)}}\langle e_1,e_2\rangle:= e_1e^*_2$. Let $P: C_c(E)\to C_c(E^0)$ be
the restriction map, and consider the  $C_c(E^0)$-valued inner product on $C_c(E)$ defined by 
$$\langle h_1, h_2\rangle:=P(h_1*h_2),\ \ (h_1,h_2\in C_c(E^0)),$$
turning it to a pre-Hilbert $C_0(E^0)$-module \cite[Proposition 3.2]{kum}.

Let $L^2(E)$ be the  completion of $C_c(E)$ in
the norm coming from the above $C_c(E^0)$-valued inner product. Then $L^2(E)$ is the canonical Hilbert module associated to the $\Gz$-bundle of Hilbert
modules \cite[1.7]{kum} with fibres $V_u:=\bigoplus_{s(g)=u} E_g$. When moreover $E$ is saturated, $V_{r(g)}\otimes_{E_{r(g)}}E_g\simeq V_{s(g)}$, for $g\in G$ \cite[3.3]{kum}. The left multiplication by elements of $C_c(E)$ is adjointable and induces a $*$-monomorphism
$\lambda_E: C_c(E)\hookrightarrow \mathscr{L}_{C_0(E^0)}(L^2(E))$. The reduced cross-sectional $C^*$-algebra $C^*_r(E)$ of $E$ is the closure of the copy of $C_c(E)$ inside $\mathscr{L}_{C_0(E^0)}(L^2(E))$ under the above monomorphism. In particular, $C^*_r(E^0)=C_0(E^0)$. For 
$u\in\Gz$, we have a representation $\pi_u: C^*_r(E)\to \mathscr{L}_{E_u}(V_u)=:\mathscr{L}_u$, and the reduced norm is given by $\|a\|_r= \sup_{u\in\Gz}\|\pi_u(a)\|_{\mathscr{L}_u}$. Each $h\in C_b(\Gz)$ is identified with an element $M_g$ of the multiplier algebra $M(C^*_r(E))$, satisfying $M_hf(g)=h(r(g))f(g)$, for $g\in G, f\in C_c(E)$. The restriction map $P: C_c(E)\to C_c(E^0)$, extends to a faithful conditional expectation $\Phi_r: C^*_r(E)\to C_0(E^0)$ \cite[Propositions 3.6, 3.10]{kum}. 
For the full cross-sectional $C^*$-algebra $C^*(E)$, we simply complete $C_c(E)$ in the maximal $C^*$-norm coming from bounded representations of this $*$-algebra \cite{yam}. The above monomorphism extends to a $*$-epimorphism $\lambda_E: C^*(E)\twoheadrightarrow C^*_r(E)$, called the regular representation. 

Following \cite{qui:discrete coactions}, let us next define the notion of a representation for Fell bundles over groupoids. Let $E$ and $F$ be  Fell bundles over \'{e}tale groupoids $G$ and $H$. Let $M(F)$ be the multiplier bundle of $F$  \cite[VIII.2.14]{Fell} with fibres $M_h(F)$, $h\in H$. We identify $M_v(F)$ with $M(F_v)$, for $v\in F^{(0)}$. A map $\phi: E \to M(F)$ is called a morphism if there is a (not necessarily unique) continuous groupoid homomorphism
$\eta: G\to H$ such that,

$(i)$ $\phi|_{E_g}: E_g\to M_{\eta(g)}(F)$ is linear,

$(ii)$ $\phi(ab)=\phi(a)\phi(b), \phi(a^*) =\phi(a)^*$, 

$(iii)$  $\phi|_{E_u}: E_u\to M(F_{\eta(u)})$ is nondegenerate,

\noindent for $g\in G$, $u\in \Gz$, and $a,b\in E,$ with $(a,b)\in E^2$.
When $E=F$ and $\eta={\id}_G$, this is just a $G$-bundle morphism. For a $C^*$-algebra $D$, a
	representation of $E$ in $M(D)$ is a morphism $\phi: E \to M(D)$, with $D$ regarded as a trivial $G$-bundle. A representation $\phi: E\to  M(D)$ lifts uniquely to a nondegenerate $*$-homomorphism $\phi: C^*(E) \to M(D)$. When $\phi(E)\subseteq D$, we call $\phi$ a representation of $E$ in $D$. 

\begin{lemma} \label{co} Let $G$ be a discrete groupoid.

$(i)$ If $G$ coact by $\delta$  on
a  $C^*$-algebra $A$, and 
$A^\delta_g:=\{a\in A \mid \delta(a)=a\otimes i_G(g) \}$ is the
spectral subspace of $A$ at $g\in G$. The Banach bundle with fibre $A^\delta_g\times\{g\}$ at $g$ is a Fell bundle over $G$.

$(ii)$ If $E$ is a Fell bundle over $G$, there is
a canonical coaction $\delta_E$ on the full
cross-sectional algebra $C^*(E)$ such that
$\delta_{E}(a)=a\otimes i_G(g)$, for $g\in G$, $a\in E_g$.
\end{lemma}
\begin{proof}
$(i)$ Let $E_A$ be the disjoint union of the fibres $A^\delta_g\times\{g\}$ (identified with $A^\delta_g$) and $p: E_A\to G$ be the projection on the second leg. Then $E^2_A$ consists of all pairs $(a,b)$ with $a\in A^\delta_g$ and $b\in A^\delta_h$, for some $g,h\in G$ with $r(h)=s(g)$. All the required conditions listed above now follow from the fact that $A$ is a $C^*$-algebra.

$(ii)$ The morphism $\delta_E: E\to C^*(E)\otimes C^*(G)$ lifts to a  representation $\delta_E: C^*(E)\to C^*(E)\otimes C^*(G)$. The coaction
identity follows as it is trivial on elements of $E$. Let us observe that $\delta_E$ is nondegenerate (mod $\Gz$). For $g,h\in G$ and a basic tensor $a\otimes i_G(h)$ with $a\in E_g$, consider the projection $p:=1_{r(g)}$ for finite subset $F\subseteq \Gz$, then $(1_{M(C^*(E))}\otimes p)(a\otimes i_G(h))$ is nonzero when $r(h)=r(g)$, in which case, for an approximate identity $(e_i)$ of $C^*(G)$, 
\begin{align*}
\delta_E(a)(e_i\otimes i_G(g^{-1}h))&=(a\otimes i_G(g))(e_i\otimes i_G(g^{-1}h))=ae_i\otimes i_G(r(g)h)\\&=ae_i\otimes i_G(r(h)h)=ae_i\otimes i_G(h)\to a\otimes i_G(h),
\end{align*}	
as $i\to\infty$, as required.
\end{proof} 

If $E$ is a Fell bundle over $G$, a
cross-sectional algebra of $E$ is the $C^*$-completion $C^*_\gamma(E)$ of the algebra $C_c(E)$ with respect to a $C^*$-norm $\gamma$. The bundle $E$ is
topologically graded if there exists a contractive conditional
expectation from $E$ to $E_u$ vanishing on each $E_g$, for 
$g \in G\backslash\Gz$ (compare to  \cite[Definition 3.4]{Exel}).
For topologically graded bundles, the reduced cross-sectional algebra $C^*_r(E)$ is minimal among all 
cross-sectional algebras $C^*_\gamma(E)$. Indeed, by an argument as in 
\cite[Theorem 3.3]{Exel}, there exists a surjective
homomorphism $\lambda_\gamma:C^*_\gamma(E)\to C^*_r(E)$ such that
$\lambda_{\gamma}\circ\iota_\gamma=\iota_{\rm min}$,
where $\iota_\gamma$ and $\iota_r$ are the
embeddings of $C_c(E)$ into $C^*_\gamma(E)$ and $C^*_r(E)$, respectively. On the other hand, since each representation of $C_c(E)$ extends to a representation of $C^*(E)$, the inclusion $\iota_\gamma$ extends to a $*$-epimorphism
$
\rho_{\gamma} \colon C^*(E)\to C^*_\gamma(E)
$
such that $\rho_\gamma\circ \iota_{\rm max}=\iota_\gamma$
where $\gamma_\mathcal{A}$ is the embedding of the algebra of finitely supported
sections on $\mathcal{A}$ into $C^*(\mathcal{A})$.
 In particular, $\lambda_{\rm max}=\rho_{\rm min}=\lambda_E$ is the regular representation. 

Let $\delta$ be a coaction of an \'{etale} groupoid $G$ on a $C^*$-algebra $A$, and $B$ be a unital $C^*$-algebra. Each nondegenerate $*$-homomorphism $\sigma:  C_0(G) \to B$, with $\sigma|_{C_0(\Gz)}\equiv
 1_B$, induces a coaction $\delta^\sigma$ of $G$ on $B$, defined by 
$$\delta^\sigma: B\to B\otimes C^*(G);\ \ b\mapsto {\rm Ad}_{(\sigma\otimes{\rm id})(i_G)}(b\otimes 1),$$
where $i_G$ is the element in $M(C_0(G)\otimes C^*(G))$ associated to the function $g\mapsto i_G(g)$ in $C_b(G, C^*(G))$. The element $u:=(\sigma\otimes{\rm id})(i_G)$ is a unitary element in $M(B\otimes C^*(G))$, since $i_G(g)i^*_G(g)=i_G(s(g))$ and $i^*_G(g)i_G(g)=i_G(r(g))$, that is, $i_Gi^*_G, i^*_Gi_G\in C_b(G, C_0(\Gz))$, and so $(\sigma\otimes{\rm id})(i_Gi^*_G)=(\sigma\otimes{\rm id})(i^*_Gi_G)=1$ in $M(B\otimes C^*(G))$. In particular, $\delta^\sigma$ is a homomorphism. The coaction identity is immediate, and nondegeneracy follows by an argument verbatim to that of \cite[Lemma 1.11]{qui:full reduced}.
         
For a $C^*$-algebra $D$, a covariant representation
of $(A, G, \delta)$ in $M(B)$ is a pair $(\pi, \sigma)$ of nondegenerate representations $\pi: A\to M(D)$ and $\sigma:  C_0(G) \to M(D)$, with $\sigma|_{C_0(\Gz)}\equiv
1_{M(D)}$, such that $(\pi\otimes{\rm id})\circ\delta = \delta^\sigma\circ\pi$, that is $\pi: A \to M(D)$ 
is equivariant for the coactions $\delta$  and $\delta^\sigma$. 

\begin{lemma} \label{tg}
Let  $\delta$ be a coaction of a discrete groupoid $G$ on $A$ and $E_A$
be the corresponding Fell bundle, then,

$(i)$ $E_A$ is  topologically graded,

$(ii)$ the full cross-sectional $C^*$-algebra of
$E_A$ is isomorphic to $A$.
\end{lemma}
\begin{proof}
$(i)$ Given $u\in\Gz$, we need to build a contractive conditional expectation onto $(E_A)_u$ vanishing on fibres over $G\backslash\Gz$. Let $1_u$ be the characteristic function of singleton $\{u\}$, and extend  $1_u: C_c(G)\to \mathbb C$  to a bounded linear functional on $C^*(G)$ as in Lemma \ref{rep2}. Let $B(G)$ consist of coefficient functions of representations of $G$ on Hilbert bundles, and for $f\in B(G)$, put $\delta_f:=({\rm id}\otimes f)\circ\delta$, then we claim that $\delta_u:=\delta_{1_u}: A\to A$ is the desired conditional expectation. Indeed, we only need to check that $\delta_u(A)=A^\delta_u$. Let $a\in A$ and put $b:=\delta_u(a)$, then,
\begin{align*}
({\rm id}\otimes f)[\delta(b)-b\otimes i_G(u)]&=({\rm id}\otimes f)\delta\big(({\rm id}\otimes 1_u)\delta(a)\big)-f(u)b\\&=({\rm id}\otimes{\rm id}\otimes f)({\rm id}\otimes{\rm id}\otimes 1_u)(\delta\otimes {\rm id})\delta(a)-f(u)b\\&=({\rm id}\otimes{\rm id}\otimes f1_u)({\rm id}\otimes\delta_G)\delta(a)-f(u)b\\&=f(u)({\rm id}\otimes 1_u)\delta(a)-f(u)b=0,
\end{align*} 
and since the coefficient functions of the representations separate the points of $G$, we conclude that $\delta(b)=b\otimes i_G(u)$, i.e., $\delta_u(A)\subseteq A^\delta_u$. The reverse inclusion is immediate. 

$(ii)$	Let $g\in G$, then $1_{g}*1_{s(g)}=1_{g}$, thus $1_{g}\in B(G)$ by \cite[Proposition 10]{pat}. An argument as in $(i)$, shows that $\delta_g(A):=\delta_{1_g}(A)=A^\delta_g$. Now the result  follows from universality of cross sectional $C^*$-algebras of the corresponding Fell bundles $(A^\delta_g\times\{g\})_{g\in G}$ and $(\delta_g(A)\times\{g\})_{g\in G}$.   
\end{proof}

\begin{dfn}\label{ind}
Let $\delta$ be a coaction of a discrete groupoid $G$ on a $C^*$-algebra $A$. Let $I$ be a closed ideal of $A$. 
We say that $I$ is an induced ideal (with respect to $\delta$) if $\bigcup_{u\in \Gz}A(I \cap A^\delta_u)A$ is a total set in $I$.
If moreover, the restricted map $$\delta: I\to (I\otimes C^*(G))\cap \delta(A)$$ is nondegenerate as a $*$-homomorphism between $C^*$-algebras, we say that $I$ is a strongly induced ideal (with respect to $\delta$).
\end{dfn}

Two comments on the above definition is in order: First, if $I$ is an induced ideal, then $\delta(I)\subseteq (I\otimes C^*(G))\cap \delta(A)$, therefore $\delta$ restricts to a coaction of $G$ on $I$ and the restricted homomorphism in the second part of the above definition is well-defined. Second, since the group coactions are assumed to be nondegenerate, the induced ideals are automatically strongly induced in the group case.    
 
\begin{lemma}\label{prp:deltaI existence}
For a $C^*$-algebra $A$ with  a coaction   $\delta$ of a discrete groupoid $G$,
		
	$(i)$ the set $\sum_{g\in G} A^\delta_g$ consisting of finite sums of elements in fibres is dense in $A$ and elements in distinct fibres are linearly independent,   
	
	$(ii)$ if $I$ is an induced ideal of $A$ then $A/I$ is topologically graded,
	
	$(iii)$ if $I$ is an induced ideal of $A$ then there is a coaction
	$\delta^I$ of $G$ on $A/I$ such that
	$
	\delta^I \circ q = (q \otimes \id) \circ \delta$, where $q:A\to A/I$ is the quotient map.

\end{lemma}
\begin{proof}
	$(i)$ In the notations of the proof of Lemma \ref{tg}, since $B(G)$ separates the points of $G$ and $C_c(G)\cap B(G)$ is dense in $B(G)$, the set $\{\delta_f(a): a\in A, f\in C_c(G)\cap B(G)\}$ is a total set in $A$, and so is the set $\bigcup_{g \in G} A^\delta_g=\{\delta_{1_g}(a): g\in G, a\in A\}$. Also, if $a=\sum_{g\in F} a_g$ is a finite sum, then for $u\in r(F)$, put $F_u:=\{g\in F: r(g)=u\}$ and observe that,
	$\delta_u(aa^*)=\sum_{g\in F_u} a_ga^*_g$. If $a=0$, then for each $u\in r(F)$, $a_g^*a_g=0$ in the $C^*$-subalgebra $\delta_u(A)=A^\delta_u$, hence $a_g=0$, for each $g\in F$, establishing the linear independence.

	$(ii)$ Let $\delta_g:=\delta_{1_g}$ be defined as in the proof of lemma \ref{tg}. We claim that $I$ is fixed by each $\delta_g$. Given $u\in\Gz$ and a finite sum $a=\sum_{i} b_{g_i}a_ic_{h_i}$ in $A(I \cap A^\delta_u)A$, with $b_{g_i}\in A^\delta_{g_i}, a_i\in A^\delta_u, c_{h_i}\in I\cap A^\delta_{h_i}$, with $r(g_i)=s(h_i)=u$, for a given $g\in G$ we have,
	\begin{align*}
	\delta_g(a)&=({\rm id}\otimes 1_g)\delta(a)=({\rm id}\otimes 1_g)\sum_i\delta(b_{g_i})\delta(a_i)\delta(c_{h_i})\\&=({\rm id}\otimes 1_g)\sum_i b_{g_i}a_ic_{h_i}\otimes i_G(g_ih_i)=\sum_{g_ih_i=g}b_{g_i}a_ic_{h_i},
	\end{align*}
	and the last term is in $I$, as $I$ is an ideal. This proves the claim, since the set of all such finite sums are dense in $I$ by the assumption of $I$ being a reduced. Let $q: A\to A/I$ be the quotient map. We show that $q(A^\delta_g)$ is closed in $A/I$. For $a\in I$ and $b\in A^\delta_g$,
	$$\|a+b\|\geq\|\delta_g(a+b)\|=\|\delta_g(a)+b\|,$$
	with $\delta_g(a)\in I\cap A^\delta_g$, thus $\|b+I\|=\|b+(I\cap A^\delta_g)\|,$ that is, the canonical map: $A^\delta_g/(I\cap A^\delta_g)\to A/I$ is isometric, and so $q(A^\delta_g)=A^\delta_g/(I\cap A^\delta_g)$ is a Banach space, as required. Finally, $\delta_g$ lifts to a projection $\delta^I_g: A/I\to A^\delta_g/(I\cap A^\delta_g)$, and for each $u\in \Gz$, the projection $\delta^I_u: A/I\to A^\delta_u/(I\cap A^\delta_u)$ is a contractive conditional expectation, vanishing on each $A^\delta_g/(I\cap A^\delta_g)$, for $g\in G\backslash\Gz$.      
	
	$(iii)$ Since $q$ is surjective, the above equality defines a linear map $\delta^I$. For $u\in \Gz$ and $a \in
	A^\delta_u$, 
	$$(q \otimes \id) \circ \delta(a) =
	q(a) \otimes i_G(u)=0,$$ if and only
	if $a \in I$, thus $\ker((q \otimes \id) \circ
	\delta) \cap A^\delta_u = I\cap A^\delta_u$. It follows from the assumption on  $I$ that $\ker\big((q
	\otimes \id) \circ \delta\big)$ contains $I$, and so $\delta^I$ is well-defined, and clearly a $*$-homomorphism.

	To show nondegeneracy (mod $\Gz$), given $n\geq 1$, let  $y_1,\cdots,y_n\in A/I,$  $g_1,\cdots g_n\in G$,  and consider the basic tensors $y_i\otimes 1_{g_i}\in A/I\otimes C^*(G)$. Choose $x_i\in A$ with $q(x_i)=y_i$, for $1\leq i\leq n$, and choose units $u_i\in \Gz$ such that $\sum_ix_i\otimes (1_{u_i}*1_{g_i})$ is in the closed span of $\delta(A)(A\otimes C^*(G))$, then since $q$ is norm decreasing, $\sum_i y_i\otimes (1_{u_i}*1_{g_i})=(q\otimes \id)\big(\sum_ix_i\otimes (1_{u_i}*1_{g_i})\big)$ is in the closed span of $$(q\otimes\id)\delta(A)(A\otimes C^*(G))=\delta^I(A)(A/I\otimes C^*(G)),$$
	thus $\delta^I$ is
	nondegenerate (mod $\Gz$).

	Finally, for $g \in G$
	and $a \in A^\delta_g$, 
	\[
	(\id \otimes \delta_G)\circ\delta^I(q(a))
	= q(a) \otimes i_G(g) \otimes i_G(g)
	= (\delta^I \otimes \id)\circ \delta^I(q(a)), 
	\]
	and again by a density argument, we get the coaction identity for $\delta^I$.
\end{proof}

\begin{dfn}
Let $\delta$ be the coaction of a discrete groupoid $G$ on a $C^*$-algebra $A$. 
Let  $\lambda_G: C^*(G)\to C^*_r(G)$ be the regular representation. 
The coaction $\delta$ is called normal if $\id\otimes \lambda_G$ is injective on $\delta(A)$.
\end{dfn}

When $\delta$ is injective (something which is always assumed to be the case for group coactions), this is equivalent to the condition that $(\id\otimes \lambda_G)\circ \delta$ is injective on $A$. If $\delta$ is injective on a $C^*$-subalgebra $B$ of $A$, then normality of $\delta$ implies that $(\id\otimes \lambda_G)\circ \delta$ is also injective on $B$.   

For the sake of brevity, let us write $\tilde\delta:=(\id\otimes \lambda_G)\circ \delta$. For $u\in\Gz$, we put,
$$A^\delta_u:=\{a\in A: \delta(a)=a\otimes i_G(u)\},\ A^{\tilde\delta}_u:=\{a\in A: \tilde\delta(a)=a\otimes \lambda_G(u)\}.$$
Next, first note that in the notations of \cite[page 18]{w},  for $u\in \Gz$ and $f\in C_c(G)$,
\begin{align*}
	|\langle 1_u,f\rangle|:=|f(u)|&\leq \sum_{s(g)=u}|f(g)|=\|{\rm Ind}\delta_u(f)1_u\|_2^2\leq\|{\rm Ind}\delta_u(f)\|\leq\|f\|_r,
\end{align*}
that is, $1_u$ could  be regarded as a continuous linear functional on $C^*_r(G)$. 
We denote this functional by $\tilde 1_u$ to distinguish it from the functional $1_u$ on $C^*(G)$ constructed in the proof of Lemma \ref{inj} (indeed, $1_u=\tilde 1_u\circ\lambda_G$). We also consider the bounded linear maps on $A$, defined by $\delta_u:=\delta_{1_u}=(\id\otimes 1_u)\circ \delta$. Recall that $\delta_u(A)=A^\delta_u$, and we have a  conditional expectation,
$$\Phi_0:=\oplus \delta_u: A\to \bigoplus_{u\in \Gz} A^\delta_u.$$
We also have a conditional expectation $\Phi: C^*(G)\to C_0(\Gz)$, mapping $f\in C_c(G)$ to $\sum_{u\in \Gz}f(u)$, or alternatively, if we identify $C_0(\Gz)$ with the $c_0$-direct sum $\bigoplus_{u\in \Gz} \mathbb C$, $\Phi$ is the map sending $f\in C_c(G)$ to $(f(u))_{u\in \Gz}$. Under this identification, $\Phi=\oplus 1_u$, where $1_u: f\mapsto f(u)$ is the above state on $C^*(G)$. This means that $\Phi_0=(\id\otimes \Phi)\circ \delta$. In the group case, where coactions  are assumed to be injective, $\Phi_0$ is faithful on positive elements iff $\Phi$ is so, but in our setting faithfulness of $\Phi$ on positive elements is in general weaker than that of $\Phi_0$. Indeed, we have the following result.   

\begin{lemma}\label{sepfam}
Let $G$ be a discrete groupoid, then the set $\{\tilde{1}_u: u\in \Gz\}$ is a faithful family of states on $C^*_r(G)$. 
\end{lemma}
\begin{proof}. First note that $$\langle \tilde 1_u,f\rangle=f(u)=\langle {\rm Ind}(\delta_u)(f)1_u,1_u\rangle, \ \ (u\in\Gz, f\in C_c(G)),$$
where $1_u$ is regarded as an element in $\ell^2(G_u)$. This in particular shows that $\tilde 1_u$ extends to a state on $C^*_r(G)$. 	

We just need to show that if $\langle\tilde 1_u,b\rangle=0$, for $b\in C^*_r(G)_{+}$ and each $u\in \Gz$, then $b=0$. Without loss of generality, we may assume that $b$ is of norm at most one. Choose element $c$ with $b=c^*c$, then $c$ is of norm at most one. Next, for a given $\varepsilon>0$, choose $f\in C_c(G)$ within $\varepsilon$ of $c$, and observe that $f^**f$ is within $3\varepsilon$ of $b$. Now,  $$\sum_{s(\gamma)=u} |f(\gamma)|^2=\sum_{r(\gamma)=u}\bar f(\gamma^{-1})f(\gamma^{-1}u)=f^**f(u)=\langle \tilde 1_u, f^**f-b\rangle\leq 3\varepsilon,$$   
for each $u\in \Gz$, which in turn implies that $\|f\|_\infty\leq \sqrt{3\varepsilon}$, and so $f=0$, as $\varepsilon$ was arbitrary, thus $b=0$. 
\end{proof}

\begin{lemma}\label{lem:when deltaI normal}
Let a discrete groupoid $G$ coact on a $C^*$-algebra $A$ with $\delta$. 
	\begin{enumerate}
		\item[$(i)$]\label{it:deltaI normal} The coaction $\delta$ is normal iff the conditional expectation
		$$\id\otimes \Phi:A\otimes C^*(G)\to A\otimes C_0(G)$$ is faithful on positive cone of $\delta(A)$. When $\delta$ is injective, this is also equivalent to faithfulness of the conditional expectation 
		$$\Phi_0:=\oplus \delta_u: A\to \bigoplus_{u\in \Gz} A^\delta_u,$$
		on positive cone of $A$.
		
		\item[$(ii)$]\label{it:cond Q} Let $I$ be an induced ideal of $A$ with quotient map $q: A\to A/I$. The coaction $\delta^I$ is normal iff $q^{-1}\big(\ker(\delta^I)\big)=\ker((q \otimes
		\lambda_G)\circ\delta)$. When $\delta^I$ is injective, this is also equivalent to $I=\ker((q \otimes
		\lambda_G)\circ\delta)$. 
	\end{enumerate}
\end{lemma}
\begin{proof}
$(i)$ If $\id\otimes \Phi$ is faithful on $\delta(A_{+})$, 	 given  $a\in A$, if $(\id\otimes \lambda_G)
	\circ \delta(a)=0$, then, $$(\id\otimes \lambda_G)
	\circ \delta(a^*a)=[(\id\otimes \lambda_G)
	\circ \delta(a^*)][(\id\otimes \lambda_G)
	\circ \delta(a)]=0,$$
and	for $u\in \Gz$,
$$(\id\otimes 1_u)
\circ \delta(a^*a)=(\id\otimes \tilde 1_u\circ\lambda_G)
\circ \delta(a^*a)=(\id\otimes \tilde 1_u)
\circ(\id\otimes \lambda_G)
\circ \delta(a^*a)=0,$$
thus, $(\id\otimes\Phi)\big(\delta(a^*a)\big)=0$, thus $\delta(a^*a)=0$, and so $\delta(a)=0$, i.e., $\id\otimes \lambda_G$ is injective on $\delta(A)$, that is, $\delta$ is normal.  
Conversely, arguing by contradiction, if $(\id\otimes \Phi)\big(\delta(a^*a)\big) = 0$,  but $(\id\otimes \lambda_G)
\circ \delta(a^*a)$ is a nonzero positive element of $A\otimes C^*_r(G)$. Choose a state $\omega$ of $A$
such that $b:=(\omega\otimes \lambda_G)
\circ \delta(a^*a)$ is a nonzero positive element of $C^*_r(G)$. For each $u\in\Gz$,  
$$\tilde 1_u(b)=\tilde 1_u\circ(\omega\otimes \lambda_G)
\circ \delta(a^*a)=\omega[(\id\otimes (\tilde 1_u\circ\lambda_G))(\delta(a^*a))]=\omega[(\id\otimes 1_u)(\delta(a^*a))]=
0,$$
which is a contradiction, by Lemma \ref{sepfam}. When $\delta$ is injective, by equality $\Phi_0=(\id\otimes \Phi)\circ \delta$, faithfulness of $\Phi_0=(\id\otimes \Phi)$ on positive elements of $\delta(A)$ is equivalent to faithfulness of $\Phi_0$ on positive elements of $A$. 
	
$(ii)$ First observe that,
	\[
	(q \otimes \lambda_G)\circ\delta
	= (\id \otimes \lambda_G) \circ (q \otimes \id) \circ \delta
	= (\id \otimes \lambda_G) \circ \delta^I \circ q.
	\]
If $\delta^I$ is normal, by the above equality $x\in \ker((q \otimes
\lambda_G)\circ\delta)$ iff $(\id \otimes \lambda_G)(\delta^I(q(x))=0$ iff $\delta^I(q(x))=0$, that is, $x\in q^{-1}\big(\ker(\delta^I)\big),$ since $\id\otimes\lambda_G$ is injective on range of $\delta^I$.  Conversely if $q^{-1}\big(\ker(\delta^I)\big)=\ker((q \otimes
\lambda_G)\circ\delta)$ and   $(\id \otimes \lambda_G)(\delta^I(q(x))=0$, then 	$(q \otimes \lambda_G)\circ\delta(x)=0,$ thus $q(x)\in \ker(\delta^I)$, that is, $\delta^I(q(x))=0$, which means that $\id\otimes \lambda_G$ is injective on the range of $\delta^I$, that is, $\delta^I$ is normal. When $\delta^I$ is injective, $q^{-1}\big(\ker(\delta^I)\big)=q^{-1}(0)=I$. 
\end{proof}

Recall that a discrete group $G$ is called exact if its
reduced $C^*$-algebra $C^*_\reduced(G)$ is exact.

\begin{lemma} \label{prop:exact}
	Let $G$ be a discrete groupoid. Consider  the following statements:
	\begin{enumerate}
		\item[($i$)] For every normal injective coaction $\delta$ of $G$ on a $C^*$-algebra
		$A$, and every induced ideal $I$, the induced coaction $\delta^I$ of $G$ on
		$A/I$ is normal and injective.
		
				\item[($ii$)] $G$ is exact. 
	\item[($iii$)] For every normal  coaction $\delta$ of $G$ on a $C^*$-algebra
$A$, and every strongly induced ideal $I$, the induced coaction $\delta^I$ of $G$ on
$A/I$ is normal.	
\end{enumerate}

Then $(i)\Rightarrow(ii)\Rightarrow(iii).$
\end{lemma}

\begin{proof}
	 $(i)\Rightarrow(ii).$ Assume that $G$ is not exact. Then there exists a $C^*$-algebra $B$
	with an ideal $J$ such that, min-tensoring from right with $C^*_r(G)$ does not preserve exactness of the short  exact sequence
	$$0\to J\to B\to B/J\to 0,$$
	that is, $J\otimes C^*_\reduced(G)\subsetneq
	\ker (q\otimes \id)$, where $q \colon B\to B/J$ is the
	quotient map. Consider the coaction $\delta:=\id_B\otimes\delta_G$ of $G$ on $A:=B\otimes C^*_\reduced(G)$, where $\delta_G$ is the canonical coaction of $G$ on $C^*_r(G)$ defined by $\delta_G(\lambda_G(g)):=\lambda_G(g)\otimes i_G(g)$, for $g\in G$. Let
	$$\Delta \colon C^*_\reduced(G)\to C^*_\reduced(G)\otimes C^*_\reduced(G); \ x\mapsto x\otimes x,\ \ (x\in C^*_\reduced(G))$$  be the co-multiplication map, and observe that, 
	$(\id\otimes\lambda_G)\circ\delta=\id\otimes\Delta$. In particular, $(\id\otimes\lambda_G)\circ\delta$ is
	injective, that is, $\delta$ is  normal and injective. We claim that, the  ideal $I:=J\otimes C^*_\reduced(G)$ of
	$A=B\otimes C^*_\reduced(G)$ is induced. Let $u\in\Gz$, then the projection $1_u$ in $C^*_r(G)$ generates a one dimensional subspace $\mathbb C 1_u$ and $(B\otimes C^*_\reduced(G))^\delta_u=B\otimes
	\mathbb  C 1_{u}$. By the identity  $1_g*1_{s(g)}=1_g$, for $g\in G$, the union of intersections of $I$ with fibres $(B\otimes C^*_\reduced(G))^\delta_u$ contains the algebraic tensor product $J\odot C_c(G)$, which is dense in $I$, establishing the claim.  Finally, for the quotient map $\tilde q: A\to A/I$, we have $(\tilde q\otimes\lambda_G)\circ\delta=q\otimes\Delta$, 
	
	\[I=J\otimes C^*_\reduced(G)\subsetneq
	\ker (q\otimes \id)\subseteq
	\ker(q\otimes\Delta) = \ker((\tilde q\otimes\lambda_G)\circ\delta),\]
	hence by Lemma \ref{lem:when deltaI normal}$(ii)$, 
	$\delta^I$ is either not normal or not injective. 
	
	 $(ii)\Rightarrow(iii).$ Assume now that $G$ is exact and  let $\delta$ be a normal  coaction of $G$
	on a $C^*$-algebra $A$, and  $I$ be a strongly  induced ideal of $A$. We show that $\delta^I$ is also normal. If $I$ is zero, there is nothing to prove, so let us assume that $I$ is nonzero. Let  $q: A\to A/I$ be the quotient map. Let $x\in\ker((q \otimes
	\lambda_G)\circ\delta)$, then $(\id_A\otimes \lambda_G)\circ \delta(x)\in\ker(q\otimes \id )=I\otimes C^*_r(G)$, where the equality follow from exactness of $G$. Since $I$ is a strongly induced ideal, $\delta$ restricts to a nondegenerate $*$-homomorphism
	$$\delta: I\to (I\otimes C^*(G))\cap \delta(A)=:D.$$ 
	Since $\id\otimes\lambda_G: A\otimes C^*(G)\to A\otimes C^*_r(G)$ is surjective, it is nondegenerate \cite[Proposition 2.2.16]{weg}, and so the $*$-homomorphism 
$$(\id\otimes\lambda_G)\circ\delta: I\to (I\otimes C^*_r(G))\cap (\id\otimes\lambda_G)\delta(A)=:C$$ 	
is nondegenerate. By Lemma \ref{ndeg}, $(\id\otimes\lambda_G)\circ\delta$ maps an approximate  identity $(e_i)$ of $I$ to an approximate identity of $C$. Now, since $(\id\otimes\lambda_G)\circ \delta(x)\in C$,   
	$$(\id\otimes \lambda_G)\circ\delta(e_ix)= (\id\otimes \lambda_G)\circ\delta(e_i)(\id\otimes \lambda_G)\circ\delta(x)\to (\id\otimes \lambda_G)\circ\delta(x),$$ 
	as $i\to \infty$. By normality of $\delta$, $\id\otimes \lambda_G$ is injective and so isometric on the range of $\delta$, hence it follows that the net $\delta(e_ix)$ is convergent in $I\otimes C^*(G)$, say to $y$. Since $\delta(I)$ is closed, there is $z\in I$ with $y=\delta(z)$, which in turn gives, $(\id\otimes \lambda_G)\circ\delta(x)=(\id\otimes \lambda_G)\circ\delta(z)$, and so $\delta(x)=\delta(z)$, again by normality of $\delta$. Thus, $\delta(x)\in \delta(I)\subseteq I\otimes C^*(G)$. Therefore, 
	$$\delta^I(q(x))=(q\otimes\id)(\delta(x))=0,$$
	that is, $x\in q^{-1}(\ker(\delta^I))$, as required.  
\end{proof}

Since the ideal $I:=J\otimes C^*_r(G)$ constructed in the proof of the implication $(i)\Rightarrow(ii)$ of the above lemma is not only induced, but also strongly induced, we immediately get the following corollary. 

\begin{cor} \label{prop:exact2}
	Let $G$ be a discrete groupoid. Then   the following statements are equivalent:
	\begin{enumerate}
		\item[($i$)] For every normal injective coaction $\delta$ of $G$ on a $C^*$-algebra
		$A$, and every strongly induced ideal $I$, the induced coaction $\delta^I$ of $G$ on
		$A/I$ is normal and injective.
		
		\item[($ii$)] $G$ is exact. 
		\item[($iii$)] For every normal  coaction $\delta$ of $G$ on a $C^*$-algebra
		$A$, and every strongly induced ideal $I$, the induced coaction $\delta^I$ of $G$ on
		$A/I$ is normal.	
	\end{enumerate}
\end{cor}

This last result urges finding examples of strongly induced ideals of $\Tc(X)$ and $\NO{X}$, for the canonical coactions of $G$. Using Lemma \ref{nondeg}, this is not hard when $G$ has uniformely bounded fibres. 

\begin{lemma} \label{nondeg2}
	If $G$ has uniformly bounded fibres, the closed ideal $\langle \Ff\rangle$ of $\Tc(X)$ generated by the core $\Ff$ is a strongly induced ideal with respect to the gauge action. Similarly, closed ideal $\langle q_{{\rm CNP}}(\Ff)\rangle$ of $\NO{X}$, generated by the image of the core $\Ff$ under the quotient map $q_{\rm CNP}: \Tc(X)\to \NO{X}$, is a strongly induced ideal with respect to the induced gauge coaction. 
\end{lemma}
\begin{proof}
	Since $\Ff$ is generated by $\bigcup_{u\in \Gz}\Tc(X)^\delta_u$,  $\langle \Ff\rangle$ is an induced ideal. Since $\Ff$ contains an approximate identity of $\Tc(X)$, we have $M(\Ff)=M(\langle \Ff\rangle)=M(\Tc(X))$. Therefore, $D:=(\Ff\otimes C^*(G))\cap\delta(\Tc(X))$ and $D^{'}:=(\langle \Ff\rangle\otimes C^*(G))\cap\delta(\Tc(X))$ have the same multiplier algebra. Now the first statement follows from Lemma \ref{nondeg}. The second statement is proved similarly. 
\end{proof}
In the notations of  Lemma \ref{tg}, let   $A^\reduced:=C^*_r(E_A)$, and let $\lambda_r: A\to A^r$ be the corresponding regular representation. Consider the  canonical regular representation $\lambda_G: C^*(G)\to C^*_r(G)$. If  $\delta$ is normal,  $\delta^n$ is a reduced coaction of $G$ on $A$. If not, then $G$ coacts  normally on $A^n:=A/{\ker(\delta^n)}$.    
As in Lemma \ref{co}$(ii)$, there is  a normal coaction $\delta^\normal$ of $G$ on $A^\reduced$ with 
$
    \delta^\normal\bigl(\lambda_{r}(a)\bigr)=
    \lambda_{r}(a)\otimes \lambda_G(g)$,  for $g\in G$ and $a\in A^\delta_g,$ and $\ker(\delta^n)=\ker(\lambda_r)$, that is, 
$A^{\reduced}$ is isomorphic to $A^{\normal}$. 

Let $(G,P)$ be a quasi-lattice ordered groupoid, and  $X$ be a
compactly aligned product system of Hilbert bimodules over $P$.
By Lemma~\ref{prop:existence_coaction}, $\Tc(X)$ admits a
coaction $\delta$, giving rise to a Fell bundle
$E_{X}:=(\Tc(X)^\delta_g\times\{g\})_{g\in G}$ over $G$, and by Lemma \ref{core}$(i)$, $C^*(E_X|_{\Gz})\simeq\Ff$. We  identify $\Tc(X)^\delta_g$,  with
$\Tc(X)^\delta_g\times\{g\}$. 

Let $G$ be discrete and consider the map  $\iota_X:\bigcup_{g\in
	G}\Tc(X)^\delta_g\hookrightarrow C_c(E_X)$, defined as the canonical embedding on each fibre. This combined with the canonical inclusion $\eta_X: C_c(E_X)\hookrightarrow \Tc(X)$ gives and embedding of $\bigcup_{g\in
	G}\Tc(X)^\delta_g$ in $\Tt_X$.  

\begin{lemma}\label{equi}
Let $\iota_{\rm max}: C_c(E_X)\hookrightarrow C^*(E_X)$ be the canonical embedding, then 

$(i)$ the map
$\iota_{\rm max}\circ\iota_X\circ i_X:X\to C^*(E_X)$
is a Nica covariant Toeplitz representation.

$(ii)$ $C^*(E_X)$ is equivariantly isomorphic to  $\Tc(X)$. 
\end{lemma}
\begin{proof}
$(i)$ Just note that $i_X$ is a Nica covariant representation and
$\iota_{\rm max}\circ\iota_X$ restricts to a linear map on
each fiber $\Tc(X)^\delta_g$ and to a $*$-homomorphism on the
fiber $\Tc(X)^\delta_{s(g)}$. 

$(ii)$ By the universal property of
$\Tc(X)$, there is a $*$-homomorphism $\psi_X:\Tc(X)\to
C^*(\mathcal{B})$ such that $\psi_X\circ i_X=
\iota_{\rm max}\circ\iota_X\circ i_X$. Choose an
epimorphism $\phi_{X}: C^*(E_X)\to
\Tc(X)$ such that $\phi_X
\circ\iota_{\rm max}=\eta_X$. We  have
\begin{equation*}
  \phi_X\circ\psi_X\circ i_X=\phi_X\circ\iota_{\rm max}\circ\iota_X\circ i_X 
  = \eta_X\circ\iota_X\circ i_X= i_X,
\end{equation*}
and
\begin{equation*}
  \psi_X\circ\phi_X\circ\iota_{\rm max}\circ\iota_X\circ i_X
  = \psi_X\circ\eta_X\circ\iota_X\circ i_X=\psi_X\circ i_X=\iota_{\rm max}\circ\iota_X\circ i_X,
\end{equation*}
that is, $\psi_X$ is the inverse of $\phi_X$, and  $\phi_X$ is an equivariant isomorphism of $C^*(E_X)$ and $\Tc(X)$.
\end{proof}

Next, suppose that $G$ is discrete and $X$ satisfies conditions of Theorem \ref{thm:projective property}. By Lemmas \ref{core} and \ref{thm:inj on core}, the kernel of the canonical
homomorphism $q_{\CNP} \colon \Tc(X)\to \NO{X}$ is generated by its
intersection with $\Ff$. Therefore Lemma~\ref{prp:deltaI
existence} applied to the coaction $\delta$ on $\Tc(X)$, yields
a coaction on $\NO{X}$, called the gauge coaction, and again denoted by $\delta$. The spectral
subspaces
\[(\NO{X})^{\delta}_g:= \{\, c\in\NO{X} \mid \delta(c)=c\otimes
i_G(g) \,\}
\]
give rise to a Fell bundle $E_\mathcal{O}$ over $G$, and by an argument as in that of Lemma \ref{equi}$(ii)$, $C^*(E_\mathcal{O})$ is equivariantly isomorphic to $\NO{X}$
for coactions $\delta_{E_{\mathcal O}}$
and  $\delta$.

\begin{proof}[Proof of Theorem~\ref{thm:projective property}]
Let $\NO{X}^\reduced$ be the reduced cross-sectional algebra of
the Fell bundle $E_\mathcal{O}$, and $\lambda_r
\colon \NO{X} \to\NO{X}^\reduced$ be the corresponding regular representation. Put $j^\reduced_X :=\lambda_{r} \circ j_X$, and  let $\delta^n$ to be the induced normal coaction
on $\NO{X}^\reduced$.

To prove ~(1), let us note that $\NO{X}$ is generated by the
CNP-covariant representation $j_X$, which is injective by Proposition \ref{faithful2}, thus  $j_X^r$
is CNP-covariant. Also, since $\lambda_{r}$ is surjective,
$\NO{X}^\reduced$ is generated by $j_X^r$. Now, for $u\in\Gz$, 
$\lambda_{r}$ restricts to a bijection from
$(\NO{X})^{\delta^n}_u$ to $(\NO{X}^\reduced)^{\delta}_u$,
and since $j_X(A_u)\subset (\NO{X})^{\delta}_u$, the
representation $j_X^\reduced$ is also injective, and we have
$\delta^n(j^\reduced_X(x))=j^\reduced_X(x)\otimes i_G(d(x))$, 
for $x\in X$.

Next to prove (2), take $\psi \colon X \to D$ as in~(2), and
let $\beta$ be a coaction on $D$ with the prescribed gauge compatibility. For $g\in G$,  let
$D^\beta_g=\{\, b\in D \mid \beta(b)=b\otimes i_G(g) \,\}$.
Then $\{D^\beta_g\}_{g\in G}$ is
a topological grading of $D$, by Lemma \ref{tg}. By the universal property of
$\Tc(X)$, there is a $*$-homomorphism $\psi_* \colon \Tc(X) \to D$
with $\psi = \psi_* \circ i_X$. Consider the ideal 
$I_u:=\psi_*(\ker(q_{\CNP}))\cap
D^\beta_u$   of
$D_u^\beta$, for $u\in \Gz$, and let $I$ be the ideal of $D$ generated by $\bigcup_{u\in\Gz}I_u$. Then by Lemma \ref{lemma:inj} and the fact that the image of $\psi$
generates $D$, $I$ is the ideal of $D$ generated by $\psi_*(\ker(q_{\CNP})\cap \Ff)$. 
By construction, $I$ is an induced ideal. Let $\pi \colon D\to D/I$ be the
quotient map, then by Lemma \ref{prp:deltaI existence}$(ii)$, $\{\pi(D^\beta_g)\}_{g\in G}$ is a topological grading of
$D/I$.
Since the image of $\psi$ generates $D$, we also have
$\pi(\psi_*(\Tc(X)^\delta_g))=\pi(D^\beta_g)$, for each $g\in G$.
Next, let us show that, for every $g\in G$,
\begin{equation*}\label{same_kernel}
\ker
(\pi\circ\psi_*)\cap \Tc(X)^\delta_g=\ker (q_{\CNP})\cap \Tc(X)^\delta_g.
\end{equation*}

Let us first handle the case $g=:u\in\Gz$. 
For the first case, by Lemma \ref{core}$(i)$, we only need to show that  $\ker
(\pi\circ\psi_*)\cap \Ff=\ker (q_{\CNP})\cap \Ff$.
If $c\in \ker (q_{\CNP})\cap \Ff$, then
$\psi_*(c)\in I$, thus $\pi(\psi_*(c))=0$, and so $c\in \ker (\pi\circ\psi_*)\cap \Ff$. Conversely, since $\pi\circ \psi_*=(\pi\circ \psi)_*$, by Lemma \ref{lemma:inj} it
suffices to show that the Toeplitz representation $\pi\circ\psi$ is injective.
If  $a\in A_u$ and $\pi(\psi_u(a))=0$, we have, $$\psi_u(a)\in I\cap
D_u^\beta\subseteq
\psi_*(\ker(q_{\CNP}))\cap D_u^\beta\subseteq \psi_*\big(\ker(q_{\CNP})\cap \Ff\big),$$ thus 
there exists  $b\in\ker(q_{\CNP})\cap \Ff$ such that
$\psi_*(b)=\psi(a)$.
Hence, $$b-i_X(a)\in\ker(\psi_*)\cap\Ff\subseteq \ker (q_{\CNP}),$$ where the inclusion follows from Lemma \ref{lemma:inj}. 
Therefore, $i_X(a)\in\ker (q_{\CNP})\cap i_X(A_u)=0$, hence
$a=0$, as $i_X$ is an injective Toeplitz representation, i.e.,  $\pi\circ\psi$ is injective,  as required. The claimed equality now follows for each $g\in G$, since  $
	c\in\ker (\pi\circ\psi_*)\cap \Tc(X)^\delta_g$ is equivalent to $ c^*c\in \ker
	(\pi\circ\psi_*)\cap \Tc(X)^\delta_{s(g)}=\ker (q_{\CNP})\cap \Tc(X)^\delta_{s(g)}$, which is in turn equivalent to 
	$c\in\ker (q_{\CNP})\cap \Tc(X)^\delta_g$.

The above equality now implies that for  $g\in G$, the map,
$$\tilde\phi_g: (\NO{X})^{\delta}_g\to \pi(D^\beta_g); \ 
q_{CNP}(a)\mapsto \pi\circ\psi_*(a), \ \ (a\in \Tc(X)^\delta_g),$$
is well defined. This gives a covariant isomorphism of Fell bundles with fibres $\pi(D^\beta_g)\times\{g\}$ and $(\NO{X})^{\delta}_g\times\{g\}$. By 
the couniversal properties
of $\NO{X}^\reduced$ and $C_r^*(E_{\pi(D)})=:\pi(D)^r$, there is an isomorphism
$\tilde{\phi} \colon \NO{X}^\reduced\to \pi(D)^r$
such that $\lambda^I_{r}\circ \tilde\phi_g=\tilde\phi \circ \lambda_{r}$  on each
$(\NO{X})^{\CNPgaug}_g$, where $\lambda_r: \NO{X}\to \NO{X}^\reduced$ and $\lambda^I_r: (D/I)\to (D/I)^r$ are the corresponding regular representations.
Put $\phi:=\tilde\phi^{-1}\circ
\lambda^I_{r}\circ \pi$. Then $\phi:D\to\NO{X}^\reduced$ is a $*$-homomorphism, and for each $x\in X$,
\begin{align*}
\phi(\psi(x))
&=\tilde\phi^{-1}(\lambda^I_{r}(\pi(\psi_*(i_X(x)))))=\tilde\phi^{-1}
(\lambda^I_{r}(\phi_{d(x)}(q_{\CNP}(i_X(x)))))\\
&=\lambda_{r}(q_{\CNP}(i_X(x)))=\lambda_{r}(j_X(x))=j_X^r(x),
 \end{align*}
as claimed in (2).

Finally, to see the uniqueness, let us suppose that $(C, \rho,\gamma)$
also satisfies (1) and (2). By the property~(2) for $\NO{X}^\reduced$ and $C$, there are
equivariant homomorphisms $\phi: C\to \NO{X}^\reduced$ and 
$\phi^{'}:\NO{X}^\reduced\to C$, which are then clearly 
inverse of each other.  
\end{proof}

\begin{dfn}\label{dfn:giup}
	For a quasi-lattice ordered groupoid $(G,P)$ and a
	$\tilde\phi$-injective compactly aligned product system $X$
	over $P$. We say that $\NO{X}$ has the gauge-invariant
		uniqueness property provided that a surjective $*$-homomorphism $\phi:\NO{X}\to B$ is injective if
	and only if the following hold:

(GI1) there is a coaction
		$\beta$ of $G$ on $D$ such that $\beta\circ\phi=(\phi\otimes\id)\circ \delta$, 
		
		(GI2) the homomorphism $\phi\vert_{j_X(A)}$ is injective.
\end{dfn}

\begin{cor} \label{inj}
Let $G$ be discrete and $X$ be compactly aligned. Suppose that either the left action on each
fibre is injective, or  $P$ is directed and $X$ is
$\tilde\phi$-injective.
\begin{enumerate}
\item If $\phi \colon \NO{X}^\reduced\to D$ is a surjective
    $*$-homomorphism, then $\phi$ is injective if and only
    if $\phi{|j_X^\reduced(A)}$ is injective and there is a
    coaction
    $\beta$ of $G$ on $D$ such that
    $\beta\circ\phi=(\phi\otimes
    \id)\circ\delta^n$.
\item Suppose that $(C,\rho)$ is a couniversal pair for
    injective gauge-compatible (not necessarily Nica covariant)
    Toeplitz representations of $X$. Then
    there is a $*$-isomorphism $\phi \colon \NO{X}^\reduced\to C$
    such that $\phi(j_X^\reduced(x))=\rho(x)$ for all $x\in
    X$.
    
\item Assume that both coactions $\beta$ and $\delta^n$ are injective. If $\psi \colon X \to D$ is an
injective Nica covariant gauge-compatible
representation whose image generates $D$, then the surjective
$*$-homomorphism $\phi \colon D \to \NO{X}^\reduced$ of
Theorem~\ref{thm:projective property}(\ref{it:c-u property}) is
injective if and only if $\psi$ is Cuntz-Pimsner covariant and
$\beta$ is normal.

\end{enumerate}
\end{cor}
\begin{proof}
(1) The necessity is trivial. For the sufficiency, note that $\phi\circ j_X^\reduced$ is an injective
gauge-compatible Nica covariant representation of $X$ in $D$, whose image generates $D$.  Theorem~\ref{thm:projective property} guarantees the existence of an epimorphism which is a left inverse for $\phi$.

(2) Since $\NO{X}^\reduced$ is generated by an injective
gauge-compatible Toeplitz representation of $X$, by the
couniversal property of $(C,\rho)$, there is an epimomorphism $\phi \colon \NO{X}^\reduced\to C$ with $\phi(j_X^\reduced(x))=\rho(x)$, for  $x\in X$. By part~(1), $\phi$ is injective and so an isomorphism.

(3) If $\phi$ is injective, then $\psi$ is CNP-covariant
and $\beta$ is normal as $j_X^\reduced$ is CNP-covariant and $\delta^n$ is normal.
Conversely, if $\psi$ is CNP-covariant and $\beta$
is normal, by the universal property of $\NO{X}$, there is  a
homomorphism $\Pi\psi  \colon \NO{X} \to D$ satisfying 
$\lambda_{r}
= \phi \circ \Pi\psi$, for the regular representation
$\lambda_{r}
\colon \NO{X} \to\NO{X}^\reduced$. For each $u\in\Gz$, the representation
$\lambda_{r}$ restricts to an isomorphism of
$(\NO{X})^{\delta}_u$  onto $(\NO{X}^\reduced)^{\delta^n}_u$,
hence $\phi$ restricts to an isomorphism of $D^\beta_u$ onto
$(\NO{X}^\reduced)^{\delta^n}_u$. Since $\beta$ is normal, by Lemma \ref{lem:when deltaI normal} there is a faithful conditional expectation $\Phi^\beta:=\oplus_u\beta_u \colon D
\to \bigoplus_u D^\beta_u$, and $\phi$ intertwines $\Phi^\beta$ and the
conditional expectation $\Phi^{\delta^n}$ from $\NO{X}^\reduced$ onto
$\bigoplus_{u}(\NO{X}^\reduced)^{\delta^n}_u$, and injectivity of $\phi$ follows by an standard argument. 
\end {proof}

For the next result, we first need to define and relate the notions of amenability and approximation property for Fell bundles over groupoids. The second notion is defined in the case of bundles over discrete groups by Ruy Exel \cite[Definition 4.4]{Exel}.

\begin{dfn} \label{ap} Let $G$ be a discrete groupoid and $E$ is a Fell bundle over $G$. We say that 
	
	$(i)$ $E$ is amenable if the regular representation $\lambda_r: C^*(E)\to C^*_r(E)$ is injective,
	
	$(ii)$ $E$ has approximation property if for each $u\in\Gz$ there exists a net $(a^u_i)$ of functions $a^u_i : G \to E_u$ such that,

$$\sup_i\big\|\sum_{r(h)=u}	a^u_i(h)^*a^u_i(h)\big\|<\infty, \ \ (u\in \Gz)$$
and,
$$\lim_i\big\|\sum_{r(h)=u}	a^u_i(gh)^*ba^u_i(h)-b\big\|=0, \ \ (u\in \Gz, g\in G_u, b\in E_g).$$
\end{dfn} 

Let us quickly observe that if $E$ has approximation property, then the functions $a^u_i$ could be chesen to be of finite support. Indeed, given $u\in \Gz$, there exists a constant $ M > 0$ such
that for each $\varepsilon>0$, each finite subset $F\subseteq G$, and elements $b_k\in E_k$, for $k\in F$, there exists a function $a^u: G\to E_u$ such that $$\big\|\sum_{r(h)=u}	a^u(h)^*a^u(h)\big\|<M, \ \ \big\|\sum_{r(h)=u}	a^u(gh)^*b_ka^u(h)-b_k\big\|<\varepsilon, \ \ (k\in F, g\in G_u).$$
Choose a finite subset $F_\varepsilon\supseteq F$ with 
$$\sup_{k\in F}\big\|\sum_{h\in G^u\backslash F_\varepsilon}	a^u(gh)^*b_ka^u(h)\|<\varepsilon,$$
for each $g\in G_u$. 
Let $a^u_{\varepsilon, F}$ be the restriction of $a^u$ to $F_\varepsilon$, and observe that,
$$\big\|\sum_{r(h)=u}	a^u_{\varepsilon, F}(h)^*a^u_{\varepsilon, F}(h)\big\|<M, \ \ \big\|\sum_{r(h)=u}	a^u_{\varepsilon, F}(gh)^*b_ka^u_{\varepsilon, F}(h)-b_k\big\|<2\varepsilon, \ \ (k\in F, g\in G_u),$$
which establishes the claim. 

Let $G$ be an \'{e}tale groupoid and $E$ be a Fell bundle over $G$. A representation of $E$ on a Hilbert space $H$ in the sense of Fell and Doran \cite[VIII. 9.1]{Fell} is nothing but a morphism $\pi: E\to \mathbb B(H)$ in the sense of the paragraph before Lemma \ref{co}, where $\mathbb B(H)$ is regarded as a trivial $G$-bundle.  

\begin{lemma} \label{map}
Let $G$ be a discrete groupoid and $E$ be a Fell bundle over $G$. Let $u\in\Gz$ and $a: G\to E_u$ be a map such that the sum $a_u:=\sum_{s(h)=u} a(h)^*a(h)$ converges unconditionally in $E_u$. Then
there exists a bounded c.p. map $\psi_a: C^*_r(E)\to C^*(E)$
such that $\|\psi_a\|\leq \|a_u\|$ and $\psi_a\circ\lambda_E(b)=\sum_{r(h)=u}	a(gh)^*ba(h)$, for $g\in G_u$ and $b\in E_g$.
\end{lemma}
\begin{proof}
Think of $C^*(E)$ as a subalgebra of $\mathbb B(H)$, for some Hilbert space $H$ and consider the corresponding representation $\pi$ of $E$ on $H$. For $u\in \Gz$, let ${\rm Ind}\delta_u$ be the regular representation of $G$ on $\ell^2(G_u)$ \cite[page 18]{w}, where $\delta_u$ is the Dirac measure at $u$ on $\Gz$. Consider the  representation $\pi$ of $E$  on the Hilbert bundle $(H\otimes \ell^2(G_u))_{u\in\Gz}$, defined  by $$\pi(b_g) = b_g\otimes {\rm Ind}\delta_{s(g)}(1_g),\ \ (g\in G, b_g\in E_g).$$ Let $B$ be the $C^*$-algebra
 generated by the range of $\pi$. 
Consider the isometry $$j_g: H\to H\otimes\ell^2(G_{s(g)});\ \ \xi\mapsto \xi\otimes 1_{g}, \ \ (g\in G).$$
For $b = \sum_k b_k$  be an element in $E$ with finitely many nonzero terms $b_k\in E_k$, $k\in G$, and observe that for $\xi\in H$,  $g\in G$ and $h\in G_{s(g)}$, 
\begin{align*} 
	j^*_g\pi(b)j_h(\xi)&=j^*_g\pi(b)(\xi\otimes 1_{h})\\&=\sum_k j^*_g \big(b_k\xi\otimes {\rm Ind}\delta_{s(k)}(1_k)1_{h}\big)\\&= \sum_{s(k)=r(h)} j^*_g \big(b_k\xi\otimes 1_{kh}\big)=b_{gh^{-1}}\xi.
\end{align*}	
In particular, if $b= 0$ then $b_k=0$, for each $k$. Since the canonical regular representation $\lambda_G: C^*(G)\to C^*_r(G)$ is the direct sum of representations ${\rm Ind}\delta_u$, for $u\in \Gz$, the above observation means that we have a graded Fell bundle $F$ with fibres $F_g:=\pi(E_g)\otimes C^*_r(G)^{\delta_G^n}_g$, for $g\in G$. Let us observe that there is a  conditional expectation $\mathbb E_u: B\to B_u$, vanishing on each $B_g$, for $g\in G\backslash\Gz$, that is,  we have a topological grading.  Indeed, if we define, $\mathbb E_u(\pi(b)):=j_u^*\pi(b)j_u\otimes {\rm Ind}\delta_u(1_u)$, for $b\in B$, then, for $g\in G\backslash\Gz$, $j_u^*\pi(b_g)j_u=0$, by the above calculation, as required. Next, let us observe that $\mathbb E:=\oplus \mathbb E_u$ is faithful.  
If $\mathbb E_u(\pi(b^*b))=0$, for each $u\in\Gz$, then $\pi(b)j_u=0$, for each $u\in \Gz$. Consider the representation of $G$ given by,
$$\rho_t: H\otimes\ell^2(G_{s(t)})\to H\otimes\ell^2(G_{r(t)});\ \ \xi\otimes 1_s\mapsto \xi\otimes 1_{st^{-1}},\ \ (t\in G),$$
then, since $\pi$ acts on the second leg by a left multiplication, where as $\rho$ acts by a right multiplication, we have, 
$\pi(b)\rho_t=\sum_k b_k\xi\otimes 1_{kst^{-1}}=\rho_t \pi(b),$ for each $b=\sum_k b_k$ and $t\in G$. Also, $\rho_tj_s=j_{st^{-1}},$ whenever $s(s)=s(t)$. These together imply that, 
$$\pi(b)j_{t^{-1}}=\pi(b)\rho_tj_{s(t)}=\rho_t\pi(b)j_{s(t)}=0,$$
thus, $b_t=j_{r(t)}\pi(b)j_{t^{-1}}=0$, for each $t$, that is, $b=0$, as claimed. We conclude, by Lemma \ref{tg}, that $B$ is isomorphic with $C^*_r(E)$, with a canonical isomorphism $\Phi: C^*_r(E)\to B$, mapping $\lambda_E(b_g)$ to $\pi(b_g)$, for each $g\in G$ and $b_g\in E_g$.  
  
Next, let $u\in \Gz$ and $a: G\to E_u$ be the map satisfying our unconditional convergence assumption. Consider the operator $V_u: H\to H\otimes \ell^2(G_u); \ \xi\mapsto \sum_{r(h)=u} a(h)\xi\otimes 1_h$. The sum on the right hand side is convergent since,
$$\sum_{r(h)=u}\|a(h)\xi\|^2\leq \sum_{r(h)=u} \langle a(h)^*a(h)\xi,\xi\rangle\leq \big\|\sum_{r(h)=u} a(h)^*a(h)\big\|\|\xi\|^2,$$
which also shows that $\|V_u\|^2\leq \|a_u\|.$ Put $\psi_u={\rm ad}_{V^*_u}: \mathbb B(H\otimes\ell^2(G_u))\to \mathbb B(H)$. Then,
\begin{align*}
\psi_u\circ\pi(b_g)\xi&=V^*_u\big(b_g\otimes{\rm Ind}\delta_u(1_g)\big)V_u\xi\\&=V^*_u\big(b_g\otimes{\rm Ind}\delta_u(1_g)\big)\big(\sum_{r(h)=u} a(h)\xi\otimes 1_h\big)\\&=V^*_u\big(\sum_{r(h)=u} b_ga(h)\xi\otimes 1_{gh}\big)\\&=\sum_{r(h)=u} a(gh)^*b_ga(h)\xi,
\end{align*}
for each $g\in G_u$. In particular, $\psi_u$ maps $B$ into $C^*(E)$. Put $\psi_a:=\psi_u\circ \Phi$. This is a c.p. map, and for $b_g\in E_g$, 
\begin{align*}
\psi_a(\lambda_E(b_g))&=\psi_u(\pi(b_g))=\sum_{r(h)=u} a(gh)^*b_ga(h).	
\end{align*}  
Finally, $\|\psi_a\|\leq \|\psi_u\|\leq \|V_u\|^2\leq \|a_u\|,$ as required. 
\end{proof}

Recall that a discrete groupoid $G$ is amenable if there is a bounded net $(f_i)$ in $C_c(G)$ with $$\sup_{u\in\Gz}\sum_{r(h)=u}|f_i(h)|^2<\infty, \ \ \big|1-\sum_{r(h)=u}f_i(h)\bar f_i(gh)\big|\to 0, \ \ (u\in\Gz, g\in G_u),$$
as $i\to \infty$ \cite[Definition 9.3]{w}.

\begin{lemma} \label{ap-am}
For a discrete groupoid $G$,

$(i)$ each Fell bundle $E$ over $G$ with approximation property is amenable,

$(ii)$ if $G$ is amenable, each Fell bundle $E$ over $G$ has the approximation property.
\end{lemma}
\begin{proof}
$(i)$ Let $(a^u_i)$ be the net of maps as in Definition \ref{ap}, which could be assumed to be of finite support by the calculations in  paragraph after that definition. Let $\psi_i:=\psi_{a^u_i}: C^*_r(E)\to C^*(E)$ be the map constructed in Lemma \ref{map}. Put $\phi_i:=\psi_i\circ\lambda_E$. Then for $g\in G$ and $b_g\in E_g$, then identifying $b_g$ with $\pi(b_g)$, we have, 
$$\phi_i(b_g)=\sum_{r(h)=u} a^u_i(gh)^*b_ga^u_i(h)\to b_g,$$   
as $i\to \infty$. Since the net $(\phi_i)$ is uniformly bounded by Lemma \ref{map}, it follows that $\phi_i(x)\to x$, as $i\to \infty$, for each $x\in C^*(E)$. We conclude that $\lambda_E$ is injective, since if $\lambda_E(x)=0$, then $\psi_i(\lambda_E(x))=\phi_i(x)\to x$, as $i\to \infty$, thus $x=0$.

$(ii)$ Let $G$ be amenable and choose a net $(f_i)$ in $C_c(G)$ as above. Let $u\in \Gz$ and choose an approximate identity $(e^u_j)$ in $E_u$ consisting of positive elements of norm at most one, and consider the map, $$a^u_{i,j}: G\to E_u; \ \ a^u_{i,j}(g):=f_i(g)e_j,\ (g\in G).$$ Then
$$\big\|\sum_{r(h)=u}	a^u_{i,j}(h)^*a^u_{i,j}(h)\big\|\leq \sup_{u\in\Gz}\sum_{r(h)=u}|f_i(h)|^2<\infty,$$
and,
\begin{align*}
\big\|\sum_{r(h)=u}	a^u_{i,j}(gh)^*ba^u_{i,j}(h)-b\big\|&= 
\big\|\sum_{r(h)=u}	f_i(h)\bar f_i(gh)e_{j}be_{j}-b\big\|\\&\leq \big|1-\sum_{r(h)=u}f_i(h)\bar f_i(gh)\big|\|b\|+\|e_jbe_j-b\|\to 0,
\end{align*}
as $(i,j)\to\infty$, for each $g\in G_u, b\in E_g.$
 
\end{proof}

\begin{cor}[The gauge-invariant uniqueness theorem]\label{cor:guit}
Let $G$ be discrete and $X$ be 
compactly aligned.  Suppose that either  the left action on
each fibre is injective, or  $P$ is directed and $X$ is
$\tilde\phi$-injective. The following statements are equivalent:
\begin{enumerate}
\item The coaction $\delta$ is normal and injective. \label{item:6}
\item The Fell bundle
  $\bigl((\NO{X})^{\delta}_g\times\{g\}\bigr)_{g\in G}$ is
  amenable. \label{item:7}
\item The $*$-homomorphism
   $\lambda_{r}
   \colon \NO{X}\to\NO{X}^\reduced$
  is an isomorphism. \label{item:8}
\item $\NO{X}$ has the gauge-invariant uniqueness property. \label{item:9}
\item If $\psi_1:X\to D_1$ and $\psi_2:X\to D_2$ are two injective gauge-compatible
CNP-covariant representations of $X$ whose images generate $D_1$ and $D_2$
respectively, then there exists a covariant $*$-isomorphism
$\phi:D_1\to D_2$ such that $\phi\circ\psi_1=\psi_2$. \label{item:10}
\end{enumerate}
\end{cor}
\begin{proof}
(\ref{item:6})$\Leftrightarrow$(\ref{item:8}). This follows
from the fact that $\ker(\delta^n)=\ker(\lambda_r)$, as seen in  the paragraph after Lemma \ref{nondeg2}.

(\ref{item:7})$\Leftrightarrow$(\ref{item:8}). This follows from the definition
of amenability for Fell bundles over groupoids and the fact
that $\NO{X}$ is isomorphic to $C^*(E_\mathcal{O})$, as observed in the paragraph after Lemma \ref{equi}. 

(\ref{item:8})$\Rightarrow$(\ref{item:9}) Suppose first that (\ref{item:8}) holds. Let $\phi:\NO{X}\to D$
be a $*$-epimorphism.  We must show that  $\phi$ is
injective if and only if conditions (GI1)~and~(GI2) hold. If $\phi$ is
injective, then one could define $\beta$ by $\beta := (\phi
\otimes \id) \circ \delta \circ \phi^{-1}$, and observe that 
condition~(GI1) is satisfied. Now by Theorem \ref{thm:projective property}, $(\phi\circ j_X)|_A$ is
injective and condition~(GI2) is also satisfied.
Conversely, suppose that conditions (GI1)~and~(GI2) hold. Then $\phi\circ j_X$ is a  gauge-compatible Nica
covariant representation whose image generates $D$, and again by
Theorem~\ref{thm:projective property}, there is a
$*$-homomorphism $\tilde\phi:D\to \NO{X}^\reduced$ satisfying
$\tilde\phi\circ(\phi\circ j_X)=j_X^\reduced
=\lambda_{r}\circ j_X$.  Since $\NO{X}$ is generated
by the elements $\{\, j_X(x) \mid x \in X \,\}$ and
$\lambda_{r}$ is an isomorphism by assumption, 
$\lambda_r^{-1}\circ \tilde\phi$ is an inverse
for $\phi$, and  $\phi$ is injective, as required.

(\ref{item:9})$\Rightarrow$(\ref{item:10}). If $\psi:X\to D$ is an
injective gauge-compatible CNP-covariant representation of $X$
whose image generates $D$, then $\Pi\psi:\NO{X}\to D$ is an epimorphism such that $\Pi\psi\circ j_X=\psi$
and (GI1)~and~(GI2) are
satisfied, hence $\Pi\psi$ is an isomorphism.

(\ref{item:10})$\Rightarrow$(\ref{item:8}). The canonical maps
$j_X^\reduced:X\to\NO{X}^\reduced$ are two injective
gauge-compatible CNP-covariant representations whose images
generate $\NO{X}$ and $\NO{X}^\reduced$, respectively. Thus
there exists a  $*$-isomorphism $\phi:\NO{X}^\reduced\to\NO{X}$
such that $\phi\circ j_X^\reduced=j_X$, therefore,  $\phi$ is a right inverse of  
$\lambda_r$. Since $\phi$ is  surjective, $\lambda_r$ is injective, as required. 
\end{proof}

We give sufficient conditions under which 
$\NO{X}$ has the gauge-invariant uniqueness property. In the following result, the coactions of $G$ on $\Tt_{X}$ and $\NO{X}$ are different, but with a slight abuse of notation, we have denoted both by $\delta$.

\begin{cor} \label{cor:amenable}
Let $G$ be discrete and $X$ be 
compactly aligned.  Suppose that either  the left action on
each fibre is injective, or  $P$ is directed and $X$ is
$\tilde\phi$-injective. Then $\NO{X}$ has the gauge-invariant
uniqueness property in the following cases:
\begin{enumerate}
\item the groupoid $G$ is exact and the coaction $\delta$ of $G$ on
  $\Tc(X)$ is normal, and $\ker(q_{\rm CNP})$ is strongly induced with respect to $\delta$,\label{item:1}
\item the Fell bundle $(\Tc(X)^\delta_g\times\{g\})_{g\in G}$ has the
approximation property, \label{item:2}
\item the Fell bundle
  $\bigl((\NO{X})^{\delta}_g\times\{g\}\bigr)_{g\in G}$ has the
  approximation property, \label{item:5}
\item the groupoid $G$ is  amenable. \label{item:3}
\end{enumerate}
\end{cor}

\begin{proof}
Case~(\ref{item:1}) follows from Corollary~\ref{cor:guit}
and Lemma \ref{prop:exact}.

For case~(\ref{item:2}), put $J=\ker (q_{\CNP})$ and for $u\in \Gz$ let
$\Psi_u \colon \Tc(X)\to\Tc(X)^\delta_u$ and $\Phi_u \colon \NO{X}\to (\NO{X}^{\delta})_u$ be the induced
conditional expectations, and put $\Psi:=\oplus_u \Psi_u$ and $\Phi:=\oplus_u\Phi_u$. Let $\lambda: \NO{X}\to\NO{X}^r$ be the corresponding regular representation. Let $\Phi_r: \NO{X}^r\to \bigoplus_{u\in\Gz} (\NO{X}^{\delta})_u$ be the canonical faithful conditional expectation  \cite[Propositions 3.6, 3.10]{kum}, then $\oplus \Phi_u=\Phi_r\circ \lambda$, thus
$\ker(\lambda)=\{\, b\in \NO{X} \mid
\Phi(b^*b)=0 \,\}$,  and likewise, for $(j_{X}^\reduced)_*:=j_{X}^\reduced\circ i_X$,  
$\ker((j_{X}^\reduced)_*)=q_{\CNP}^{-1}(\ker(\lambda))=
\{\, c\in\Tc(X) \mid \Psi(c^*c)\in J \,\}$. In particular, $\ker((j_{X}^\reduced)_*)\supseteq \ker(q_{\CNP})$. Let us show that, if
the Fell bundle $(\Tc(X)^\delta_g\times\{g\})_{g\in
G}$ has the approximation property, then 
$\ker((j_{X}^\reduced)_*)\subseteq\ker(q_{\CNP})$. Let $b\in \ker(j_X^\reduced)$, then $\Psi_g(b)\in J\cap \Tc(X)^\delta_g$, for each $g\in G$. Thus, $\Psi_g(b)^*\Psi_g(b)\in J\cap \Tc(X)^\delta_{r(g)}$, and so
$$\Psi_g(b)=\lim_n\Psi_g(b)(\Psi_g(b)^*\Psi_g(b))^{1/n}\in J.$$ 
For the net $(\phi_i)$ as in the proof of Lemma \ref{ap-am}$(i)$, we have,
$$\phi_i(\Phi_g(b))=\sum_{r(h)=u} a^u_i(gh)^*\Phi_g(b)a^u_i(h)\in J,$$   
therefore, $$b=\lim_i\phi_i(b)=\lim_i\phi_i\big(\sum_{g\in G}\phi_i(\Phi_g(b))\big)\in J,$$
as claimed. We have shown that, $\ker((j_{X}^\reduced)_*)=\ker(q_{\CNP})$, which in turn implies that $\lambda$ is an isomorphism, and
 ~(\ref{item:2}) follows from
Corollary~\ref{cor:guit}. 

In case (\ref{item:5}), if the Fell bundle $\bigl((\NO{X})^{\delta}_g\times\{g\}\bigr)_{g\in G}$ has
the approximation property, then it is amenable by
Lemma \ref{ap-am}$(i)$, hence the result follows from
Corollary~\ref{cor:guit}. 

Finally, if $G$ is amenable, then by Lemma \ref{ap-am}$(ii)$, the bundle $\bigl((\NO{X})^{\delta}_g\times\{g\}\bigr)_{g\in G}$
has the approximation property, so case (\ref{item:3}) follows
from case (\ref{item:5}).
\end{proof}

\section{examples} \label{example}

We have constructed the full and reduced $C^*$-algebras $\NO{X}$ and $\NO{X}^r$ of a product system over a quasi-lattice ordered groupoid $(G,P)$. In this section, we examine this construction for certain particular systems over concrete groupoids. In particular,  for a natural extension of the Katsura construction of $C^*$-algebras of a $C^*$-correspondence, we illustrate how the guage invariant uniqueness theorem is applied. We also show how to use our construction to associate a Cuntz-Nica-Pimsner $C^*$-algebra to  a multi-layered topological $k$-graph.  

Our strategy to build effective examples of quasi-lattice ordered (discrete) groupoids is to start with a quasi-lattice ordered group $(H,Q)$ and use a ``divider'' $d: Q\to \mathbb Z_2$ to divide the elements of $Q$ into two distinct classes, and for $P=\mathbb Z_2\times Q$, define the partially defined multiplication map using $d$. We then get the groupoid $G$ and its unit space $\Gz$ formally as $G=P\cup P^{-1}$ and $\Gz=P\cap P^{-1}$. The idea is simple. Let $d: Q\to \mathbb Z_2$ be a semigroup homomorphism. For $P=\mathbb Z_2\times Q$, put 
$$P^{(2)} :=\{\big((i,t), (i+d(t),s)\big)\in P\times P: i\in \mathbb Z_2, s,t\in Q\},$$
where the sum $i+d(t)$ is  (mod $2$), and define the multiplication by,
$$(i,t)(i+d(t),s)=(i,st),\ \ (i\in \mathbb Z_2, s,t\in Q).$$
Now we assume that $Q\cup Q^{-1}=H$ and $Q\cap Q^{-1}=\{e\}$, where $e$ is the identity element of $H$ and define $(i,t)^{-1}=(i+d(t), t^{-1})$, and put $G:=P\cup P^{-1}$. Then $r(i,t)=(i, e)$ and $s(i,t)=(i+d(t),e)$. Thus, $P\cap P^{-1}=\Gz=\mathbb Z_2\times \{e\}$. We extend $d$ to a group homomorphism $d: H\to \mathbb Z_2$, by putting $d(t^{-1}):=d(t)$, then $G=\mathbb Z_2\times H$ is a groupoid with $G^{(2)}$ and the multiplication rule defined as above. We write $G=\mathbb Z_2\rtimes_d H$ and $P=\mathbb Z_2\rtimes_d Q$. Let us list the properties of $G$ in the next lemma. Recall that a groupoid $G$ is called transitive if for each $u,v\in \Gz$, there is $g\in G$ with $s(g)=u$ and $r(g)=v$.

\begin{lemma} \label{ex}
	Assume that $(H,Q)$ is a quasi-lattice ordered group with $Q\cup Q^{-1}=H$. Then,

$(i)$ $(G,P):=(\mathbb Z_2\rtimes_d H, \mathbb Z_2\rtimes_d Q)$ is a quasi-lattice ordered groupoid with $P\cup P^{-1}=G$,

$(ii)$ if $Q$ is directed (resp., totally ordered), so is $P$,

$(iii)$ $G:=\mathbb Z_2\rtimes_d H$ is transitive if $d$ is a epimorphism, and a group bundle otherwise,

$(iv)$ if $H$ is countable, then $G:=\mathbb Z_2\rtimes_d H$ is amenable iff $H$  is amenable,

$(v)$ if $H$ is countable, then $G:=\mathbb Z_2\rtimes_d H$ is exact iff $H$ is exact. 
\end{lemma}
\begin{proof}
$(i)$ Recall that for $s,t\in Q$, $s\leq t$ iff $s^{-1}t\in Q$. We  define, $(i,s)\leq (j,t)$ iff $i=j$ and $(i,s)^{-1}(j,t)\in P$, which simply means that $i=j$ and $s^{-1}t\in Q$. By assumption, a set $\{s,t\}$ with an upper bound in $Q$ has a least upper bound $s\vee t$ in $Q$. Now if $\{(i,s), (j,t)\}$ has an upper bound $(k,u)$ in $P$, then $u$ is an upper bound for $\{s,t\}$ and likewise, $(i+d(s), s\vee t)$ is a least upper bound for  $\{(i,s), (j,t)\}$ in $P$. The converse implication is proved similarly. 

$(ii)$ This follows as in the proof of $(i)$.

$(iii)$ Note that $\Gz$ consists of two points $(0,e)$ and $(1,e)$. If $d$ is surjective, one could choose $t\in Q$ with $d(t)=1$, and observe that $s(0,t)=(0,e)$ and $r(0,t)=(0+d(t),e)=(1,e)$, and $G$ is transitive. If not, $d(t)=0$, for each $t\in H$, and $s(i,t)=r(i,t)=(i,e)$, for each $(i,t)\in G$, and $G$ is a group bundle. 

$(iv)$ By $(iii)$ and \cite[Corollary 10.54]{w}, $G$ is amenable iff its isotropy groups are all amenable. The result follows, as in this case, there are just two isotropy groups, namely $\{i\}\times H_0$, for $i=0,1$, for $H_0:=\ker(d)$. Since $H_0\unlhd H$ with $H/H_0\simeq\mathbb Z_2$, $H_0$ is amenable iff $H$ is amenable. 

$(v)$ By the proof of $(iv)$ and \cite[Theorem 2.78]{w}, and the fact that $\Gz$ consists of two points, we have $$C^*(G)\simeq C^*(H_0)\otimes \mathbb K(\ell^2(G^{0}))\simeq \mathbb M_2(C^*(H_0)).$$
Similarly, it follows from \cite[Lemma 5.32]{w} that,
$$C^*_r(G)\simeq C^*_r(H_0)\otimes \mathbb K(\ell^2(G^{0}))\simeq \mathbb M_2(C^*_r(H_0)).$$
Now if $H_0$ is exact, so is $G$ by \cite[Proposition 10.2.7]{bo}. The converse is also true by \cite[Proposition 10.2.3]{bo}.  On the other hand, let us observe that the exact sequence
$$0\to C^*_r(H_0)\to C^*_r(H)\to C^*_r(\mathbb Z_2)\to 0$$
is locally split, in the sense of Kirchberg \cite[Definition 2.1.5]{bo}. Since $C^*_r(\mathbb Z_2)$ is nuclear, the quotient map $C^*_r(H)\to C^*_r(\mathbb Z_2)$ is nuclear by \cite[Exercise 2.1.4]{bo}. It follows now from Choi-Effros lifting Theorem \cite[Theorem C.3]{bo} that this quotient map is liftable, as claimed. Now, since the above exact sequence is locally split, by \cite[Theorem 10.2.4]{bo}, $C_r(H)$ is exact iff $C^*_r(H_0)$ is exact. Summing up, $G$ is eact iff $H$ is so.          
\end{proof}

As our main concrete examples, let us see how this works for $(H,Q):=(\ZZ, \ZZ^+)$ and $(H,Q):=(\mathbb F_2, \mathbb F_2^{+})$.

\begin{example} \label{Z}
For $\ZZ^+:=\NN\cup\{0\}$ and $(H,Q):=(\ZZ, \ZZ^+)$, let us consider the divider 
$$d: \ZZ\to \ZZ_2; \ \ d(n)=[n]_2:=n{\rm (mod \ 2)}.$$
Then $(G,P):=(\ZZ_2\rtimes_d\ZZ, \ZZ_2\rtimes_d\ZZ^+)$ is a quasi-lattice ordered groupoid with $\Gz$ consisting of two points $(0,0)$ and $(1,0)$. Take two $C^*$-algebras $A$ and $B$, and $A$-$B$-correspondence ${}_AX_B^1$ and $B$-$A$-correspondence ${}_BX_A^2$. For $k\in\ZZ^+$ and $p:=(0,2k)\in P$, let us put, $$X_p:=X^1\otimes_BX^2\otimes_A\cdots\otimes_BX^2,$$ with $2k$-terms, with $X^1$ and $X^2$ sitting respectively in the odd and even positions, regarded as an $A$-$A$-correspondence, with the convention that for $k=0$, $X_{(0,0)}:=A$. For $p:=(0, 2k+1),$ $X_p$ is defined by a tensor product as above with $(2k+1)$-terms, starting and ending with $X^1$, regarded as an $A$-$B$-correspondence. The correspondences $X_p$, for $p:=(1,2k)$ or $(1, 2k+1)$ are defined likewise, with the first term being $X^2$ instead of $X^1$, with the convention that for $k=0$, $X_{(1,0)}:=B$. For $A_{(0,0)}:=A$ and $A_{(1,0)}:=B$, we then have,
$$X_{(i,n)}\otimes_{A_{([i+n]_2,0)}}X_{([i+n]_2,m)}=X_{(i,n+m)}.$$ 
By Lemma \ref{ex}$(ii)$ and observation of the last paragraph of Subsection \ref{subsec_prod_syst}, this system is compactly aligned. In particular, we could construct the corresponding full and reduced Cuntz-Nica-Pimsner $C^*$-algebras $\NO{X}$ and $\NO{X}^r$. Indeed, by Lemma \ref{ex}$(iv)$, Lemma \ref{ap-am}, and Corollary \ref{cor:guit}, they are the same in this case. It is illustrative to see that $\NO{X}$ could also be constructed directly by considering covariant representations of $X^1$ and $X^2$ in the sense of Katsura. Part of this is performed in Appndix \ref{A.1}, where a universal $C^*$-algebra $\NO{{X^1, X^2}}$ is constructed using Katsura method.  Now we just need to observe that there is a one-one correspondence between covariant representations $(\pi_A, \pi_B, t_{X^1}, t_{X^2})$ in the sense of Appendix \ref{A.1} and Nica-covariant representations of the above product system $X$ over $(\ZZ_2\rtimes_d\ZZ, \ZZ_2\rtimes_d\ZZ^+)$ in the sense of Section \ref{subsection:reps_prod_syst}.                 
\end{example}

It is natural to ask if a given product system $X$ over a  quasi-lattice ordered groupoid $(G,P)$ admits a couniversal triple $(C,\rho,\gamma)$ for injective (resp., Nica covariant) gauge-compatible
Toeplitz representations in the sense of Theorem \ref{thm:projective property}. The answer is negative in general, when we drop assumptions of the theorem, or keep the assumptions and ask for couniversal objects for injective non Nica covariant gauge-compatible
Toeplitz representations. The next lemma addresses the problem for product systems over group bundles. 

\begin{lemma} \label{no}
Let $X$ be a product system $X$ over a  quasi-lattice ordered groupoid $(G,P)$, where $G=\bigcup_{u\in\Gz} G^u_u$ is a group bundle. Let $X(u)$ be the product subsystem over  the quasi-lattice ordered group $(G^u_u,P^u_u)$. Then $X$ admits a couniversal triple $(C,\rho,\gamma)$ for injective (resp., Nica covariant) gauge-compatible
Toeplitz representations iff all subsystems are so. 
\end{lemma}
\begin{proof}
Let a triple
$(C, \rho, \gamma)$ be
couniversal for gauge-compatible injective (resp., Nica covariant)
representations of $X$, that is,

(1) $\rho: X\to C$ is an injective
	(resp., Nica covariant) representation of $X$ whose image
	generates $C$, and $\gamma$ is a normal coaction of $G$ on $C$
	with $\gamma(\rho(x))=\rho(x)\otimes i_G(d(x))$, for  $x\in X$,

(2) if $\psi \colon X \to D$ is an
	injective (resp., Nica covariant) gauge-compatible representation whose image
	generates $D$, then there is a
	surjective $*$-homomorphism $\phi \colon D \to
	C$ such that $\phi(\psi(x)) =
	\rho(x)$, for $x \in X$.
	
For $u\in\Gz$, let $\rho(u)$ be the restriction of $\rho$ to $X(u)$ and $C_u$ be the $C^*$-subalgebra of $C$ generated by the range of $\rho_u$. Let $\gamma_u$ be the restriction of $\gamma$ to $C_u$, then since for $x\in X(u)$,
$$\gamma_u(\rho_u(x))=\gamma(\rho(x))=\rho(x)\otimes i_G(d(x))=\rho_u(x)\otimes i_{G^u_u}(d(x))\in C_u\otimes C^*(G_u^u), $$
$\gamma_u$ is a coaction of $G^u_u$ on $C_u$, whose injectivity and nondegeneracy follow from that of $\gamma$. Finally, given 
injective (resp., Nica covariant) gauge-compatible representations $\psi_u \colon X(u) \to D_u$ whose image
generates $D_u$, we faithfully represent all $D_n$'s in the same Hilbert space $H$ and let $D$ be the $C^*$-subalgebra of $\mathscr{L}(H)$ generated by these subalgebras, and let  $\psi:\colon X \to D$ be the map whose restriction to $X(u)$ is $\psi_u$. This is well-defined as $X(u)$'s are disjoint, and an injective (resp., Nica covariant) gauge-compatible representation, as each $\psi_u$ is so. Finally, the  image of $psi$ generates $D$. For the surjective $*$-homomorphism $\phi \colon D \to
C$ such that $\phi(\psi(x)) =
\rho(x)$, for $x \in X$, let $\phi_u: D_u\to C$ be the restriction of $\phi$ to $D_u$, then $\phi_u\circ\psi_u=\rho_u$, and so the range of $\phi_u$ is $C_u$. Summing up, the triple
$(C, \rho, \gamma)$ is
couniversal for gauge-compatible injective (resp., Nica covariant)
representations of $X(u)$. The converse implication is proved similarly.    
\end{proof}

We construct an example of a product system $X$ in which the left
actions are not injective, and $P$ is not directed, and the
conclusion of Lemmas~\ref{lemma:inj} and
~\ref{thm:inj on core} both fail in this example. 
We also use the above lemma to construct a product system
which admits no couniversal $C^*$-algebra for injective (not necessarily Nica covariant) gauge-compatible
Toeplitz representations.

\begin{example} 
	Consider the free group $\FF_2$ on two generators $a,b$ and identity element $e$. Let $(H,Q):=(\FF_2,\FF_2^+)$, where $\FF_2^+$ is the free semigroup generated by $a$ and $b$. Consider the canonical length function $\ell: \FF_2\to \ZZ^+$, assigning its length to each word. Define the divider $d: \FF_2\to \ZZ_2$ by $d(x):=[\ell(x)]_2$, for $x\in \FF_2$, which maps the words of odd length to 1 and the rest  to $0$. The kernel of $d$ is the normal subgroup $\FF_2^{\rm even}$ consisting of words of even length. 
	
	For the quasi-lattice ordered group  $(G,P):=(\ZZ_2\rtimes_d\FF_2,\ZZ_2\rtimes_d\FF_2^+)$, define a
	product system over $P=\ZZ_2\rtimes_d\FF_2^+$ by $X_{(i,a^n)}=\CC$, for $i\in\ZZ_2$ and $n\in \NN$
	and $X_p=0$, for all other elements of $P$. This system is clearly
	compactly aligned,
	but the left actions are not all
	injective, since for instance  $(0,a)\vee (0, b)=\infty$. Consider the injective Nica covariant Toeplitz representation $\psi \colon X\to \CC$, defined by $\psi_p(x)=x$ for $p\in P$ and $x\in X_p$. 
	Let $1_p$ be the identity in $\mathscr L_{A_{r(p)}}(X_p)$, for $p\in P$, and let us observe
	that $1_p\in \mathscr K_{A_{r(p)}}(X_p)$, for $p=(i,e)$ or $(i,a)$, for each $i\in\ZZ_2$, and 
	$$\psi^{(i,e)}(1_{(i,e)})-\psi^{(i,a)}(1_{(i,a)})=1-1=0, $$ while,
	$$\iotale^{(i,e)}_{(i,t)}(1_{(i,e)})-\iotale^{(i,a)}_{(i,t)}(1_{(i,a)})\neq 0,$$ for
	large $t$.  Indeed, for for $q\in P\backslash\Gz$,
	$
	I_q=0$  if $q=(i,a^n)$, for some $i\in\ZZ_2$ and $n\in \NN$, and $I_q=\CC$, otherwise. If $q\geq b$, then
	$X_{(i,e)}\cdot I_q=X_{(i,e)}$, for each $i\in\ZZ_2$, and, $$\iota_{(i,e)}^{(i,e)}(1_{(i,e)})-\iota_{(i,a)}^{(i,e)}(1_{(i,a)})=
	\iota_{(i,e)}^{(i,e)}(1_{(i,e)})\not= 0,
	$$
	 that is, $$\iotale^{(i,e)}_{(i,t)}(1_{(i,e)})-\iotale^{(i,a)}_{(i,t)}(1_{(i,a)})\not= 0,$$ for
	 $t \ge b$.
	
	Alternatively, let $(G,P):=(\ZZ_2\rtimes_{d_0}\FF_2, \ZZ_2\rtimes_{d_0}\FF_2^+)$ for non-surjective divider $d_0: \FF_2\to \ZZ_2$, which sends all elements to 0.  Let $X_p=\CC$, for every $p\in
	P$. Then again $X$ is compactly aligned,
	but this time  all the left actions are injective. Denote the generators
	of $\FF_2$ by $a$ and $b$. We claim that $X$ admits no couniversal $C^*$-algebra for injective (not necessarily Nica covariant) gauge-compatible
	Toeplitz representations. By Lemmas \ref{ex}$(iii)$ and \ref{no}, we need to check this for the restriction $Y$ of $X$ to the isotropy group $(H,Q):=\big(\FF_2^{\rm even}, (\FF_2^{\rm even})^+\big).$ Suppose that $(C,\rho,\gamma)$ is a
	couniversal triple for injective gauge-compatible Toeplitz representations of $Y$.
	Consider the
	injective Toeplitz representation $\psi$ of $Y$ on $C^*_\reduced(\FF_2^{\rm even})$
	determined by $\psi(1_q) := \lambda_{\FF_2^{\rm even}}(q)$, for $q\in Q$. Similar to \cite[Example 1.15]{qui:full reduced}, one could observe that there is a (full) coaction
	$\delta_{\FF_2^{\rm even}}$ of $\FF_2^{\rm even}$ on $C^*_\reduced(\FF_2^{\rm even})$ satisfying,
	$$\delta_{\FF_2^{\rm even}}(\lambda_{\FF_2^{\rm even}}(h)) = \lambda_{\FF_2^{\rm even}}(h) \otimes i_{\FF_2^{\rm even}}(h),\ \ (h \in H:=\FF_2^{\rm even}).$$ In particular,  $\psi$ is 
	gauge-compatible. By the couniversal property, there is a surjective homomorphism $\phi \colon
	C^*_\reduced (\FF_2^{\rm even})\to C$, satisfying $\phi(\psi(1_q)) =
	\rho(1_q)$, for  $q \in Q$. Since
	$\lambda_{\FF_2}(a)$ and $\lambda_{\FF_2}(b)$ are unitaries,
	by surjectivity of $\phi$,
	\begin{align*}\label{eq:prod=1}
		\rho(1_{aa})\rho(1_{aa})^* \rho(1_{bb})\rho(1_{bb})^*
		&= \phi\big(\lambda_{\FF_2^{\rm even}}(aa)\lambda_{\FF_2^{\rm even}}(aa)^*\lambda_{\FF_2^{\rm even}}(bb)\lambda_{\FF_2^{\rm even}}(bb)^*\big)
		\\&= \phi(1_{C^*_\reduced(\FF_2^{\rm even})})
		= 1_C.
	\end{align*}
	Now, since $j_Y^\reduced$ is  an injective gauge-compatible
	Toeplitz representation of $Y$, again by the couniversal property, there is a surjective homomorphism $\psi_r \colon
	\NO{Y}^\reduced \to C$ such that $\eta(j_Y^\reduced(1_h)) =
	\rho(1_h)$, for $h \in H$. But $j_Y^\reduced$ is Nica
	covariant, and  $aa \vee bb = \infty$ in $\FF_2^{\rm even}$, therefore,
	\begin{equation}\label{eq:prod=0}
		\rho(1_{aa})\rho(1_{aa})^* \rho(1_{bb})\rho(1_{bb})^*
		= \psi_r\big(j_X^\reduced(1_{aa})j_X^\reduced(1_{aa})^*j_X^\reduced(1_{bb})j_X^\reduced(1_{bb})^*\big)
		= 0,
	\end{equation}
	which in turn implies that 
	$1_C = 0$, which contradicts the assumption
	that $\rho$ is injective, proving the claim.
Note that although the product system $X$ in this example does not satisfy the
assumptions of Theorem~\ref{thm:inj on core}, nevertheless it
admits a couniversal algebra which is a proper quotient of  $\NO{X}^\reduced$. Also,
it is not difficult to see that every Toeplitz representation of the system $X$ above is automatically Nica covariant, and 
there is a bijective correspondence between (resp., 
gauge-compatible injective) Toeplitz representations
of $X$ and (resp., 
gauge-compatible injective) Toeplitz representations of the subsystem generated by $X_{(0,a)}$ and $X_{(1,a)}$. 
\end{example}

Next, we give examples of groupoids with  uniformly bounded fibres in the sense of Oty \cite{o}, as promised. It is worth mentioning the motivation behind considering this class of groupoids: for a locally compact groupoid $G$, elements of the Fourier-Stieltjes algebra $B(G)$ are coefficient functions of representations of $G$, and thereby could be regarded as bounded linear functional on $C^*(G)$. They also decompose as combinations of four continuous positive definite functions, regarded as positive linear functionals on $C^*(G)$. If $G$ is $r$-discrete with uniformly bounded fibres, then each bounded continuous function on $G$ is a linear combination of  four continuous positive definite functions, and so is in the Fourier-Stieltjes algebra \cite[Theorem 4.10]{o}. 

\begin{example}
Let $G$ be a locally compact, second countable, Hausdorff groupoid. A Borel map $\mathcal O: G\to\ZZ^+$ is  called an ordering map if it is
 on every fiber $G^u$ of $G$, for $u\in \Gz$  \cite[Definition 2.1]{o}. Now if the fibres of $G$ have constant cardinality, the existence of  a continuous ordering map $\mathcal O$ on $G$ forces $G$ to be \'{e}tale \cite[Proposition 2.2]{o}. Examples of $r$-discrete groupoids having fibres with constant cardinality (and so with  uniformly bounded fibres) include  transformation
 groups of countable discrete groups acting on compact spaces, and 
 principal transitive groupoids on finitely many elements

\end{example} 
\section*{Appendix}

This appendix has three subsections. In the first part, we directly construct the Cuntz-Pimsner $C^*$-algebra $\mathcal O_{{X^1,X^2}}$ for a pair of $C^*$-correspondences ${}_AX^1_B$ and ${}_BX^2_A$ by extending the Katsura construction  \cite{ka2}. This is the same as the  Cuntz-Nica-Pimsner $C^*$-algebra $\NO{X}$, for the product system $X$ generated by $\{X^1, X^2\}$ over the quasi-lattice ordered groupoid  $(\ZZ_2\rtimes_d\ZZ, \ZZ_2\rtimes_d\ZZ^+)$, as observed in Example \ref{Z}. In the second part, we show how to realize groupoid crossed products as Cuntz-Nica-Pimsner $C^*$-algebras over the acting groupoid. The last part is devoted to the construction of the $C^*$-algebra of a multilayered topological $k$-graph, and showing that it could be realized as the Cuntz-Nica-Pimsner $C^*$-algebra $\NO{X}$ of a product system on some quasi-lattice ordered groupoid.

\subsection*{{\rm A.1.} $C^*$-algebra of a pair of correspondences}\label{A.1}

Let $A$ and $B$ be C$^*$ algebras. An $A$-$B$-correspondence is a right Hilbert $B$-module $_{A}X_{B}$ together with a $*$-homomorphism $\phi_X : A \rightarrow \mathscr{L}_A(X)$. We say that $X$ is full if $\langle X,X\rangle_B$ is a total set in $B$.  
We usually drop the index $A$ and use $\mathscr{L}(X)$, keeping in mind that adjointable operators in this subsection are always considered as module maps for the left module action. We also freely use standard notations of Hilbert modules, such as $\theta_{\xi,\eta}\in\mathscr{K}(X)$, for the rank-one operator associated to $\xi,\eta\in X$.

In this subsection, we construct the Cuntz-Pimsner $C^*$-algebra $\mathcal O_{{X^1,X^2}}$ for a pair of $C^*$-correspondences ${}_AX^1_B$ and ${}_BX^2_A$, and show that it coincides with the Cuntz-Nica-Pimsner $C^*$-algebra $\NO{{X}}$, for the product system $X$ generated by $X^1$ and $X^2$. For simplicity of notations, we put $X:=X^1$ and $Y:=X^2$. 
For the correspondences  $_{A}X_{B}$ and  $_{B}Y_{A}$,
the tensor product $X \odot Y $ of $X$ and $Y$ (amalgamated over $B$) is an $A$-$A$-correspondence. Similarly, $Y \odot X $ (amalgamated over $A$) is a $B$-$B$-correspondence.
For $n \in \mathbb{N}$ we define the correspondences $\XtY{n}$ and $\YtX{n}$ inductively by $\XtY{0} := B$, $\YtX{0} := A $, $\XtY{1} := X$, $\YtX{1}:= Y$, $\XtY{n+1} := (\XtY{n}) \otimes X,\  \YtX{n+1} := (\YtX{n}) \otimes Y,$
for even $n \geq 0$, and likewise, $\XtY{n+1} := (\XtY{n}) \otimes Y,  \ \YtX{n+1} := (\YtX{n}) \otimes X,$ for odd $n \geq 1$. We write $\phi_{(X,n)}$ for $\phi_{\XtY{n}}$, and likewise  for $\phi_{(Y,n)}$.
For $m, n \in \mathbb{N}$ and  $S \in \LL(\XtY{n})$,  $S \otimes{\rm id}_m$ in $\LL(\XtY{n+m})$ is defined canonically. For each $m$  we define $\iota_{(Y,m)}^{(X,n)} \in \LL(\YtX{m},\XtY{n+m})$ by
$\iota{(Y,m)}^{(X,n)}(\xi)(\eta) := \xi \otimes \eta,$
for $\xi \in \XtY{n}, \eta \in \YtX{m}$ and odd $n$, and likewise for even $n$. Then, $\iota_{(Y,m)}^{(X,0)}  = \phi_{(Y,m)}$ and
$
\myTau{X}{n}{Y}{m}{\xi}^* \,(\xi' \otimes \eta) = \phi_{(X,n)} (\la\xi,\xi'\ra)\eta$ 
for $\eta \in \YtX{m}$, $ \xi' \in \XtY{n}$, and odd $n$. 	For each $m$ and odd $n, n_1, n_2, m_1$ we have $\phi_{(X,n+m)}(a) \ordTau{m}{n}(\xi) = \ordTau{m}{n}(\phi_{n}(a) \xi)$,
$$\ordTau{m}{n}(\xi) \ordTau{m}{n}(\eta)^* = \theta_{\xi,\eta} \otimes {\rm id}_{m}, \, \ordTau{m}{n}(\xi)^* \ordTau{m}{n}(\eta) = \phi_{(Y,m)}(\la\xi,\eta\ra_{\XtY{n}}), 
\, \ordTau{m}{n}(\xi) \phi_{(Y,m)}(v) = \ordTau{m}{n}(\xi v),$$
and $\ordTau{n_2+m}{n_1}(\xi_1) \, \ordTau{m}{n_2}(\xi_2) = \ordTau{m}{n_1+n_2}(\xi_1 \otimes \xi_2),$ for $\xi , \eta \in \XtY{n}$,  $\xi_1 \in \XtY{n_1}$, $\xi_2 \in \XtY{n_2}$, and $a\in A, b \in B$. 

\vspace{.2cm}
\indent {\bf Definition A.1.}
	A representation of $(X,Y)$ on a C$^*$-algebra $D$ is a 4-tuple consisting of two *-homomorphisms $\pi_A : A \rightarrow D$ , $\pi_B : B \rightarrow D$ and two linear maps $t_X : X \rightarrow D$ and $t_Y : Y \rightarrow D$ satisfying,
	
	$(i) \, t_X(\xi)^*t_X(\xi^{'}) = \pi_B(\la\xi,\xi^{'}\ra)$, 

	$(ii) \, t_Y(\eta)^*t_Y(\eta^{'}) = \pi_A(\la\eta,\eta^{'}\ra)$, 

	$(iii) \, \pi_A(a)t_X(\xi) = t_X(\phi_X(a)\xi)$, 

	$(iv) \, \pi_B(b)t_Y(\eta) = t_Y(\phi_Y(b)\eta)$, 
	
	$(v)$ $t_X(\xi)t_X(\xi^{'})=t_Y(\eta)t_Y(\eta^{'})=t_X(\xi)t_Y(\eta)^*=0$,
	
	$(vi)$ 	 $t_X(\xi)\pi_A(a)=t_Y(\eta)\pi_B(b)=\pi_B(b)t_X(\xi)=\pi_A(a)t_Y(\eta)=0$, 

\noindent for $\xi, \xi^{'}\in X,  \eta, \eta^{'} \in Y$, $a\in A, b\in B$. When $X$ and $Y$ are full, $(vi)$ follows from $(v)$. Here we do not impose the assumption of $X$ and $Y$ being full, which is not essential for the rest of this subsection.  

\vspace{.2cm}
 The $C^*$-algebra generated by images of $\pi_A$ , $\pi_B$ , $t_X$ and $t_Y$ is denoted by $ C^*\ordRep$. It follows that $t_X(\xi) \pi_B(b) = t_X(\xi b)$, for $\xi \in X$ and $b \in B$.
The representation is called injective if $\pi_A$ and $\pi_B$ are injective. In this case, $t_X$ and $t_Y$ are isometries. 
For a representation $\ordRep$ as above, we define a canonical $*$-homomorphism $\psi_{t_X} : \mathscr{K}(X) \rightarrow D$ by $\psi_{t_X}(\theta_{\xi,\eta}) = t(\xi)t(\eta)^* \in D $ for $\xi , \eta \in X$. 
Set $t^{(X,0)} = \pi_B$, $t^{(Y,0)} = \pi_A$, $t^{(X,1)} = t_X$, $t^{(Y,1)} = t_Y$. Define $t^{(X,n)}:  \XtY{n} \to C^*\ordRep$ by $t^{(X,n)}(\xi\otimes\eta) = t_X(\xi)t^{(Y,n-1)}(\eta)$, for $n\geq  2$, $\xi \in X$ and $\eta \in \YtX{n-1}$. Also  define $\psi_{t^{(X,n)}}:\mathscr{K}(\XtY{n}) \rightarrow C^*\ordRep$ by
$\psi_{t^{(X,n)}}(\theta_{\xi,\eta}) = t^{(X,n)}(\xi)t^{(X,n)}(\eta)^*$, for
$\xi , \eta \in \XtY{n}$.
	With the above notations,  
	$\pi_A(a)\psi_{t_X}(k) = \psi_{t_X}(\phi_X(a)k)$ and $\psi_{t_X}(k) t_X(\xi) = t_X(k\xi)$, for $a \in A \, , \xi \in X \, , k \in \mathscr{K}(X)$. Also, for $m \leq n$ and $m$  odd,  $t^{(X,m)}(\eta)^*t^{(X,n)}(\xi) = t^{(Y,n-m)}(\lambda)$, for $\xi \in \XtY{n}$ and $\eta \in \XtY{m}$, where $\lambda = \ordTau{n-m}{m}(\eta)^*(\xi) \in \YtX{n-m}$.
	
	The C$^*$-algebra $C^*\ordRep$ is the norm closure of the $*$-algebra generated by all the elements of the form $t^{(Q,n)}(\xi) t^{(Q,m)}(\eta)^*$ with $ Q,L \in \{ X,Y \}, Q \neq L , m,n \in \N , \xi \in Q \overset{n}{\otimes} L,$ and  $\eta \in Q \overset{m}{\otimes} L$. We define the closed ideal $J_X$ of $A$ by
$
	J_X := \phi_{X}^{-1}(\mathscr{K}(X)) \, \cap \, Ker(\phi_X)^{\perp}$.
The closed ideal $J_{Y}$ of $B$ is defined similarly. A representation $\ordRep$ is  covariant if $\pi_A(a) = \psi_{t_X}(\phi_X(a))$ and $\pi_B(b) = \psi_{t_Y}(\phi_Y(b))$, for all $a \in J_X$ and $b \in J_Y$.

	For an injective representation $\ordRep$, given $a \in A$ with $\pi_A(a) \in \psi_{t_X}(\mathscr{K}(X))$,  we have $a \in J_X$ and $\pi_A(a)=\psi_{t_X}(\phi_X(a))$. Indeed, if $\pi_A(a) \in \psi_{t_X}(\mathscr{K}(X))$, then there exists $k \in \mathscr{K}(X)$ with $\pi_A(a) = \psi_{t_X}(k)$ and
	$
	t_X(\phi_X(a)\xi) = \pi_A(a) t_X(\xi) = \psi_{t_X}(k) t(\xi) = t_X(k\xi),
	$
	for $\xi \in X$. Since $t_X$ is injective,
	$
	\phi_X(a)\xi = k \xi$,  for all $\xi \in X$. Thus $  \phi_X(a) = k.
	$
	Now if $b \in Ker\, \phi_X$,
	\[
	\pi_A(ab) =  \psi_{t_X}(k) \, \pi_A(b) = \psi_{t_X}(\phi_X(a)) \, \pi_A(b) = \psi_{t_X}(\phi_X(a)\,\phi_X(b)) = 0,
	\]
	thus $ab = 0$, since $\pi_A$ is injective and $a \in J_X$.

Next, we construct the universal $C^*$-algebras associated to pairs  of correspondences. 
	Let $\uniRepT$ be the universal representation of $(X,Y)$ and define the corresponding Toeplitz $C^*$-algebra by $\mathscr{T}_{X,Y} = C^*\uniRepT$. Then by  universality, for every representation $\ordRep$,  there is a surjection $\rho : \mathscr{T}_{X,Y} \rightarrow C^*\ordRep$ with  $\pi_X = \rho\, \circ \, \overline{\pi}_X, \, \pi_Y = \rho\, \circ \, \overline{\pi}_Y, \, t_X = \rho\, \circ \, \overline{t}_X$, and $t_X = \rho\, \circ \, \overline{t}_X$.

\vspace{.2cm}
\indent {\bf Definition A.2.}
The $C^*$ algebra $\mathcal{O}_{X,Y}$ of the correspondences $X$ and $Y$ is defined to be $\mathcal{O}_{X,Y} := C^*\uniRepO$, where $\uniRepO$ is a covariant representation.

\vspace{.2cm}
For each representation $\ordRep$ of $(X,Y)$ there exists a natural surjection $\rho : \mathcal{O}_{X,Y} \rightarrow C^*\ordRep$ such that $\pi_A = \rho \circ \hat{\pi}_A$ , $\pi_B = \rho \circ \hat{\pi}_B$ , $t_X = \rho \circ \hat{t}_X$ and $t_Y = \rho \circ \hat{t}_Y$. 
 For C$^*$-algebras $A$ and $B$ and $A$-$B$ and $B$-$A$-correspondences $X$ and $Y$, the Hilbert $A$-module $ \mathscr{F}_A(X,Y)$ is defined as the direct sum of  Hilbert $A$-modules $\XtY{2n}$ and $\YtX{2n+1}$, for $n \geq 0$, and is called the $A$-{Fock space} of $(X,Y)$. The $B$-{Fock space}  $ \mathscr{F}_B(X,Y)$ is defined similarly as the direct sum of  Hilbert $B$-modules $\XtY{2n+1}$ and $\YtX{2n}$, for  $n \geq 0$.
We define the canonical $*$-homomorphism $$\phi_{A}^{\infty} : A \rightarrow \mathscr{L}(\mathscr{F}_A(X,Y)) \oplus \mathscr{L}(\mathscr{F}_B(X,Y));\ \phi_{A}^{\infty}(a) := \sum_{i=0}^{\infty} \phi_{(X,i)}(a),\ \ (a\in A),$$ and linear map $$
\tau_{X}^{\infty} : X \rightarrow \mathscr{L}(\mathscr{F}_A(X,Y)) \oplus \mathscr{L}(\mathscr{F}_B(X,Y));\ \tau_{X}^{\infty}(\xi) := \sum_{m=0}^{\infty} \ordTau{(Y,m)}{(X,1)}(\xi),\ \ (\xi\in X).$$ The $*$-homomorphism $\phi_{A}^{\infty}$ and linear map $\tau_{Y}^{\infty}$ are defined similarly. We have,
\begin{align*}
	\tau_{X}^{\infty}(\xi)^*\,\tau_{X}^{\infty}(\eta) &= \sum_{m,n = 0}^{\infty} \ordTau{(Y,m)}{(X,1)}(\xi)^*\ordTau{(Y,n)}{(X,1)}(\eta) = \sum_{m = 0}^{\infty} \ordTau{(Y,m)}{(X,1)}(\xi)^* \,  \ordTau{(Y,m)}{(X,1)}(\eta) &\\&= \sum_{m = 0}^{\infty} \phi_{(Y,m)}(\la\xi,\eta\ra)= \phi_{B}^{\infty}(\la\xi,\eta\ra),
\end{align*}
 and	
\begin{align*}
	\phi_{A}^{\infty}(a) \, \tau_{X}^{\infty}(\xi) &= \sum_{i,m = 0}^{\infty}\phi_{(X,i)}(a) \ordTau{(Y,m)}{(X,1)}(\xi) = \sum_{i = 0}^{\infty}\phi_{(X,i)}(a) \ordTau{(Y,i)}{(X,1)}(\xi) &\\&= \sum_{i = 0}^{\infty} \ordTau{(Y,i)}{(X,1)}(\phi_{(X,i)}(a)\xi)=  \tau_{X}^{\infty}(\phi_{(X,i)}(a)\xi),
\end{align*}
for  $\xi , \eta \in X$, and $a \in A$. Also, $\phi_{(X,0)}: B \rightarrow \mathscr{L}(\XtY{0})$ is an isomorphism onto $\mathscr{K}(\XtY{0})$, and likewise for $\phi_{(0,Y)}$. Summing up, the Fock representation $\fockRep$ is an injective representation of $(X,Y)$ on $\mathscr{L}(\mathscr{F}_A(X,Y)) \oplus \mathscr{L}(\mathscr{F}_B(X,Y))$.
The corresponding linear map $\tau_{\infty}^{(X,n)} : \XtY{n} \rightarrow \mathscr{L}(\mathscr{F}_A(X,Y)) \oplus \mathscr{L}(\mathscr{F}_B(X,Y))$ satisfies 
$\tau_{\infty}^{(X,n)}(\xi) = \sum_{m = 0}^{\infty} \ordTau{(Y,m)}{(X,n)}(\xi)$, for $ \xi \in \XtY{n}$ and  odd $n$, and likewise for even $n$. Moreover, for  $a \in J_X$,
and $\xi , \eta \in X$, $\psi_{\tau_{X}^\infty}(\theta_{\xi , \eta}) = \sum_{m = 1 }^{\infty} \theta_{\xi , \eta} \otimes {\rm id}_{m-1}$, and for $k \in \mathscr{K}(X)$,  $\psi_{\tau_{X}^\infty}(k) = \sum_{m = 1}^{\infty} k \otimes {\rm id}_{m-1}$. Therefore,
\[
\phi_{A}^{\infty}(a) - \psi_{\tau_{X}^\infty}(\phi_X(a)) = \sum_{i = 0}^{\infty} \phi_{(X,i)}(a) - \sum_{m = 1}^{\infty} \phi_{X}(a) \otimes {\rm id}_{m-1} = \phi_{(X,0)}(a),
\]
since  $\phi_{(X,i)}(a) = \phi_X(a) \otimes {\rm id}_{m-1}$.
In particular, if $\phi_{A}^{\infty}(a) \in \psi_{\tau_{X}^\infty}(\mathscr{K}(X)), $  then $a \in J_X$ and $\phi_A^\infty(a) = \psi_{\tau_{X}^\infty}(\phi_X(a))$, thus,  $\phi_{(X,0)}(a) = 0$, and so $a = 0$, as $\phi_{(X,0)}$ is injective.

For Hilbert $J_X$ and $J_Y$-modules $\mathscr{F}_A(X,Y)J_X$ and $\mathscr{F}_B(X,Y)J_Y$, it follows from \cite[Corollary 1.4]{ka3} that,
\[
\mathscr{K}(\mathscr{F}_A(X,Y)J_X) = \overline{span}\{\theta_{\xi a,\eta} \in \mathscr{K}(\mathscr{F}_A(X,Y))\ | \ \xi ,\eta \in \mathscr{F}_A(X,Y) , a \in J_X \},
\]
which is a closed ideal of $\mathscr{L}(\mathscr{F}_X(X,Y)).$ If $k \in \mathscr{K}(\mathscr{F}_A(X,Y))$ and $\la\xi , k\eta\ra \in J_X$, for each $\xi , \eta \in \mathscr{F}_A(X,Y)$, then $k \in \mathscr{K}(\mathscr{F}_A(X,Y)J_X)$  \cite[Lemma 1.6]{ka3}. Moreover, for $n$ even and $m$ odd,  $\xi \in \XtY{n}$ $\eta \in \YtX{m}$, and $a \in J_X$,
$
	\theta_{\xi a , \eta} = \tau_\infty^{(X,n)}(\xi) \phi_{(X,0)}(a) \tau_\infty^{(Y,m)}(\eta)^* = \tau_\infty^{(X,n)}(\xi)(\phi_A^\infty(a) - \psi_{\tau_{X}^\infty}(\phi_X(a))) \tau_\infty^{(Y,m)}(\eta)^*$ is in $C^*\fockRep$, and likewise for other cases. In particular, there is a canonical continuous embedding  $$\mathscr{K}(\mathscr{F}_A(X,Y)J_X) \subseteq C^*\fockRep.$$

In covariant representations, we usually have canonical choices for the C$^*$-algebra $D$. One of these choices is
\[
D := \LL(\F_A(X,Y)) / \K(\F_A(X,Y)J_X) \oplus \LL(\F_A(X,Y)) / \K(\F_A(X,Y)J_X).
\]
Let
$
\sigma : \LL(\F_A(X,Y)) \oplus \LL(\F_B(X,Y)) \rightarrow D,$ be
defined as the direct sum of the corresponding quotient maps. For  $\phi_A = \sigma \circ \phi_A^\infty $,  $\phi_B = \sigma \circ \phi_B^\infty $, $\tau_X = \sigma \circ \tau_X^\infty$, and  $\tau_Y = \sigma \circ \tau_Y^\infty$, and $\fockcovRep$ is a covariant representation in $D$.

In order to show that there is an injective covariant representation, let us consider the $*$-homomorphism $f_X: \LL(\XtY{n})\to \XtY{n+1}; \ T\mapsto T \otimes {\rm id}$, and likewise for $f_Y$, and observe that the restriction of $f_X$  to $\K ((\XtY{n}) J_X)$ is  injective, for $n$ even, and the same  for the restriction  of $f_Y$ to $\K ((\XtY{n}) J_Y)$, for $n$  odd: Take $k \in \K((\XtY{n})J_X)$ with $k \otimes {\rm id} = 0$, then  $
	 \la\eta , \phi_X(\la\xi,k \xi'\ra)\eta'\ra= \la\xi \otimes \eta , (k \otimes {\rm id})(\xi' \otimes \eta')\ra=0,$
	for  $\xi , \xi' \in \XtY{n}$ and $\eta, \eta' \in X$, that is,  $\phi_X(\la\xi,k \xi'\ra) = 0$. Since $k \in \K(\XtY{n} J_X)$,  $\la\xi , k\xi'\ra \in J_X$, thus $\la\xi , k\xi'\ra = 0$, as  $\phi_X$ is injective on $J_X$. Therefore, $k = 0$. The other case is proved similarly.

\vspace{.2cm}
\indent {\bf Theorem A.3.}
\emph{The covariant representation $\fockcovRep$ is injective}.
\begin{proof} Let  $P_{(X,n)} \in \LL(\F_A(X,Y)) \oplus \LL(\F_B(X,Y))$ be the orthogonal  projection onto the direct summand $\XtY{n}$, and take $a \in A$ with $\phi_A(a) = 0$. Then,
$
		\phi_{(X,n)}(a) = P_{(X,n)} \phi_A^\infty(a)  P_{(X,n)}$ is in $P_{(X,n)} \K (\F_A(X,Y) J_X) P_{(X,n)}= \K ((\XtY{n}) J_X),$
	for each $n$. For $n = 0$, this yields $a \in J_X$. Since $\phi_{(X,1)} = \phi_X$ is injective on $J_X$,  $\| a \| = \| \phi_X (a) \|$. Moreover,   $$\| \phi_{(X,n)} (a) \| = \| \phi_{(X,n)}(a) \otimes {\rm id} \| = \| \phi_{(X,n+1)} (a) \|,$$ for each $n\geq 1$, thus $\| \phi_{(X,n)} (a) \| = \|a\| $ . We claim that $\lim_{n \rightarrow \infty} \| P_n k P_n \| = 0$, for each $k$  in $\K(\F_A(X,Y))$ or $ \K(\F_B(X,Y))$. If this is the case, we immediately conclude that  $a = 0$. Similarly, for each $b \in B $,  $\phi_B(b) = 0$ yields $ b = 0$. To prove the claim, we only need to check the case that $k = \theta_{\xi , \eta}$, where $\xi , \eta \in \F_A(X,Y)$, with  $\xi \in \XtY{n_1}$ and $\eta \in \XtY{n_2}$, for some $n_1, n_2 \geq 1$. But in this case, by a straightforward calculation,  $\| P_n \theta_{\xi , \eta} P_n \|\to 0, $ as $n\to\infty$.
\end{proof}

\vspace{.2cm}
\indent {\bf Corollary A.4.}
$(i)$ In the universal representation $\uniRepT$ of $(X,Y)$,  $\overline{\pi}_A(a)$ is in $\psi_{\overline{t}_X}(\K(X)) $ only if $a=0$. Similar conclusion holds for $Y$ and $B$.

$(ii)$ The universal representation $\uniRepO$ of $(X,Y)$ in $\mathcal{O}_{X,Y}$ is injective.

\subsection*{{\rm A.2.} groupoid crossed products}  This subsection is devoted to the study of  groupoid crossed products in our context. We show how to interpret these as $C^*$-algebras of certain product systems over the acting groupoid. 

Let $G$ be an \'{e}tale groupoid and $p:A\to \Gz$ be an upper-semicontinuous $C^*$-bundle over $G$  with
enough sections, that is, for each $u\in \Gz$, and $a\in A_u:=p^{-1}(u)$, there is a section $f$ (i.e., a continuous right inverse of $p$) with $f(u) = a$. By Hofmann Dupr\'{e}–Gillete Theorem \cite[Proposition 3.3]{mw} the $C^*$-algebra of sections of $A$, vanishing at infinity, is a $C_0(\Gz)$-algebra with fibres $A_u$. With a slight abuse of notation, we denote this $C^*$-algebra again by $A$. An action $\alpha$ of $G$ on $A$ is a bundle of maps $(\alpha_g)_{g\in G}$ such that,

$(i)$ $\alpha_g: A_{s(g)}\to A_{r(g)}$ is an isomorphism,

$(ii)$ $\alpha_{gh}=\alpha_g\circ \alpha_h$,

$(iii)$ $a\mapsto \alpha_g(a)$ is continuous,

\noindent for each  $g,h\in G$ with $r(h)=s(g)$. Consider the pull-back $r^*A$ and for compactly supported sections $f, f_1,f_2: G\to r^*A$, define,
$$f_1*f_2(g):=\int_{G} f_1(h)\alpha_h(f_2(h^{-1}g))d\lambda^{r(g)}(h),\ \ f^*(g):=\alpha_g\big(f(g^{-1})^*\big).$$
The $I$-norm is  defined on the $*$-algebra of these sections by $$\|f\|= \max\big\{\sup_{u}\int_G\|f(g)d\lambda^{u}(g), \sup_{u}\int_G\|f(g)d\lambda_{u}(g)\big\},$$ and the full $C^*$-norm is then defined by,
$$\|f\| :=
\sup\{\|L(f)\| : \ L\in{\rm Rep}\big(C_c(G; r^*A)\big)\},$$
where ${\rm Rep}\big(C_c(G; r^*A)\big)$ consists if $\|\cdot\|_I$-decreasing $*$-representations of $C_c(G; r^*A)$. 
The crossed product $A\rtimes_\alpha  G$ is the completion of $C_c(G, r^*A)$ in this norm. Fially, if $\pi$ is any faithful representation of $A$, the reduced crossed product $A\rtimes^r_{\alpha}  G:=(A\rtimes_\alpha  G)/\ker{\rm Ind \pi}$, and this is independent of the choice of $\pi$. 
   
If $\Gz*H$ is a Borel Hilbert Bundle over $\Gz$, the corresponding isomorphism
groupoid is 
$${\rm Iso}(\Gz* H ) := \{(u, V, v) : V\in\mathscr{U}(H_v,H_u)\}$$
with the weakest Borel structure such that
$$(u, V, v)\mapsto \langle V f_1(v)|f_2(u)\rangle$$
is Borel for all $f_1, f_2\in B(\Gz* H )$. A  representation of $G$ is a triple $(\mu, \Gz* H, U)$ consisting of a quasi-invariant
measure $\mu$ on $\Gz$, a Borel Hilbert bundle $\Gz * H$  and a Borel
homomorphism $\hat U : G \to {\rm Iso}(\Gz* H )$ such that
$\hat U_g = (r(g), U_g, s(g))$, for $g\in G$. We usually drop the hat and write $U$ for $\hat U$.  A covariant representation $(\pi,\mu, \Gz*H ,U)$ 
consists of a unitary representation $(\mu, \Gz* H ,U)$ and a $C_0(\Gz)$-linear
representation $\pi : A\to  \mathscr{L}(L^2(\Gz*H ,\mu))$ with direct integral decompositions, 
$$L^2(\Gz*H ,\mu):=\int_{\Gz} H_ud\mu(u),\ \ \pi=\int_{\Gz} \pi_ud\mu(u),$$
 such that, $$(\pi(a)\xi)(u)=\pi_u(a(u))\xi_u,\ \ (u\in\Gz, a\in A),$$
where $a\in A$ is identified with the corresponding section in $C_0(\Gz, A)$ with $a(u)\in A_u$,  
such that 
$U_g\pi_{s(g)}(b) = \pi_{r(g)}(\alpha_g(b))U_g,$ for $\nu$-a.a. $g\in G$ and all $ b\in A_{s(g)}$. In this case, we simply say that $(\pi, U)$ is a covariant pair. 
The crossed product $A \rtimes_{\alpha} G$ is the
universal $C^*$-algebra generated by a covariant pair $(i^{\alpha}_A, i^{\alpha}_G)$, with $i^{\alpha}_A$ and  $i^{\alpha}_G$ having range inside the multiplier algebra $M(A\rtimes_{\alpha} G)$, defined by,
$$(i^{\alpha}_A(a)f)(g):=a(r(g))f(g),\  (i^{\alpha}_G(g)f)(h)=\alpha_g\big(f(g^{-1}h)\big),\ \ (h\in G^{r(g)}),$$
for $f\in C_c(G, r^*A), a\in A, g\in G$
 \cite[Lemme 4.6]{r2} (cf., \cite[Theorem 7.12]{mw}). When $G$ is discrete, a typical element of $A \rtimes_{\alpha} G$ is then written as 
 $i^{\alpha}_A(a)i^{\alpha}_G(g)$, for $a\in A$, $g\in G$. 
 
Let $(G,P)$ be a quasi-lattice ordered groupoid, and let $\alpha$ be an action of $G$ on a $C_0(\Gz)$-algebra $A$. 
Consider the product
system $X$ over $P^{\rm op}$ (opposite of $P$) defined as follows:
for $p \in P^{\rm op}$, let $X_p:=A_{r(p)}$, regarded as an 
$A_{s(p)}$-$A_{r(p)}$-correspondence with, 
$$
\langle x,y \rangle_p := x^*y, \quad \phi_p(a)(x) := \alpha_p(a)x,
\quad x \cdot a = xa,
$$
for  $x,y \in X_p$ and $a \in A_{s(p)}$,
and define the isomorphisms $X_p \otimes_{A_{s(q)}} X_q \to X_{qp}$ by extending the map $x
\otimes_{A_{s(q)}} y \mapsto \alpha_q(x)y$.
The left action $\phi_p$ satisfies $\phi_p(A_{s(p)})\subseteq \mathscr{K}_{A_{s(p)}}(X_p)$, 
and   $X_p$ is essential as a left $A_{s(p)}$-module, that is, $\phi_p(A_{s(p)})X_p$ is total set in $X_p$, for  $p\in  P^{\rm op}$. It follows that, for each $p,q\in P^{\rm op}$ with $r(p)=s(q)$ and $p\vee q<\infty$, and each  $S\in \mathscr{K}_{A_{s(p)}}(X_p)$, $T\in\mathscr{K}_{A_{s(q)}}(X_p)$, 
$$\iota_p^{p\vee q}(S)\iota_q^{p\vee q}(T)=(S\otimes{\rm id}_{s(q)})({\rm id}_{s(p)}\otimes T)\in\mathscr{K}_{A_{s(p\vee q)}}(X_{p\vee q}),$$ 
since both terms on the right hand side of the above equality are compact by \cite[Corollary 3.7]{Pimsner1997}. Thus,  $X$ is compactly aligned.

Next, let us consider the coaction $\widehat{\alpha}$ of $G$ on $A
\rtimes_{\alpha} G$ given by
$\widehat{\alpha}(i^{\alpha}_A(a)i^{\alpha}_G(g))
= i^{\alpha}_A(a)i^{\alpha}_G(g) \otimes i_G(g)$
for  $a \in A$ and $g \in G$, and by the universal properties, the crossed product $A
\rtimes_{\alpha} G$ is isomorphic to the full
cross-sectional algebra of the resulting Fell bundle over $G$. The
reduced crossed product $A \rtimes^r_{\alpha} G$ is the
reduced cross-sectional algebra of the same bundle, and is a quotient of $A \times_{\alpha} G$ by $\ker{\rm Ind \pi}$, for any faithful representation $\pi$ of $A$.
We write $(\lambda_A^{\alpha}, \lambda_G^{\alpha})$ for the generating covariant 
representation with values in $A \rtimes^r_{\alpha} G$.
Then the normalisation $\widehat{\alpha}^n$  of $\widehat{\alpha}$
is a normal coaction of $G$ on $A \rtimes^r_{\alpha} G$.

\vspace{.3cm}
\indent {\bf Proposition A.5.}
	Let $G=P\cup P^{-1}$ is discrete, and $P$ is directed. Let $G$ act on a $C_0(\Gz)$-algebra $A$ by $\alpha$ and $X$ be the operator system defined as above. Then,
	
	$(i)$ there
	is an isomorphism $\phi \colon A \rtimes^r_{\alpha} G \to
	\NO{X}^\reduced$, mapping
	$\lambda_G^{\alpha}(p)^* \lambda_A^{\alpha}(x)$ to
	$j^\reduced_{X }(x)$, for $x \in X_p := A_{r(p)}$, satisfying
	$\delta^n\circ\phi=(\phi\otimes \id_{C^*(G)}) \circ
	\widehat{\alpha}^n$,
	
	$(ii)$ 	there is an isomorphism $A \times_{\alpha} G \to \NO{X}$, mapping
	takes $i_G^{\alpha}(p)^* i_A^{\alpha}(x)$ to
	$j_{X }(x)$, for $p\in P$ and $x \in X_p := A_{r(p)}$.
\begin{proof}
	Let us  first observe that $A \rtimes^r_{\alpha} G$ is
	generated by a Nica covariant representation of $X$. 
For $p \in P$, consider the map,
	\[
	\psi_p \colon X_p \to A \rtimes^r_{\alpha} G; \ \ \psi_p(x) := \lambda_G^{\alpha}(p)^* \lambda_A^{\alpha}(x),\quad (x \in A_{r(p)}).
	\]
	For $p,q\in P^{\op}$  and  $x,y \in X_p,$ $z
	\in X_q$, if  $r(p)=s(q)$,
	\begin{align*}
		\psi_p(x)\psi_q(z)
		&= \lambda_G^{\alpha}(p)^* \lambda_A^{\alpha}(x) \lambda_G^{\alpha}(q)^* \lambda_A^{\alpha}(z) \\
		&= \lambda_G^{\alpha}(p)^* \lambda_G^{\alpha}(q)^* \lambda_G^{\alpha}(q)
		\lambda_A^{\alpha}(x) \lambda_G^{\alpha}(q)^* \lambda_A^{\alpha}(z) \\
		&= \lambda_G^{\alpha}(qp)^* \lambda_A^{\alpha}(\alpha_q(x))\lambda_A^{\alpha}(z) = \psi_{qp}(xz),
	\end{align*}
	and since $x\in A_{r(p)}$, if $r(p)\neq s(q)$, $x(s(q))=0$, hence, for $f\in C_c(G, s^*A)$ and $g\in G_{s(q)}$,
	\begin{align*}
		  i_A^{\alpha}(x) i_G^{\alpha}(q)^*(f)(g)&=x(s(q))\alpha_q(f(gq^{-1}))=0,
	\end{align*}
thus, for the regular representation $\lambda: A\rtimes_\alpha G\to A\rtimes_\alpha^r G$ (indeed, for its extension to the multiplier algebras),  
	\begin{align*}
	\lambda_A^{\alpha}(x) \lambda_G^{\alpha}(q)^*&=\lambda( i_A^{\alpha}(x) i_G^{\alpha}(q)^*)=0, 
\end{align*}
therefore, 	$\psi_p(x)\psi_q(z)=0$. On the other hand, 
	$$\psi_p(x)^* \psi_p(y)
		= \lambda_A^{\alpha}(x)^* \lambda_G^{\alpha}(p) \lambda_G^{\alpha}(p)^* \lambda_A^{\alpha}(y)
		= \lambda_A^{\alpha}(x^*y)
		= \psi_{r(p)}(\langle x,y \rangle_{p}).
	$$	
Note that here we work with the quasi-ordered lattice groupoid $(G^{\rm op}, P^{\rm op})$, and so  replaced $C_c(G, r^*A)$ with $C_c(G, s^*A)$.
	
	$(i)$ Each $X_p$ is essential, the left action of $A_{s(p)}$ on each $X_p$
	is injective and by compacts, and $P$ is directed.
	Let us also observe that $\psi^{(p)} \circ \phi_p =
	\lambda_A^{\alpha}$,  for  $p\in P$. Take an approximate identity $(e_k)_{k}$
	of $A_{r(p)}$, then $\phi_p(a)=\lim_{k}
	\alpha_p(a) \otimes e_k^*$, hence,
	\[
	\psi^{(p)}(\phi_p(a))
	= \lim_{k} \psi_p(\alpha_p(a)) \psi_p(e_k)^*
	= \lim_{k } \lambda_G^{\alpha}(p)^* \lambda_A^{\alpha}(\alpha_p(a) e_k^*)
	\lambda_G^{\alpha}(p)
	= \lambda_A^{\alpha}(a),
	\]
	for $p\in P$ and $a\in X_p$, as claimed. By an argument similar to \cite[Proposition 5.1]{SY}, it follows that $\psi$ is CNP-covariant.  
	
	The image of $\psi$ generates $A \rtimes^r_{\alpha} G$, as the latter is spanned by elements of the form
	$\lambda_G^{\alpha}(g) \lambda_A^{\alpha}(a)$, and  $\psi$ is injective as $\lambda_A^{\alpha}$
	is so. Since the left action on each fibre of $X$ is injective,  $X$ is
	$\tilde\phi$-injective. Now, the normalisation $\widehat{\alpha}^n$
	satisfies $\widehat{\alpha}^n (\psi(x)) = \psi(x) \otimes
	i_G(p),$ for  $p \in P$ and $x \in X_p$,
	and Theorem~\ref{thm:projective property} gives an epimorphism
	$\phi \colon A \rtimes^r_{\alpha} G \to \NO{X}^\reduced$
	with $\phi\circ\psi=j_{X}^r$.
	Finally, by Corollary~\ref{inj}(3) $\phi$ is injective.
	
	$(ii)$ 
	Since the isomorphism $\phi \colon A \rtimes_{\alpha}^r G \to
	\NO{X}^\reduced$ 
	intertwines the coactions, the
	corresponding Fell bundles are isometrically isomorphic. Therefore, the full
	cross-sectional algebras
	$A \rtimes_{\alpha} G$ and $\NO{X}$ are canonically isomorphic.
\end{proof}

\subsection*{{\rm A.3.} $C^*$-algebra of multilayered topological $k$-graphs} In this subsection, we 
illustrate our construction by constructing the $C^*$-algebra of multilayered topological $k$-graphs. This is done by slightly extending Yeend's path groupoid.
In what follows, we write $\NN^k$ to denote the set of $k$-tuples of non-negative (possibly zero) integers. 

Following \cite{y}, we define  a \emph{multilayered topological
	$k$-graph} as a triple $( \Lambda, {\rm deg}, d )$ consisting of: (1)~ a
small category $\Lambda$  endowed with a second countable
locally compact Hausdorff topology under which the composition map
is continuous and open, the range map $r$ is continuous and the source
map $s$ is a local homeomorphism; (2)~a divider $d: \mathbb N^k\to \ZZ_2^k$; and (3)~a continuous functor $\dd \colon \Lambda
\to \ZZ_2^{k}\rtimes_d\NN^k$, called the degree map, satisfying the
factorisation property: if $\dd(\lambda) = (i, m+n)$, then there exist
unique $\mu,\nu$ with $\dd(\mu) = (i,m)$, $\dd(\nu) =(i+d(m), n)$ such that $\lambda =
\mu\nu$.

Elements of $\Lambda$ are called paths, and paths of degree $0$
are called vertices. For $i\in \ZZ_2^k$ and $m \in \NN^{k}$, we put $\Lambda_i^{m}
:= \dd^{-1} (i,m)$. If $0 \le m \le n \le p$ in $\NN^{k}$ and
$\lambda \in \Lambda_i^{p}$ then we write $\lambda (i,m,n)$ for the
unique path in $\Lambda_{i+d(m)}^{n-m}$ such that $\lambda = \mu \lambda
(i,m,n) \nu$, where $\mu \in \Lambda_i^{m}$ and $\nu \in
\Lambda_{i+d(n)}^{p-n}$. For $0 \le m \le p$ in $\NN^{k}$ and $\lambda
\in \Lambda_i^{p}$, we write $\lambda (i,m)$ for the source $s (\lambda (i,0,m))
= \lambda (i,m,m)$.   For $i\in\ZZ_2^k$ and $m,p \in \NN^{k}$ with $m \le p$, the
map $$\sigma_i^{m} \colon \Lambda_i^{p} \to \Lambda_{i+d(m)}^{p-m}; \ \sigma_i^{m} (\lambda) := \lambda (i,m,p)$$ is continuous. For $U, V
\subset \Lambda$, we write
\[
UV := \{\, \lambda\mu \mid \lambda\in U,\ \mu\in V,\ s(\lambda)=r(\mu)\,\}.
\]
For $U \subseteq \Lambda_i^{p}$ and $V \subseteq \Lambda_j^{q}$,
\[
U \vee V := U \Lambda_{i+d(p)}^{(p \vee q)-p} \cap V \Lambda_{j+d(q)}^{(p \vee q)-q}
\]
is the set of minimal common extensions of paths from
$U$ and $V$. A multilayered topological $k$-graph $( \Lambda, \dd, d )$ is
called compactly aligned if $U \vee V$ is compact whenever  $U$
and $V$ are so.

Let $( \Lambda, \dd, d)$ be a multilayered topological
$k$-graph. Define $A_i := C_{0} ( \Lambda_i^{0} )$, for $i\in\ZZ^k_2$. For  $n \in
\NN^{k}$ let $X_{(i,n)}$ be the  $A_i$-$A_{i+d(n)}$-correspondence, constructed as the completion of $C_c(\Lambda_i^n)$ with operations,
\[
\langle f, g \rangle_{i}^{n} ( v ) = \textstyle{\sum_{\eta \in \Lambda_i^{n}v}} \bar f ( \eta ) g (\eta),
\quad
\phi_a(f)( \lambda ) = a (r(\lambda))
f(\lambda), \quad f \cdot b ( \lambda ) = 
f(\lambda) b(s(\lambda)),
\]
for $v\in\Lambda_{i+d(n)}^0, a\in A_i, b\in A_{i+d(n)},$ and $f,g\in C_c(\Lambda_i^n)$. As in \cite[Section~1]{ka1}, one could see that  $X_{(i,n)}$ is a subspace of $C_0(\Lambda_i^n)$.

	For $f \in X_{(i,m)}$ and $g \in X_{(i+d(m),n)}$, let us define, $$fg \colon
	\Lambda_i^{m+n} \to \CC;\ \ (fg)(\lambda) := f(\lambda (i,0,m))
	g(\lambda (i,m,m+n)).$$ To show that the family of 
	correspondences $ X_{(i,n)}$
	is a product system over $P:=\ZZ_2^k\rtimes_d\NN^{k}$, let us first observe that for $f_1, f_2 \in X_{(i,m)}$ and $g_1, g_2 \in X_{(i+d(m),n)}$,  
	\begin{align*}
		\langle f_1g_1, f_2g_2 \rangle_{i}^{m+n} ( v )
		&= \textstyle{\sum_{\nu \in \Lambda_{i+d(m)}^{n} v}}  \textstyle{\sum_{\mu \in \Lambda_i^{m} r( \nu )}}
		\bar f_1(\mu) f_2(\mu) \bar g_1(\nu) g_2(\nu)
		\\&= \textstyle{\sum_{\nu \in \Lambda_{i+d(m)}^{n} v}} \langle f_1, f_2 \rangle_{i}^{m} ( r( \nu ) ) \bar g_1(\nu) g_2(\nu)
		\\&= \langle g_1, \langle f_1,f_2 \rangle_{i}^{m} \cdot g_2 \rangle_{i+d(m)}^{n} ( v ),
	\end{align*}
	for $v \in \Lambda_{i+d(m)+d(n)}^{0}$. In particular, for $f_2 = f_1 = f$ and $g_2 = g_1 = g$,  $fg
	\in X_{m+n}$ by definition of $X_{(i, m+n)}$.
	Further, since $$\langle g_1, \langle f_1,f_2 \rangle_{i}^{m}
	\cdot g_2 \rangle_{i+d(m)}^{n} = \langle f_1 \otimes_{A_{i+d(m)}} g_1, f_2
	\otimes_{A_{i+d(m)}} g_2 \rangle_{A_i},$$ the map $f \otimes_{A_{i+d(m)}} g \mapsto
	fg$ extends to an isometric adjointable operator mapping $X_{(i,m)}
	\otimes_{A_{i+d(m)}} X_{(i+d(m),n)}$ into $X_{(i, m+n)}$. This map  has dense range, since the span of the set $\{\, fg \mid f \in C_{c} (\Lambda_i^{m}), g \in C_{c}
	(\Lambda_{i+d(m)}^{n}) \,\}$ in $C_{c} (\Lambda_i^{m+n})$ is a subalgebra
	of $C_{0} (\Lambda_i^{m+n})$, and so  uniformly dense in
	 $X_{(i,m+n)}$ by an argument as in \cite[Lemma 1.26]{ka1}. It follows that the map $X_{(i,m)}
	 \otimes_{A_{i+d(m)}} X_{(i+d(m),n)}\to X_{(i, m+n)}$, constructed above, is surjective. 

	For $i\in\ZZ_2^k$ and $m \in \NN^{k}$, let us  denote by $F_i^{m}$ the set of functions
	$f \in C_{c} (\Lambda_i^{m})$ such that the source map restricts
	to a homeomorphism of $\supp(f)$. By definition, for
	each $f \in F_i^m$ and $v \in \Lambda_i^0$ with $\Lambda_i^{m} v$ non-empty, there is an  element $\lambda_{f,i, v}\in \Lambda_i^m v$ such that $f(\mu) = 0$, for all $\mu \in \Lambda_i^m v \backslash
	\{\lambda_{f,i,v}\}$. A partition of unity argument, using that the source map in
	$\Lambda$ is a local homeomorphism, shows that each element of
	$C_{c} (\Lambda_i^{m})$ is a finite sum of elements of $F_i^m$, that is, $F_i^{m}$ is a total set in $X_{(i,m)}$.

Next, given
	$i\in \ZZ_2^k$ and $m,n \in \NN^{k}$,  for $f \in X_{(i,m)}$, $g \in F_i^{m}$, $c \in X_{(i, m
		\vee n)}$, and $\xi \in \Lambda_i^{m \vee n}$, observe that,
	\[
	\big(\iota_{(i,m)}^{(i, m \vee n)} (f \otimes g^{*}) (c)\big)(\xi)
	=
	f(\xi(i,0,m)) \bar g(\lambda_{g, \xi(i,m)}) c(\lambda_{g, \xi(i,m)}\xi (i+d(m),m,m \vee n)).
	\]
Indeed, by density, we only need to check this for the case $c=c_1c_2$, for $c_1 \in X_{(i,m)}$ and $c_2 \in X_{(i+d(m), m \vee n-m)}$, for which we have,
	\begin{align*}
		\big(\iota_{m}^{m \vee n} (f \otimes g^{*}) (c_1c_2)\big)(\xi)
		&= f(\xi (i,0,m)) \textstyle{\sum_{\mu \in \Lambda_i^{m} \xi (i,m)}} \bar g(\mu) c_1(\mu) c_2(\xi (i+d(m),m,m \vee n)) \\
		&= f(\xi (i,0,m)) \bar g(\lambda_{g, \xi(i,m)}) c_1(\lambda_{g, \xi(i,m)}) c_2(\xi(i+d(m),m,m \vee n)) \\
		&= f(\xi (i,0,m)) \bar g(\lambda_{g, \xi(i,m)}) c_1c_2(\lambda_{g, \xi(i,m)}\xi(i+d(m),m,m \vee n)),
	\end{align*}
as claimed.

\indent {\bf Lemma A.6.}
Let $i\in\ZZ_2^k$ and $m,n\in\NN^k$, 

$(i)$ for $T\in\Ka(X_{(i,m)})$, the following function vanishes at infinity $$\chi_T:\Lambda_i^m\to\RR; \ \ \chi_T(\lambda)
	=\sup_f\abs{T(f)(\lambda)},$$ 
	where the supremum runs over those $f\in F_i^m$ with uniform norm at most one.
	
\noindent If further  $\Lambda$ is compactly aligned,

$(ii)$ for each $f_m \in F_i^{m}$ and $f_n \in
F_{i+d(m)}^{n}$, $C:=\supp (f_m) \vee \supp (f_n) \subseteq
\Lambda_i^{m \vee n}$ is compact,

$(iii)$ for $p\in\{m,n\}$, $f_m,g_m \in F_i^{m}$ and $f_n, g_n \in
F_{i+d(m)}^{n}$, $C:=\supp (f_n) \vee \supp (f_m)$,  a finite relatively compact open cover  $(
V_{k}^{p})_k$  of
$\sigma_i^{p} (C)$ such that  $s$ restricts to a homeomorphism on each
${V_{k}^{p}}$,  a partition of unity $\phi_{k,i}^{p} \colon \sigma_i^{p} (C) \to
[0,1]$, subordinated to the family $(V_{k}^{p} \cap \sigma_i^{p} (C))_k$, and functions $\rho_{k,i}^{p} \colon
\Lambda_{i+d(p)}^{(m \vee n)-p} \to [0,1]$, satisfying
$\rho_{k,i}^p= \sqrt{\phi_{k,i}^{p}}$ on $\sigma_i^{p} (C)$, and $\rho_{k,i}^p=0$ off $V_k^p$, let $a_{k,j,i}, b_{j,i} \in C_{c} (\Lambda_i^{m \vee n})$
be defined by $a_{k,j,i} := g_{m}(\rho_{k,i}^{m} \cdot \langle f_{m}
\rho_{k,i}^{m},f_{n} \rho_{j,i}^{n} \rangle_{i}^{m \vee n})$ and
$b_{j,i} := g_{n} \rho_{j,i}^{n}$, then we have,
\begin{equation*}\label{compactprod}
	\iota_{(i,m)}^{(i,m \vee n)} (g_{m} \otimes f_{m}^{*}) \iota_{(i+d(m),n)}^{(i, m \vee n)} (f_{n} \otimes g_{n}^{*})
	= \textstyle{\sum_{j,k}} a_{k,j,i} \otimes b_{j,i}^{*}.
\end{equation*}

\begin{proof}
	$(i)$ By a standard density argument, it is enough to prove the claim for the rank one case $T=g\otimes h^*$ for $g,h\in X_{(i,m)}$. For  $f\in F_i^m$ of norm at most one, and $\lambda\in\Lambda_i^m$, 
	\begin{equation*}
		\abs{(g\otimes h^*)(f)(\lambda)}=\bigl\lvert g(\lambda)
		\textstyle{\sum_{\eta\in \Lambda_i^ms(\lambda)}}
		\overline{h(\eta)}f(\eta)\bigr\rvert
		=\abs{g(\lambda)\bar h(\lambda_{f,i,s(\lambda)}) f(\lambda_{f,i,s(\lambda)})}
		\le \abs{g(\lambda)}\norm{h}_\infty,
	\end{equation*}
	which vanishes at infinity along with $g$.  vanishes at infinity on $\Lambda^m$, so will $\chi_{g\otimes h^*}$.

$(ii)$ This follows from the assumption that  $\Lambda$ is compactly aligned. 

$(iii)$ Since the
	source map is injective on  $\supp(\rho_{k,i}^m)$, for $\mu \in
	\Lambda_{i+d(m)}^{(m \vee n)-m}$ with $\Lambda_{i+d(m)}^{m} r(\mu)$  non-empty, and each $j$,
		\begin{align*}
		(\textstyle{\sum_{k}}& \rho_{k,i}^{m} \cdot \langle f_{m} \rho_{k,i}^{m},f_{n} \rho_{j,i}^{n} \rangle_{i}^{m \vee n}) (\mu) \\
		&= \textstyle{\sum_{k}} \rho_{k,i}^{m} (\mu) \sum_{\alpha \in \Lambda_i^{m \vee n} s(\mu)}
		\bar f_{m} (\alpha (i,0,m)) \rho_{k,i}^{m} (\alpha (i+d(m),m,m \vee n)) (f_{n} \rho_{j,i}^{n})(\alpha) \\
		&= \textstyle \sum_{k} (\rho_{k,i}^{m} (\mu))^{2} \bar f_{m}(\lambda_{f_m, i, r(\mu)}) (f_{n} \rho_{j,i}^{n})(\lambda_{f_m, i+d(m), r(\mu)}\mu) \\
		&= \bar f_{m}(\lambda_{f_m, i, r(\mu)}) (f_{n} \rho_{j,i}^{n})(\lambda_{f_m, i+d(m), r(\mu)} \mu).
	\end{align*}
	Fix $c \in X_{(i, m \vee n)}$ and $\xi \in \Lambda_i^{m \vee n}$, and
	write $\lambda_{(i,m)} := \lambda_{f_m, i, \xi (m)}$, $\lambda := \xi(i+d(m),m, m \vee n)$, and $\beta := (\lambda_{(i,m)} \lambda)(i+d(m), n,m \vee n)$.
	If $\beta \in \sigma_{i+d(m)}^{n} (C)$ then, since the source map is injective on each $\supp
	(\rho_{j,i}^{n})$,
	\begin{align*}
		\textstyle{\sum_{j}} \rho_{j,i}^{n} (\beta) \langle b_{j,i},c \rangle_{i}^{m \vee n} (s(\xi))
		&= \textstyle{\sum_{j}} (\rho_{j,i}^{n} (\beta))^{2} \bar g_{n}(\lambda_{g_n, i+d(m), r(\beta)}) c(\lambda_{g_n,  i+d(m), r(\beta)} \beta) \\
		&= \bar g_{n}(\lambda_{g_n, i+d(m), r(\beta)}) c(\lambda_{g_n,  i+d(m), r(\beta)} \beta),
	\end{align*}
	hence,
	\begin{align*}
		\big(&\textstyle{\sum_{j,k}} a_{k,j,i} \otimes b_{j,i}^{*}\big) (c)(\xi)=\textstyle{\sum_{j,k}} g_{m}(\xi (i,0,m))
			(\rho_{k,i}^{m} \cdot \langle f_{m} \rho_{k,i}^{m},f_{n} \rho_{j,i}^{n} \rangle_{i}^{m \vee n}) (\lambda)
			\langle b_{j,i},c \rangle_{i}^{m \vee n} (s(\xi)) \\
		&= g_{m}(\xi (i,0,m)) \textstyle{\sum_{j,k}} \rho_{k,i}^{m} \cdot \langle f_{m} \rho_{k,i}^{m},f_{n} \rho_{j,i}^{n} \rangle_{i}^{m \vee n}) (\lambda)
			\langle b_{j,i},c \rangle_{i}^{m \vee n} (s(\xi)) \\
		&= g_{m}(\xi (i,0,m)) \bar f_{m}(\lambda_{(i,m)}) f_{n}((\lambda_{(i,m)}\lambda) (i+d(m),0,n))
		\bar g_{n}(\lambda_{g_n, i+d(m), r(\beta)}) c(\lambda_{g_n, i+d(m), r(\beta)} \beta)\\
		&=\iota_{(i,m)}^{(i,m \vee n)} (g_{m} \otimes
		f_{m}^{*}) \iota_{(i+d(m),n)}^{(i,m \vee n)} (f_{n} \otimes g_{n}^{*}),
	\end{align*}
as required.
\end{proof}

\indent {\bf Corollary A.7.}
	Let $\Lambda$ be a multilayered topological $k$-graph.
	The product system $X$ defined above
	is compactly aligned if and only if $\Lambda$ is
	compactly aligned.
\begin{proof}
	If $\Lambda$ is compactly aligned, by the above lemma, for $i\in\ZZ_2^k$, and  $m,n \in \NN^{k}$,  for $f_1,f_2 \in F_i^{m}$ and $g_1,g_2 \in F_{i+d(m)}^{n}$, $\iota_{(i,m)}^{(i, m
		\vee n)} (f_1 \otimes f_2^{*}) \iota_{(i+d(m), n)}^{m \vee n} (g_1 \otimes g_2^{*}) \in
	\mathscr{K}_{A_i} (X_{(i,m \vee n)})$, which in turn, by a density argument implies that $X$ is compactly aligned. 
	
	If $\Lambda$ is not compactly aligned, then there exist $i\in\ZZ_2^k$, $m,n\in\NN^k$,
	 $U\subseteq\Lambda_i^m$, and $V\subseteq\Lambda_{i+d(m)}^n$ such that $U$ and $V$ are compact, but $U\vee V$ is not compact. For  $C\in\{U,V\}$,  let $(
	V_{k}^C)_k$ be a finite relatively compact open cover of
	$C$ such that  $s$ restricts to a homeomorphism on each
	$V_{k}^C$. Let $\phi_{k,i}^C \colon C \to
	[0,1]$ be a partition of unity 
	subordinated to $(V_{k}^C \cap C)$, and consider functions $\rho_{k,i}^C$ as in the above lemma. Put $T^i_C=\sum_{k}\rho_{k,i}^C\otimes(\rho_{k,i}^C)^*$.
	Then $T^i_U\in\mathscr{K}_{A_i}(X_{(i,m)})$ and $T^i_V\in\mathscr{K}_{A_{i+d(m)}}(X_{(i+d(m),n)})$. We claim that
	$T:=\iota_{(i,m)}^{(i,m\vee n)}(T^i_U) \iota_{(i+d(m),n)}^{(i, m\vee n)}(T^i_V)$ is not compact. For $f\in X_p$, $C=U$ or $V$, and $\lambda\in C$, 
	\begin{equation*} \label{eq:1}
		\begin{split}
			T^i_C(f)(\lambda) &= \sum_{k}\bigl(\rho_{k,i}^C\otimes (\rho_{k,i}^C)^*\bigr)(f)(\lambda) = \sum_{k}\rho_{k,i}^C(\lambda)
			\sum_{\eta\in\Lambda^ps(\lambda)}\bar\rho_i^C(\eta)f(\eta)
			\\&=\sum_{k}\rho_{k,i}^C(\lambda) \bar \rho_{k,i}^C(\lambda)f(\lambda)
			=f(\lambda).
		\end{split}
	\end{equation*}
	Given $\lambda\in U\vee V$, choose $f_\lambda\in F_i^{m\vee n}$ of norm at most one such that $f_\lambda(\lambda)=1$. Then, 
	\begin{equation*}
		T(f_\lambda)(\lambda)=\iota_{(i,m)}^{(i, m\vee n)}(T^i_U) \iota_{(i+d(m),n)}^{(i, m\vee n)}(T^i_V)(f_\lambda)(\lambda)
		=\iota_n^{m\vee n}\iota_{(i+d(m),n)}^{(i, m\vee n)}(T^i_V)(f_\lambda)(\lambda)
		=f_\lambda(\lambda)=1,
	\end{equation*}
	and by part $(i)$ of the above lemma, $T$ is not compact. It follows that $X$ is not compactly aligned. 
\end{proof}

Following Yeend \cite{y}, we associate two groupoids $G_\Lambda$ and
$\mathcal{G}_\Lambda$ to each compactly
aligned multilayered topological higher-rank graph $\Lambda$, and show that for $X=X_\Lambda$ constructed above, $\Tc(X_\Lambda)$ and $\NO{X_\Lambda}$ are isomorphic to $C^*(G_\Lambda)$ and
 $C^*(\mathcal{G}_\Lambda)$, respectively.

The first groupoid 
$G_{\Lambda}$ is the path groupoid of  $\Lambda$
consisting of quadruples
$(x,i,m,y)$ where $x$ and $y$ are (possibly infinite) paths in
$\Lambda$, $i\in \ZZ_2^k$ and $m \in \ZZ^{k}$, such that there exist $j,k\in \ZZ_2^k$ and $p,q \in \NN^{k}$
with $(j,p) \le \dd(x)$, $(k,q) \le \dd(y)$, $j+d(p)=k+d(q)$, $p-q=m$, and $\sigma_j^p(x) = \sigma_k^q(y)$.
As in \cite[Theorem 3.16]{y}, one can show that $G_{\Lambda}$ is a locally compact
$r$-discrete topological groupoid admitting a Haar system
consisting of counting measures. A basis for the topology of $G_{\Lambda}$
could be described  as follows. Define $\Lambda \ast_{s} \Lambda :=
\{\, (\lambda, \mu) \in \Lambda \times \Lambda \mid s(\lambda)=s(\mu) \,\}$,
and for $U,V \subseteq \Lambda$ define $U \ast_{s} V := (U \times V) \cap (\Lambda \ast_{s} \Lambda)$.
For $F \subseteq \Lambda \ast_{s} \Lambda$, $i\in\ZZ_2^k$ and $m \in \ZZ^{k}$, let $Z(F,i, m)$ be the set of all quadryples $(\lambda x,i,m,\mu x) \in G_{\Lambda}$ with $(\lambda,\mu) \in F, r(\dd(\lambda))=r(\dd(\mu)),$ and $\dd(\lambda)\dd(\mu)^{-1}=(i,m)$. 
Then the family of sets of the form $Z(U \ast_{s} V,i, m) \cap Z(F,i, m)^{c}$, with $i\in \ZZ_2^k$, $m \in \ZZ^{k}$,
$U,V \subseteq \Lambda$  open, and $F \subseteq \Lambda \ast_{s} \Lambda$ compact,
is a basis for the topology on $G_{\Lambda}$.

For $V\subset \Lambda^0$, a subset $E\subseteq V\Lambda$ is called exhaustive
(cf., \cite[Definition 4.1]{y}) for $V$ if for all $\lambda\in V\Lambda$ there exists $\mu\in E$ such that $\{\lambda\}\vee\{\mu\}\ne\emptyset$. A (possibly  infinite) path $x$ is called a boundary path (cf., \cite[Definition 4.2]{y}) if for all $i\in \ZZ_2^k$ and $m\in\NN^k$ with $(i,m)\le \dd(x)$, and for each compact set $E\subseteq \Lambda$ such that $r(E)$ is a neighbourhood of $x(i,m)$ and $E$ is exhaustive for $r(E)$, there exists $\lambda\in E$ such that $x(i+d(m),m,m+\dd(\lambda))=\lambda$. We write $\partial \Lambda$ for the set of all boundary paths. The set $\partial\Lambda\times \ZZ_2^k$ is a closed and invariant subset of $G_\Lambda^{(0)}$, where for a path $x$, $(x,i)$ is identified  with the element $(x,i,0,x)\in G_\Lambda^{(0)}$. Also,  $v\partial\Lambda\ne \emptyset$, for all $v\in\Lambda^0$. The boundary-path groupoid $\mathcal{G}_\Lambda$ is then defined (cf., \cite[Definition 4.8]{y}) as the reduction of $G_\Lambda$ to $\partial\Lambda\times\ZZ_2^k$. For a subset $U$ of $\Lambda$, we let $Z_\partial(U)$ be  the of those boundary paths $x\in\partial\Lambda$ such that $x(i,0, n)\in U,$ for some $(i,n)\in \ZZ_2^k\rtimes_d\NN^k$ with $(i,n)\leq  \dd(x)$.
The collection of sets of the form $Z_\partial(U)\backslash Z_\partial(F)$, with $U\subseteq \Lambda$ relatively compact and open, and $F\subseteq \bar U\Lambda$ compact, 
form a basis for a locally compact Hausdorff topology on $\partial\Lambda$.

By an argument as in \cite[Proposition 5.16]{clsv} (cf., \cite[Proposition 4.3]{y}), one can observe that  $\partial\Lambda\times\ZZ_2^k$ is  the smallest closed and invariant subset $Y$ of $G_\Lambda^{(0)}$ such that $vY$ is nonempty for all $v\in\Lambda^0$. 

\vspace{.3cm}
\indent {\bf Lemma A.8.}
	Let $\Lambda$ be a compactly aligned multilayered topological $k$-graph and $X=X_\Lambda$ be the corresponding product system. Then there is a unique gauge-compatible, Nica-covariant, Toeplitz representation $\psi \colon X \to C^{*} ( G_{\Lambda} )$
	such that, for each $i\in\ZZ_2^k$, $n \in \NN^{k}$, and $f \in C_{c} (\Lambda_i^{n})$,  $\psi (f) \in C_{c} (G_{\Lambda})$,
	and,
	\begin{equation*}\label{TR}
		\psi (f) ((x,k,p,y))
		=
		\begin{cases}
			f(x(i,0,n)) &\text{if $k=i, p = n$ and $\sigma_i^{n} (x) = y,$}\\
			0 &\text{otherwise.}
		\end{cases}
	\end{equation*}
	Moreover, $C^{*} ( G_{\Lambda} )$ is generated by the set of  all such functions $\psi (f)$.
\begin{proof}
	The fact that $\psi (f) \in C_{c} (G_{\Lambda})$ and that $\psi$ is linear and 
	multiplicative on $C_{c}
	(\Lambda_i^{0})$ is proved by an argument similar to that of \cite[Proposition 5.17]{clsv}. Thus $\psi$ extends to a linear map from
	$X_{(i,n)}$ to $C^{*} (G_{\Lambda})$, whose restriction to $A_i$ is a homomorphism. For $f \in C_{c}
	(\Lambda_i^{m})$ and $g \in C_{c} (\Lambda_{j}^{n})$, if $j=i+d(m)$, 
	\begin{align*}
		\psi (f) \ast \psi (g) &((x,k,p,y))
		= \textstyle{\sum_{(x,k,r,w) \in G_{\Lambda}} \psi (f) ((x,k,r,w)) \psi (g) ((w,k+d(r),p-r,y))}
		\\&=
		\begin{cases}
			f(x(i,0,m)) g(x(i+d(m),m,m+n)) &\text{if $p=m+n$ and $\sigma_i^{m+n} (x) = y,$}\\
			0 &\text{otherwise}
		\end{cases}
		\\&= \psi (fg) ((x,k, p,y)).
	\end{align*}
and if $j\neq i+d(m)$, all the terms in the sum of the right hand side of the first equality above vanish, thus 
$\psi (f) \ast \psi (g)=0$. 	Finally, for $f,g \in C_{c} (\Lambda_i^{n})$, $\langle f,g \rangle_{i}^{n} \in C_{c}
	(\Lambda_i^{0})$, and we have,
	\begin{align*}
		\psi (f)^{*}& \ast \psi (g) ((x,k,p,y))
		\\&=
		\begin{cases}
			\textstyle{\sum_{\alpha \in \Lambda_i^{n} r(x)}} \overline{\psi (f)} ((\alpha x,i,n,x)) \psi (g) ((\alpha x,i+d(n),n,x)) &\text{if $k=i, p=0,$  $x=y,$}\\
			0 &\text{otherwise}
		\end{cases}
		\\&= \psi (\langle f,g \rangle_{i}^{n}) ((x,k,p,y)).
	\end{align*}
	To see that the set of such functions
	generates $C^{*} (G_{\Lambda})$, we just need to observe  that they generate a dense subalgebra of $C_{0} (G_{\Lambda})$ in the uniform norm, which follows by   Stone-Weierstrass theorem as in the proof of \cite[Proposition 5.17]{clsv}. Finally, uniqueness follows from the facts that $C_{c} (\Lambda_i^{n})$ is dense in $X_{(i,n)}$.
	
	Next, by an argument as in \cite[Lemma 5.18]{clsv}, 
	for $i\in \ZZ_2^k$, $m \in \NN^{k}$, and $f,g \in C_{c} (\Lambda_i^{m})$, 
		\begin{align*}
		\psi^{(m)}& (f \otimes g^{*}) ((x,k,p,y))
		\\&=
		\begin{cases}
			f(x(i,0,m)) \overline{g(y(i,0,m))} &\text{if $k=i, p=0$, $(i,m) \le \dd(x)$, $\sigma_i^{m} (x) = \sigma_i^{m} (y),$}\\
			0 &\text{otherwise}
		\end{cases}
	\end{align*}	
The canonical continuous cocycle $c:(x,i,m,y) \mapsto (i,m)$ on
	$G_\Lambda$ induces a coaction $\beta$ of $\ZZ_2^k\rtimes_d\ZZ^k$ on
	$C^*(G_\Lambda)$, satisfying $\beta(f) = f \otimes
	i_{\ZZ_2^k\rtimes_d\ZZ^k}(i,m)$,  whenever $\supp(f) \subseteq c^{-1}(i,m)$. Now since, $\beta(\psi(f)) = \psi(f) \otimes i_{\ZZ_2^k\rtimes_d\ZZ^k}(m)$, for
	$f \in X_{(i,m)}$, $\psi$ is gauge-compatible.
	To show the Nica covariance,  let $i\in\ZZ_2^k$, $m,n \in \NN^{k}$, $f_{m},g_{m} \in F_i^{m}$ and $f_{n},g_{n}
	\in F_{i+d(m)}^{n}$, and  write, as in the previous lemma,
	\[
	\iota_{(i,m)}^{(i,m \vee n)} (g_{m} \otimes f_{m}^{*}) \iota_{(i+d(m),n)}^{(i, m \vee n)} (f_{n} \otimes g_{n}^{*})
	= \textstyle{\sum_{k,j,i}} a_{j,i} \otimes b_{j}^{*}.
	\]
	It suffices to
	show that,
	\[
	\psi^{(m)} (g_{m} \otimes f_{m}^{*}) \psi^{(n)} (f_{n} \otimes g_{n}^{*})
	= \psi^{(m \vee n)} (\iota_{(i,m)}^{(i,m \vee n)} (g_{m} \otimes f_{m}^{*}) \iota_{(i+d(m),n)}^{(i, m \vee n)} (f_{n} \otimes g_{n}^{*})).
	\]
	Now, both sides  are equal to
	zero unless $p = 0$, so it is left to show that,
		\begin{equation*}\label{NicaCov2}
		\begin{split}
			\textstyle{\sum_{(x,i,0,z) \in G_{\Lambda}}} \psi^{(m)} (g_{m} \otimes f_{m}^{*}) &((x,i,0,z)) \psi^{(n)} (f_{n} \otimes g_{n}^{*}) ((z,i,0,y)) \\
			&= \psi^{(m \vee n)} (\iota_{(i,m)}^{(i,m \vee n)} (g_{m} \otimes f_{m}^{*}) \iota_{(i+d(m),n)}^{(i, m \vee n)} (f_{n} \otimes g_{n}^{*})) ((x,i,0,y)).
	\end{split}\end{equation*}
	We may suppose that $(i,m\vee n) \le \dd(x)$ and
	$\sigma_i^{m \vee n}(x) = \sigma_i^{m \vee n}(y)$, since otherwise both sides are again zero. Let $\lambda_{(i,m)} := \lambda_{f_m,
		i,x(m)}$, and observe  that,
	 if $y(i+d(m), n, m \vee n) \not= [\lambda_{(i,m)} \sigma_i^{m} (x)](i+d(m), n,m \vee n)$, both sides are  zero, and otherwise, 
	\begin{align*}
		\psi^{(m \vee n)}&\big(\iota_{(i,m)}^{(i, m \vee n)}(g_{m} \otimes f_{m}^{*}) \iota_{(i+d(m), n)}^{(i, m \vee n)} (f_{n} \otimes g_{n}^{*})\big) ((x,i,0,y))\\
		&= {\textstyle \sum_{j,k}}  \psi^{(m \vee n)} (a_{k,j,i} \otimes b_{j,i}^{*}) ((x,i,0,y)) \\
		&= g_{m}(x(i,0,m)) \bar f_{m}(\lambda_{(i,m)}) f_{n}([\lambda_{(i,m)} \sigma_i^{m} (x)](i,0,n)) \bar g_{n}(y(i,0,n))
		\\&=\textstyle{\sum_{(x,i,0,z) \in G_{\Lambda}}} \psi^{(m)} (g_{m} \otimes f_{m}^{*}) ((x,i,0,z)) \psi^{(n)} (f_{n} \otimes g_{n}^{*}) ((z,i,0,y)), 
	\end{align*}
	as required.
\end{proof}

For the canonical  cocycle map 
$$c: G_\Lambda \to \ZZ_2\rtimes_d\ZZ^k;\ \ (x,i,m,y) \mapsto (i,m),$$ let $\Gg_\Lambda[c]:=\ker(c)\cap\Gg_\Lambda$. The next result is an adaptation of \cite[Proposition 4.2]{rswy}

\vspace{.3cm}
\indent {\bf Lemma A.9.}
	Let $(\Lambda,\dd,d)$ be a compactly aligned multilayered topological $k$-graph. Then the kernel $\Gg_\Lambda[c]$  is an amenable, principal, \'{e}tale groupoid with the same unit space as $\Gg_\Lambda$.
\begin{proof}
	For $(i,m) \in \ZZ_2^k\rtimes_d\NN^k$, and the sets
\begin{align*}
S_{(i,m)}:&=\{(x,i,0,y) : \dd(x) = \dd(y) \ge (i,m), \sigma_i^m(x) = \sigma_i^m(y)\}\\&=\{(\alpha z,i,0,\beta z) : z \in \partial\Lambda, \alpha, \beta \in \Lambda_i^m r(z)\},
\end{align*}
$
		R_{(i,m)} := \{(x,i,0,x) : x \in \partial\Lambda\} \cup S_{(i,m)}
$ 
is a subgroupoid of $\Gg_\Lambda[c]$, and an $\mathcal F_\sigma$ subset in $\partial\Lambda \times\{i\}\times
	\partial\Lambda$. Identifying the latter set with $\partial\Lambda \times
	\partial\Lambda$, we may regard $R_{(i,m)}$ as an equivalence relation on $\partial\Lambda$.  We claim that each $R_{(i,m)}$ is proper as a
	Borel groupoid \cite[Definition 2.1.2]{dr}, that is, 
	the orbits of points of $\partial\Lambda$ in $R_{(i,m)}$ are locally closed \cite[Examples 2.1.4(2)]{dr},
	\cite[Theorem~2.1]{ram}. For 
	$$S_{(i,m)}(x):=\{(\alpha \sigma_i^m(x), i,0, \beta \sigma_i^m(x)) : \dd(\alpha) = \dd(\beta) = (i,m)\},$$ The orbit $[x]$ of $x$ in $R_{(i,m)}$ is equal to $\{x\}$ when $\dd(x) \not\ge (i,m)$, and, 	\begin{equation*}\label{eq:orbit description}
		[x] =
		R_{(i,m)}^{(0)} \cap Z_\partial(\Lambda_i^m \ast_s \Lambda_i^m)  \cap \overline{ S_{(i,m)}(x)},
	\end{equation*}
otherwise, by an argument verbatim to that in \cite[Proposition 4.2]{rswy}. 
	By an argument as in \cite[Theorem 3.16]{y},  $\Gg_\Lambda$ is \'etale, thus $R_{(i,m)}^{(0)} = \Gg_\Lambda^{(0)}$ is open, and  $[x]$ is the intersection of an open  set and a
	closed set in the second case, which  is locally closed, establishing the claim. In particular, the Borel groupoid $R_{(i,m)}$ is measurewise amenable. It then follows that $\Gg_\Lambda$ is measurewise amenable as a direct limit of $R_{(i,m)}$'s \cite[Proposition~5.3.37]{dr}. Now $\Gg_\Lambda[c]$ is  \'etale and second-countable, and so the orbits are countable in $\Gg_\Lambda[c]$. In particular, $\Gg_\Lambda[c]$ is also
	topologically amenable  \cite[Theorem~3.3.7]{dr}. The fact that $\Gg_\Lambda[c]$ is principal is an immediate consequence of the factorization property of $\Lambda$.
\end{proof}

\indent {\bf Theorem A.10.}	Let $(\Lambda, \dd, d)$ be a compactly aligned multilayered topological $k$-graph with divider $d: \mathbb (\ZZ^k)^+\to \ZZ_2^k$ and degree map $\dd: \Lambda\to \ZZ_2^k\rtimes_d\ZZ^k$.
	Let $G_\Lambda$ and $\Gg_\Lambda$ be the path groupoid and
	boundary-path groupoid of $\Lambda$, and let  $q \colon C^*(G_\Lambda) \to C^*(\Gg_\Lambda)$ be the canonical quotient map.
	Then the homomorphism
 $\psi_* \colon \Tc(X_\Lambda) \to C^*(G_\Lambda)$ 	coming from the universal property of $\Tc(X_\Lambda)$, for the corresponding product system $X_\Lambda$ over $(G,P):=(\ZZ_2^k\rtimes_d\ZZ^k, \ZZ_2^k\rtimes_d\NN^k)$, is an isomorphism. Moreover, both of the regular representations
	$\lambda_\Lambda: C^*(G_\Lambda)\to C^*_\reduced(G_\Lambda)$ and
	$\lambda_\Lambda^{'}: C^*(\Gg_\Lambda)\to C^*_\reduced(\Gg_\Lambda)$ are isomorphisms,
	and there is a unique $*$-isomorphism
	$\phi \colon C^*(\Gg_\Lambda) \to \NO{X}$,
	making the following diagram commutative,
	\[
	\begin{CD}
		\Tc(X)@>{\psi_*}>> C^*(G_\Lambda) \\
		@V{q_{\CNP}}VV @V{q}VV \\
		\NO{X} @<{\phi}<< C^*(\Gg_\Lambda).
	\end{CD}
	\]
\begin{proof}
	Since $\psi(X_\Lambda)$ generates
	$C^*(G_\Lambda)$,  $\psi_*$ is surjective.
	Let $\lambda_\Lambda: C^*(G_\Lambda)\to
	C^*_\reduced(G_\Lambda)$ be the regular representation, and put
	$\tilde{\psi}_*:=\lambda_\Lambda\circ\psi_*$. We claim  that
	$\tilde{\psi}_*$ it injective. 
	The canonical  cocycle map 
	$c: G_\Lambda \to \ZZ_2\rtimes_d\ZZ^k;\ \ (x,i,m,y) \mapsto (i,m)$ induces a coaction
	$\beta$ of $G:=\ZZ_2^k\rtimes_d\ZZ^k$ on
	$C^*_\reduced(G_\Lambda)$, defined at $f\in C_c(G_\Lambda)$, with $\supp(f) \subseteq c^{-1}(i,m)$, by $\beta(f) = f \otimes
	i_{G}(i,m)$, and extended by continuity. Let $\delta$ be the canonical coaction of $G$ on $\Tc(X_\Lambda)$. Then 
	$\tilde{\psi}_*$ is equivariant for coactions $\delta$ and $\beta$.
	The groupoid $G=\ZZ_2^k\rtimes_d\ZZ^k$ is amenable by Lemma \ref{ex}, the coaction $\delta$ is normal by Lemma \ref{ap-am} and Corollary \ref{cor:guit}. Therefore, 
	it is enough to show that
	$\ker\tilde{\psi}_*\cap\Ff=\{0\}$. By the observations in the paragraph after Lemma \ref{core} and \cite[Lemma~1.3]{ALNR},
	it is enough to show that
	$\ker\tilde{\psi}_*\cap B_F=\{0\}$, for each $F \in \fvcl{P}$. Let $F \in \fvcl{P}$ and  given $p \in F$,
	take $T_p \in \mathscr{K}_{A_{r(p)}}(({X_\Lambda})_p)$  such that,
$
		\sum_{p \in F}
		\tilde{\psi}_*(i_{X_\Lambda}^{(p)}(T_p))=0.
$
	As in the proof of \cite[Theorem 5.20]{clsv}, we use induction on  $|F|$ to show
	that $T_p=0$. By the above lemma, the representation $\lambda_\Lambda\circ\psi$ is
	injective, and so is 
	$\tilde{\psi}_*\circ
	i_{X_\Lambda}^{(p)}=(\lambda_\Lambda\circ\psi)^{(p)}$, by   an argument as in the proof of \cite[Lemma 2.4]{ka1}. This proves the claim that $T_p=0$, when  $F=\{p\}$ is a singleton. Next, let  $|F|\geq 2$ and let $p_0$ be a
	minimal element of $F$, then $p\not\le p_0$, for  $p\in
	F\backslash\{p_0\}$. It follows  from the formula of the third math display in the proof of the above lemma that if
	$(x,i,0,y)\in G_\Lambda$ with $\dd(x)=\dd(y)=p_0$, then
	$\tilde{\psi}_* (i_{X_\Lambda}^{(p)}(T_p))((x,i,0,y))=0$, for each $p\in
	F\backslash\{p_0\}$. By assumption, $\sum_{p \in F}
	\tilde{\psi}_*(i_{X_\Lambda}^{(p)}(T_p))=0$, thus $\tilde{\psi}_* (i_{X_\Lambda}^{(p_0)}(T_{p_0}))((x,i,0,y))=0$
	for all $(x,i,0,y)\in G_\Lambda$ with $\dd(x)=\dd(y)=p_0$. Using the above mentioned math display once more, we get $\tilde{\psi}_*
	(i_{X_\Lambda}^{(p_0)}(T_{p_0}))((x,i,p,y))=0$, for each $(x,i,p,y)\in
	G_\Lambda$, and thus that $\tilde{\psi}_*
	(i_{X_\Lambda}^{(p_0)}(T_{p_0}))=0$. Arguing as above, this implies that
	$T_{p_0}=0$, and so, $\sum_{p \in F\backslash\{p_0\}}
	\tilde{\psi}_*(i_{X_\Lambda}^{(p)}(T_p))=0,$ which then by the inductive hypothesis, gives 
	 $T_p=0$, for each $p\in F\backslash\{p_0\}$, establishing the claim. It now follows that $\tilde{\psi}_*$ is
	injective, which in turn implies that both $\lambda_\Lambda$ and $\psi_*$ are isomorphisms.
	
	Let $\rho_\Lambda : X_\Lambda \to C^*(\Gg_\Lambda)$ be the canonical Toeplitz
	representation, and note that $\rho_\Lambda=\lambda_\Lambda\circ q \circ \psi$, where $q \colon C^*(G_\Lambda) \to C^*(\Gg_\Lambda)$ is the canonical quotient map. By the above lemma,  $\rho_\Lambda$ is a
	Nica covariant representation of $X_\Lambda$ whose range  generates
	$C^*_\reduced(\Gg_\Lambda)$. By the paragraph before Lemma A.8,
	implies that $(\rho_\Lambda)_{(i,0)} : C_0(\Lambda_i^0) \to C^*_\reduced(\Gg_\Lambda)$ is
	injective, for each $i\in \ZZ_2^k$, and  so is $\rho_\Lambda$, regarded as a representation of $X_\Lambda$.
	The above cocycle map $c$, restricted to 
	$\Gg_\Lambda$,  induces a coaction $\gamma$ of $G$
	on $C^*_\reduced(\Gg_\Lambda)$, such that
	$\gamma(\rho_\Lambda(f))=\rho_\Lambda(f)\otimes i_{G}(d_\Lambda(f))$, for $f\in X_\Lambda$, where $d_\Lambda: X_\Lambda\to P$ is the map characterized by $(X_\Lambda)_p=d_\Lambda^{-1}(p)$, for $p\in P$. The coaction  $\gamma$ is normal by Lemma \ref{ap-am} and Corollary \ref{cor:guit}. 
	The above equality for  $\rho_\Lambda$, simply says that $\rho_\Lambda$ is gauge-compatible. Since $P=\ZZ_2^k\rtimes_d\ZZ^k$ is is directed by Lemma \ref{ex}, the product system $X_\Lambda$ over $P$ is
	automatically $\tilde\phi$-injective, by an argument as in \cite[Lemma 4.3]{SY}. Now by
	Theorem~\ref{thm:projective property}, there exists a
	surjective $*$-homomorphism $\phi':C^*_\reduced(\Gg_\Lambda)\to
	\NO{X}$ such that $\phi'\circ \lambda_\Lambda \circ q \circ
	\psi_*=q_{\CNP}$. Let  us denote the coaction of $G$ on $\NO{X_\Lambda}$, induced by the coaction $\delta$ of $G$ on $\Tc(X_\Lambda)$, again by $\delta$. Then we have, $\delta\circ\phi^{'}=(\phi^{'}\circ{\rm id}_{C^*(G)})\circ\gamma$.  Put $\phi:=\phi'\circ \lambda_\Lambda$. 
	
	We
	claim that $\phi$ is injective.   
	Put
	$\Gg_\Lambda[c]:=c^{-1}(\ZZ_2^{k}\times\{0\})\cap \Gg_\Lambda$. This is a subgroupoid of $\Gg_\Lambda$ and so there is a conditional expectation $\Phi^\gamma: C^*_r(\Gg_\Lambda)\to C^*_r(\Gg_\Lambda[c])$. Now we just need to observe that $\ker(\phi^{'})\cap C^*_r(\Gg_\Lambda[c])=0$. The unit space $\Gg_\Lambda[c]^{(0)}$ of the \'{e}tale groupoid $\Gg_\Lambda[c]$ consists of elements of the form $u=(x,i,0,x)$ with $i\in \ZZ_2^k$ and $x\in\partial\Lambda$, and each isotropy group $(\Gg_\Lambda[c])_u^u$, for $u$ as above, is a singleton, consisting of $u$ itself, thus it follows from \cite[Lemma 2.15]{T}, that $\ker(\phi^{'})\cap C^*_r(\Gg_\Lambda[c])=0$ follows once we show that $\ker(\phi^{'})\cap C_0(\Gg_\Lambda^{(0)})=\ker(\phi^{'})\cap C_0(\Gg_\Lambda[c]^{(0)})=0$. Let $Y$ be the set of those elements $(x,i,0,x)\in G_\Lambda^{(0)}$ with $f((x,i, 0, x)) = 0$,  for each $f\in \ker(q_{CNP}\circ \tilde\psi_{*}^{-1})\cap C_0((G_\Lambda)^{(0)})$. By an argument as in the proof of \cite[Theorem 5.20]{clsv}, one can observe that $Y$ is $G_\Lambda$-invariant. Next, since $q_{\rm CNP}\circ i_{X_\Lambda}=j_{X_\Lambda}$ is injective as a representation of $X_\Lambda$, thus $q_{\rm CNP}\circ \tilde\psi_{*}^{-1}$ is injective on $\lambda_\Lambda(\psi_{*}\circ i_{X_\Lambda}(C_0(\Lambda^0)))$, that is, $vY$ is non-empty, for $v\in\Lambda^0$, and so by the characteristic property of the boundary, $\partial\Lambda\subseteq Y$. Now, given $f\in \ker(\phi^{'})\cap C_0(\Gg_\Lambda^{(0)})$, there is $\tilde f\in \ker(q_{\rm CNP}\circ\psi_{*}^{-1})\cap C_0(G\Lambda^{(0)})$ with $\lambda_\Lambda\circ q\circ\lambda_\Lambda^{-1}(\tilde f)=f$, and since $\tilde f=0$ on $\partial\Lambda\times\ZZ_2^k$, $q\circ\lambda_\Lambda^{-1}(\tilde f)=0$, thus $f=0$, as claimed.
	
	Finally, let us show that $\lambda_\Lambda$ is injective. Since $\Gg_\Lambda = G_\Lambda|\partial\Lambda\times\ZZ_2^k$, each element in  $\ker(q\circ\lambda_\Lambda^{-1})$ could be approximated by some $f\in C_c(G_\Lambda)$, such that  for each $i\in \ZZ_2^k$ and $m\in \ZZ^k$, $f((x,i, m, y)) = 0$, whenever $x, y\in \partial\Lambda$. 
	We then could use an approximate identity for $\ker(q\circ\lambda_\Lambda^{-1})$ in $I:=C_0\big(G_\Lambda^{(0)}\backslash(\partial\Lambda\times\ZZ_2^k)\big)$ to observe that $I$ is  generated by its intersection with $C_0(G_\Lambda^{(0)})$, and thus is generated by its intersection with
	$C^*_r(G_\Lambda[c])$, for $G_\Lambda[c]:=c^{-1}(\ZZ_2^k\times\{0\})$, that is, $I$ is an induced ideal in the sense of previous section. Thus, by Lemmas \ref{prp:deltaI existence} and \ref{prop:exact}, the coaction $\beta$ induces a normal coaction $\beta^I$ of $G$ on $C^*(G_\Gamma)$ such that $q\circ\lambda_\Lambda$ is equivariant for $\beta$  and $\beta^I$, and $\lambda_\Lambda$ is equivariant
	for $\beta^I$ and $\gamma$. In particular, it is enough to check that $\lambda_\Lambda$ injective on the fixed-point algebra
	$C^*(\Gg_\Lambda[c])$ for the coaction $\beta^I$. Now  the restriction of $\lambda_\Lambda$ to $C^*(\Gg_\Lambda[c])$ is nothing but the regular representation from $C^*(\Gg_\Lambda[c])$ onto $C^*_r (\Gg_\Lambda[c])$, which is an isomorphism, as $\Gg_\Lambda[c]$ is amenable by the previous lemma.
\end{proof}

We immediately get the
following gauge-invariant uniqueness result for $C^*(\Gg_\Lambda)$.

\vspace{.3cm}
\indent {\bf Corollary A.11.}
	Let $\Lambda$ be a compactly aligned multilayered topological $k$-graph.
	Let $\Gg_\Lambda$ be the analog of the Yeend's
	boundary-path groupoid for $\Lambda$. For the unique gauge-compatible, Nica-covariant, Toeplitz representation $\rho_\Lambda \colon X_\Lambda \to C^{*} (\Gg_{\Lambda} )$, as in the previous theorem, a surjective $*$-homomorphism $\phi:C^*(\Gg_\Lambda)\to D$ is injective if and only if,
		
		(1) there is a coaction $\beta$ of $\ZZ_2^k\rtimes_d\ZZ^k$ on $D$ such that,
		$$\beta(\phi(\rho_\Lambda(f)) = \phi(\rho_\Lambda(f))\otimes i_{\ZZ_2^k\rtimes_d\ZZ^k}((i,m)),$$ for each $i\in\ZZ_2^k$, $m\in \ZZ^k$, and $f \in
		C_c(\Lambda_i^m)$, 
		
		(2) $\phi\vert_{\rho_\Lambda(X_{(i,0)})} \colon C_0(\Lambda_i^0) \to D$ is  injective.

\section*{Acknowledgment}

The authors would like to thank Professor Aidan Sims for reading the first draft of this paper posted on arXiv and suggesting some corrections. In particular, we have added condition (3) to the definition of the Toeplitz representation in Section \ref{subsection:reps_prod_syst}, and the comments on the kernel of  canonical coaction of the Toeplitz algebra (see the paragraph before Lemma \ref{core}).


\begin{thebibliography}{00}
\bibitem{ALNR} S. Adji, M. Laca, M. Nilsen, and I. Raeburn,
    \emph{Crossed products by semigroups of endomorphisms and the
    Toeplitz algebras of ordered groups}, Proc.  Amer.  Math.  Soc.
    {\bf 122} (1994), 1133--1141.
    
\bibitem{dr}     C. Delaroche and J. Renault, \emph{Amenable groupoids}, Monographies de L’Enseignement
    Math\'{e}matique {\bf 36}, Geneva, 2000.

\bibitem{Arveson1989} W. Arveson, \emph{Continuous analogues of
    {F}ock space}, Mem. Amer. Math. Soc. \textbf{80} (1989),
    No.~409, iv+66 pp.

\bibitem{Arveson2008} W. Arveson, \emph{The noncommutative {C}hoquet boundary},
  J. Amer. Math. Soc. \textbf{21} (2008), No.~4, 1065--1084.

\bibitem{Black} B. Blackadar, \emph{Operator algebras. Theory
    of
    $C\sp *$-algebras and von Neumann algebras}. Encyclop.
    Math. Sci.  {\bf 122}, Operator Algebras and Non-commutative
    Geometry III, Springer-Verlag, Berlin, 2006.
    
    \bibitem{bo}
    N. P. Brown and N. Ozawa, \emph{C$^*$-algebras and finite-dimensional approximations}, Grad. Stud. Math. {\bf 88}, Amer. Math. Soc., Providence, 2008.
    
    
\bibitem{clsv} T.M. Carlsen, N.S. Larsen, A. Sims and S.T. Vittadello, \emph{Co-universal algebras associated to product systems,
    and gauge-invariant uniquenss theorems}, Proc. London Math. Soc. (3) {\bf 103} (2011), 563–600.


\bibitem{CL2002} J. Crisp and M. Laca, \emph{On the Toeplitz
    algebras of right-angled and finite-type Artin groups},
 J. Austral. Math. Soc. {\bf 72} (2002), 223--245.

\bibitem{CL} J. Crisp and M. Laca, \emph{Boundary quotients and
    ideals of Toeplitz $C\sp*$-algebras of Artin groups}, J.
    Funct. Anal. \textbf{242} (2007), 127--156.

\bibitem{Cuntz1977} J. Cuntz, \emph{Simple {$C\sp*$}-algebras
    generated by isometries}, Comm. Math. Phys. \textbf{57}
    (1977), No.~2, 173--185.

\bibitem{CK1980} J. Cuntz  and W. Krieger, \emph{A class of
    {$C\sp{\ast} $}-algebras and topological {M}arkov chains},
    Invent. Math. \textbf{56} (1980), No.~3, 251--268.
    
\bibitem{dkr}    V. Deaconu, A. Kumjian and B. Ramazan, \emph{Fell bundles
    associated to groupoid morphisms}, Math. Scand. {\bf 102} (2008),
    No 2, 305--319.

\bibitem{Dinh1991} H.T. Dinh, \emph{Discrete product systems
    and their {$C\sp *$}-algebras}, J. Funct. Anal.
    \textbf{102} (1991), No.~1, 1--34.

\bibitem{Dix:C*-algebras} J. Dixmier, \emph{$C\sp*$-algebras,}
     North-Holland
    Math. Library {\bf 15}, North-Holland, 
    Amsterdam, 1977.
    
    

\bibitem{EKQR} S. Echterhoff, S. Kaliszewski, J. Quigg and I.
    Raeburn, \emph{A categorical approach to imprimitivity theorems
    for $C\sp *$-dynamical systems},  Mem. Amer. Math. Soc. {\bf180}
    (2006), viii+169 pp.

\bibitem{EKQ} S. Echterhoff, S. Kaliszewski and J.  Quigg, \emph{Maximal coactions},
Internat. J. Math. {\bf 15} (2004), 47--61.

\bibitem{EQ}
  S. Echterhoff and J. Quigg, \emph{Induced coactions of discrete groups
  on {$C\sp *$}-algebras}, Canad. J. Math. \textbf{51} (1999), no.~4, 745--770.

\bibitem{Exel} R. Exel, \emph{Amenability for {F}ell bundles},
  J. Reine angew. Math. \textbf{492} (1997), 41--73.


\bibitem{ELQ} R. Exel, M. Laca and J. Quigg, \emph{Partial
    dynamical systems and $C^*$-algebras generated by partial
    isometries}, J. Operator Theory {\bf47} (2002), 169--186.

\bibitem{Fell} J.M.G. Fell and R.S. Doran,
    \emph{Representations of {$\sp *$}-algebras, locally
    compact groups, and {B}anach {$\sp *$}-algebraic bundles}, 
    {V}ol. 2, Pure and Applied Mathematics   {\bf 126}, Academic
    Press, Boston, 1988.

\bibitem{Fowler1999} N.J. Fowler, \emph{Compactly-aligned
    discrete product systems, and generalization of {$\mathcal{O}_\infty$}}, Internat. J. Math. \textbf{10} (1999), No.~6,
    721--738.

\bibitem{F99} N.J. Fowler, \emph{Discrete product systems of
    Hilbert bimodules}, Pacific J. Math.  {\bf 204} (2002), 335--375.


\bibitem{FMR} N.J. Fowler, P.S. Muhly and I. Raeburn, \emph{Representations of Cuntz-Pimsner algebras}, Indiana
Univ. Math. J. \textbf{52} (2003), 569--605.

\bibitem{FR} N.J. Fowler and I. Raeburn, \emph{The Toeplitz algebra of a Hilbert bimodule}, Indiana Univ. Math. J. {\bf 48} (1999), No. 1, 155--181.

\bibitem{gre:am78} P. Green, \emph{The local structure of twisted covariance algebras},
	Acta Math. {\bf 140}(1978),
	191--250.

\bibitem{HR1997} A. an Huef and I. Raeburn,
    \emph{The ideal structure of Cuntz-Krieger algebras},
     Ergod. Theo. Dynam. Sys. \textbf{17} (1997), 611--624.
     
     
\bibitem{ion}         M. Ionescu, A. Kumjian, J. N. Renault, A. Sims, D. P. Williams, \emph{$C^*$-algebras of extensions of groupoids by group bundles}, J. Funct. Anal. {\bf 280} (2021), No. 5, 108892.

\bibitem{ka1} T. Katsura, \emph{A class of $C^{*}$-algebras
    generalizing both graph algebras and homeomorphism
    $C^{*}$-algebras I, fundamental results}, Trans. Amer.
    Math. Soc. \textbf{356} (2004), 4287--4322.


\bibitem{ka2} T. Katsura, \emph{On $C^*$-algebras associated with $C^*$-correspondences},
J. Funct. Anal. \textbf{217}, 366-401.

\bibitem{ka3} T. Katsura, \emph{Ideal structure of $C^*$-algebras associated with $C^*$-correspondences},
Pacific J. Math. \textbf{230} (2007), 107--145.

\bibitem{kum} A. Kumjian, \emph{Fell bundles over groupoids}, Proc. Amer. Math. Soc. {\bf 126} (1998), No. 4, 1115-1125. 

\bibitem{L} M. Laca, \emph{Purely infinite simple
    {T}oeplitz algebras}, J. Operator Theory \textbf{41}
    (1999), No.~2, 421--435.


\bibitem{Lan} E.C. Lance, \emph{Hilbert $C^*$-modules: A
    toolkit for operator algebraists}, London Math. Soc. Lecture Note Series {\bf 210}, Cambridge Univ. Press, Cambridge, 1994.
    
    
    
\bibitem{mw}    P. S. Muhly and D. P. Williams, \emph{Renault’s equivalence theorem for
    groupoid crossed products}, NYJM Monographs {\bf 3}, 2008. 

\bibitem{mw2}    P. S. Muhly and D. P. Williams, \emph{Equivalence and disintegration theorems for Fell bundles
and their $C^*$-algebras}, Dissert. Math. {\bf 456} (2008), 1–57.


\bibitem{N} A. Nica, \emph{$C^*$-algebras generated by isometries and
Wiener-Hopf operators}, J. Operator Theory \textbf{27} (1992), 17--52.


\bibitem{o} K.J. Oty, \emph{Fourier-Stieltjes algebras of r-discrete groupoids}, J.  Operator Theory {\bf 41} (1999), 175--197. 


\bibitem{or} K.J. Oty and A. Ramsay, \emph{Actions and coactions of measured groupoids on $W^*$-algebras}, J.  Operator Theory {\bf 561} (2006), No. 1,  199--217. 

\bibitem{pat} A. Paterson, \emph{The Fourier algebra for locally compact groupoids}, Canad. J. Math. {\bf 56} (2004), N. 6, 1259--1289. 

\bibitem{Pimsner1997} M.V. Pimsner, A class of \emph{$C\sp
    *$-algebras generalizing both {C}untz-{K}rieger algebras
    and crossed products by} $\mathbb Z$, Free probability
    theory (Waterloo, ON, 1995), Amer. Math. Soc., Providence, 1997, pp. 189--212.

\bibitem{qui:full duality} J. Quigg, \emph{Full $C^*$-crossed
    product duality}, J. Austral. Math. Soc. \textbf{50}
    (1991), 34--52.

\bibitem{qui:discrete coactions}
    J. Quigg, \emph{Discrete coactions and $C^*$-algebraic
    bundles}, J. Austral. Math. Soc. \textbf{60} (1996),
    204--221.

\bibitem{qui:full reduced} J. Quigg, \emph{Full and reduced
    $C^*$-coactions}, Math. Proc. Camb. Phil. Soc. {\bf 116}
    (1994), 435--450.

\bibitem{QuiggRa} J. Quigg and I. Raeburn, \emph{Characterisations of crossed products by partial actions}, J. Operator
Theory {\bf 37} (1997), 311--340.

\bibitem{RS} I. Raeburn and A. Sims, \emph{Product systems of graphs and the Toeplitz
  algebras of higher-rank graphs}, J. Operator Theory \textbf{53} (2005),
  No.~2, 399--429.

\bibitem{TFB} I. Raeburn and D.P. Williams, \emph{Morita
    equivalence
    and continuous-trace $C^*$-algebras}, Math. Surv.
    Monogr. {\bf 60}, Amer. Math. Soc., Providence, 1998.
    
 \bibitem{ram}   A. Ramsay, The Mackey-Glimm dichotomy for foliateions and other Polish groupoids, J. Funct. Anal.
    {\bf 94} (1990), No. 2, 358--374.

\bibitem{r} J. Renault, \emph{A Groupoid approach to $C^{*}$-algebras}, Lecture Notes in Math. {\bf 793}, Springer-Verlag, Berlin, 1980.

\bibitem{r2} J. Renault, \emph{Repr\'{e}sentations des produits crois\'{e}s d’alg\`{e}bres de groupoides}, J.
Operator Theory {\bf 18} (1987), 67--97.

\bibitem{rswy} J. Renault, A. Sims, D.P. Williams and T. Yeend, \emph{Uniqueness theorems for topological higher-rank graph $C^*$-algebras}, Proc. Amer. Math. Soc.  {\bf 16}(2) (2018), 669--684.

\bibitem{SY} A. Sims and T. Yeend, \emph{Cuntz-Nica-Pimsner
    algebras
    associated to product systems of Hilbert bimodules}, J. Operator
    Theory {\bf 64} (2010), No. 2, 349-376.
    
\bibitem{T}    K. Thomsen, \emph{Semi \'{e}tale groupoids and applications},  arXiv: 0901.2221.

\bibitem{tim} T. Timmermann, \emph{The relative tensor product and a minimal fiber product in the setting of $C^{*}$-algebras}, J. Operator Theory {\bf 68} (2012), No. 2, 365--404.

\bibitem{weg} N. E. Wegge-Olsen, \emph{K-theory and $C^*$-algebras: A friendly approach}, Oxford University Press, Oxford, 1993.

    
\bibitem{w}    D. P. Williams, \emph{A tool kit for groupoid $C^*$-algebras, } Math. Surv. Monographs {\bf 241},    
    Amer. Math. Soc., Providence, 2019. 
    
\bibitem{yam}    S. Yamagami, \emph{On primitive ideal spaces of $C^*$-algebras over certain locally compact
    groupoids}, in: Mappings of Operator Algebras (H. Araki and R. Kadison, eds.), Progress in
    Math. {\bf 84}, Birkha\"user, Boston, 1991, pp. 199-204. 

\bibitem{yama} T. Yamanouchi, \emph{Duality for actions and coactions of measured groupoids on von
Neumann algebras}, Mem. Amer. Math. Soc. {\bf 484}, Amer. Math. Soc., Providence, 1993.
    

\bibitem{y} T. Yeend, \emph{Groupoid models for the $C^{*}$-algebras of topological higher-rank graphs}, J. Operator Theory \textbf{57} (2007), 95--120.

\end{thebibliography}
\end{document}